\documentclass[11pt,a4paper,reqno]{amsart}
\usepackage[utf8]{inputenc}
\usepackage{amsmath,amssymb,amsthm, mathtools}
\usepackage{enumerate}
\usepackage{xfrac}
\usepackage[symbol,perpage]{footmisc}
\usepackage{hyperref}
\usepackage{enumitem}
\usepackage[british]{babel}
\usepackage{graphicx, subcaption, xcolor}
\usepackage[margin=1.4in]{geometry}

\theoremstyle{plain}
\newtheorem{thm}{Theorem}[section]
\newtheorem{prop}[thm]{Proposition}
\newtheorem{lemma}[thm]{Lemma}
\newtheorem{cor}[thm]{Corollary}
\newtheorem{summ}[thm]{Summary}

\theoremstyle{definition}
\newtheorem{defi}[thm]{Definition}

\newtheorem{obs}[thm]{Remark}
\newtheorem{ex}[thm]{Example}

\usepackage{chngcntr}
\counterwithin{figure}{section}

\theoremstyle{plain}
\newtheorem*{namedthm}{\namedthmname}
\newcounter{namedthm}
\makeatletter
\newenvironment{named}[1]
{\def\namedthmname{#1}%
	\refstepcounter{namedthm}%
	\namedthm\def\@currentlabel{#1}}
{\endnamedthm}
\makeatother

\def\Cx{\mathbb{C}}
\def\Chat{\widehat{\mathbb{C}}}

\def\DD{\mathbb{D}}

\def\dist{\mathrm{dist}}

\def\A{\mathcal{A}}

\def\dD{\partial\mathbb{D}}

\renewcommand{\tilde}{\widetilde}
\renewcommand{\setminus}{\smallsetminus}

\title{Boundary behaviour of universal covering maps}

\author[G. R. Ferreira]{Gustavo R. Ferreira}
\address{Centre de Recerca Matem\`atica, Barcelona, Spain}
\email{grodrigues@crm.cat}

\author[A. Jov\'e]{Anna Jov\'e}
\address{Departament de Matemàtiques i Informàtica, Universitat de Barcelona, Barcelona, Spain}
\email{ajovecam7@alumnes.ub.edu}
\thanks{The first author acknowledges financial support from the Spanish State Research Agency through the Mar\'ia de Maeztu Program for Centers and Units of Excellence in R\&D (CEX2020-001084-M). The second author acknowledges financial support from the Spanish government grant FPI PRE2021-097372.}

\date{\today}

\begin{document}
\begin{abstract}
	Let $ \Omega \subset\Chat$ be a multiply connected domain, and let $ \pi\colon \mathbb{D}\to\Omega $ be a universal covering map. In this paper, we analyze the boundary behaviour of $ \pi $, describing the interplay between radial limits and angular cluster sets,   the tangential and non-tangential limit sets of the deck tranformation group, and the geometry and the topology of the boundary of $ \Omega $.
	
	As an application, we describe accesses to the boundary of $ \Omega $ in terms of radial limits of points in the unit circle, establishing a correspondence, in the same spirit as in the simply connected case. We also develop a theory of prime ends for multiply connected domains which behaves properly under the universal covering, providing an extension of the Carath\'eodory--Torhorst Theorem to multiply connected domains.
\end{abstract}
\maketitle

\section{Introduction}

Let $ \Omega\subset\Chat $ be a simply connected domain. Then, the existence of a conformal map $ \varphi\colon\mathbb{D}\to\Omega $, the so-called {\em Riemann map}, allows us to describe the boundary $ \partial\Omega $, establishing beautiful connections between geometry, analysis and topology. Indeed,  the celebrated Carathéodory-Torhorst Theorem asserts that the Riemann map $ \varphi $ extends to the unit circle continuously if and only if $ \partial\Omega $ is locally connected. Moreover, Carathéodory provided a topological compactification of $ \Omega $ to deal with the non-locally connected situation, known as the theory of \textit{prime ends}. Prime ends, together with accessible boundary points, are strongly related with the theory of radial limits and cluster sets, which in their turn describe analytically the topology of $ \partial\Omega $ -- see Section \ref{section-prelimnary-universal-covering} for a self-contained overview.

However, if we let $ \Omega\subset\Chat $ be a multiply connected domain and consider the universal covering $ \pi\colon\mathbb{D}\to\Omega $, the previous results break down -- indeed, $ \pi\colon\mathbb{D}\to\Omega $ has infinite degree, so it is no longer a homeomorphism. In particular, having $ \left\lbrace z_n\right\rbrace _n\subset\mathbb{D} $ with $ z_n\to\partial \mathbb{D} $ does not necessarily imply $ \pi(z_n)\to\partial\Omega $, meaning that the tools developed by Carathéodory, Torhorst, and many others need to be adapted.

On the one hand, when dealing with multiply connected domains, the group of deck transformations $ \Gamma $ of $ \pi  $ (i.e. automorphisms $ \gamma $ of $ \mathbb{D} $ such that $ \pi  =\gamma\circ\pi$) becomes especially relevant, since it  uniquely determines the planar domain we are dealing with up to conformal isomorphism. It is clear that any attempt at describing the boundary behaviour of $ \pi $ should not disregard  $ \Gamma $, but rather fully describe the connection between the two.

The study of the dynamics of such groups of deck transformations (torsion-free Fuchsian groups), included in the general theory of Klenian groups, as well as their limit set and the non-tangential limit set (see Def. \ref{def-limit-set}), is a topic of independent interest; see \cite{Kat92,MT98}. In particular, Non-tangential limit sets  were introduced by Hedlund \cite{Hed36}, and remained mostly underexplored since the 1970s, when results of Beardon and Maskit \cite{BM74}, Patterson \cite{Pat76}, Sullivan \cite{Sul79},  and others revived the subject, motivated largely by the ergodic theory of geodesic flows on hyperbolic surfaces and Ahlfors' Area Conjecture. However, a rigorous and detailed exploration of the connection of these sets with analytic concepts such as radial limits and cluster sets remains undeveloped, despite a vague connection presented in \cite{FM95} and \cite{BJ97}.  

On the other hand, many attempts have been made to describe prime ends of multiply connected domains, leading to parallel non-equivalent definitions (see e.g.  \cite{Kaufmann,Ahlfors-conformal-invariants,Nakki,suita, Epstein,prime-ends-metric-spaces}, and references therein). For instance, prime ends defined by means of extremal length appear naturally in the theory of quasiconformal invariants (see \cite{suita}, \cite{Nakki}, \cite[Sect. 4.6]{Ahlfors-conformal-invariants}), but this definition does not lend itself well to applicationns in general metric spaces. In \cite{Epstein}, an extension of the theory to higher dimensions is provided, not making use of the Uniformization Theorem nor quasiconformal mappings, resulting in a definition which is purely topological -- but its connection with the universal covering map is obscure. 

Although the previous prime end theories are useful for other purposes, many of them fail to satisfy the basic properties of Carathéodory's construction, for instance not providing a compactification of $ \Omega $, disagreeing in the simply connected case with the compactification by Carathéodory, or only being applicable to finitely connected domains. Moreover, the relation with the universal covering is, in most cases, largely unexplored.

The goal of our paper is to analyze the interplay between the analytic boundary behaviour of the universal covering (in terms of cluster sets and radial limits), the geometry of $ \Omega $, and the group of deck transformations and limit sets, especially in the case where $ \Omega $ has infinite connectivity. As a consequence of our analysis, we describe accesses to boundary points by means of radial limits of the universal covering $ \pi\colon\mathbb{D}\to\Omega $, and develop a prime end theory which is especially well-behaved under $ \pi $. Our tools combine exhaustions of multiply connected domains as first constructed by Ohtsuka \cite{Oht54} with properties of limit sets and non-tangential limit sets of deck transformations of the universal covering, together with specific results of radial limits of holomorphic maps, such as the Lehto-Virtanen Theorem \ref{thm-lehto-virtanen}. Due to the technicalities involved, we postpone the statement of our main theorems and constructions until Section \ref{sect-main-result} (see, in particular, Main Theorem \ref{thm-main-result-complete}). Instead, we state here some of the theory's more remarkable consequences.

First, we analyze the cardinality of the non-tangential limit set $ \Lambda_{NT} $, as a subset of the limit set $ \Lambda $, and the existence of ambiguous points for the universal covering, i.e. points $ e^{i\theta}\in\partial\mathbb{D} $ for which there exists two curves $ \eta_1,\eta_2\subset\mathbb{D} $, ending at $ e^{i\theta} $, such that $ \overline{\pi(\eta_1)}\cap \overline{\pi(\eta_2)}=\emptyset $ (Def. \ref{def-ambiguous}).

\begin{named}{Theorem A}\label{teo:A} {\bf (Cardinality of limit sets)}
		Let $ \Omega $ be a multiply connected domain of connectivity greater than two, and let $ \pi\colon \mathbb{D}\to\Omega $ be a universal covering. Then, the following hold.
	\begin{enumerate}[label={\em (\alph*)}]
		\item $ \Lambda_{NT} $ always consists of uncountably many points.
		\item $ \Lambda\smallsetminus \Lambda_{NT} $ is non-empty if and only if there exists a non-isolated boundary component or an isolated boundary point. Moreover, $ \Lambda\smallsetminus \Lambda_{NT} $ is uncountable if and only if there exists a non-isolated boundary component.
		\item If the radial limit of $ \pi  $ at $ e^{i\theta}\in\partial\mathbb{D} $  exists, then $ e^{i\theta}\in\partial\DD\setminus\Lambda_{NT}$.
		\item There always exists countably many ambiguous points for $ \pi\colon \mathbb{D}\to\Omega $.
	\end{enumerate}
\end{named}

We remark that  \ref{teo:A}(b) is not new: it was already proven in \cite{BM74} for general Kleinian groups. However, we give an independent and easy proof  for the case of torsion-free Fuchsian groups $\Gamma$ such that $\DD/\Gamma$ is a plane domain. Concerning \ref{teo:A}(d), note that Riemann maps do not have ambiguous points, so this is a considerable difference between simply and multiply connected domains.

As a consequence of \ref{teo:A}, we give an analytic characterisation of having a \textit{well-behaved} universal covering in the sense of Hopf-Tsuji-Sullivan, adding ourselves the last equivalent condition of the following list. Notice that, by a theorem of Nevanlinna, conditions (a) and (b) were already known to imply (c); what is remarkable is that (c) itself implies (a) and (b).

\begin{named}{Corollary B}\label{teo:B} {\bf (Hopf-Tsuji-Sullivan Ergodic Theorem)}
	Let $\Omega$ be a hyperbolic plane domain, and let $\pi\colon\DD\to\Omega$ be a universal covering. Then, the following are equivalent.
\begin{enumerate}[label={\normalfont (\alph*)}]
	\item The non-tangential limit set $\Lambda_{NT}$ has zero Lebesgue measure on $\partial\DD$.
	\item  The set $\partial\Omega$ has positive logarithmic capacity.
	\item The radial limits of $\pi$ exist Lebesgue almost everywhere on $ \partial\mathbb{D} $.
\end{enumerate}
\end{named}

The techniques we develop in this paper also allow us to use the universal covering as a means to topologically characterise the boundary of the domain. In this, we show how the universal covering mimics the well-known features of the Riemann map for simply connected domains concerning accesses to boundary points, local connectivity, and prime ends.

First, we describe accessible points on  $\partial  \Omega $ by means of radial limits of the universal covering. The radial limit and the angular cluster set of $ \pi $ at $ e^{i\theta}\in\partial\mathbb{D} $ are denoted by $  \pi^*(e^{i\theta})$  and $ Cl_\mathcal{A}(\pi, e^{i\theta}) $, respectively (see Def. \ref{defi-radial-lim}).
\begin{named}{Theorem C}\label{teo:C}
	{\bf (Accessible points and radial limits)}
	Let $ \Omega $ be a multiply connected domain, and let $ \pi\colon\mathbb{D}\to\Omega $ be a universal covering. Then, all boundary components $ \Sigma\subset \partial \Omega $ are accessible from $ \Omega $.	More precisely, for every boundary component $ \Sigma\subset \partial \Omega $ there exists $ e^{i\theta}\in \partial \mathbb{D} $ such that the angular cluster set $ Cl_\mathcal{A}(\pi, e^{i\theta})\subset\Sigma $.
	
	\noindent Moreover, a point $ p\in\partial \Omega $ is accessible from $ \Omega $ if and only if there exists $ e^{i\theta}\in \partial \mathbb{D} $ such that $ \pi^* (e^{i\theta}) =p$. 
\end{named}

In particular, \ref{teo:C} can be seen as an extension to multiply connected domains of a classical theorem by Lindelöf (Thm. \ref{thm-lindelof}) for Riemann maps. 

In fact, our techniques go further. Indeed, for simply connected domains, there is a bijection between accesses to a boundary point $ x\in\partial\Omega $ (homotopy classes with fixed endpoints of curves landing at $ x $) and points $ e^{i\theta} \in\partial\mathbb{D}$ having $ x $ as radial limit under the Riemann map. This is known as the Correspondence Theorem \ref{correspondence-theorem}.
It is easy to see that the same definition of access is not suitable when working with multiply connected domains, since every accessible point in $ \partial\Omega $ would have infinitely many accesses. However, we are able to define a new concept of access which takes into account the relations induced by the fundamental group, and show that then accesses to $ x\in\partial \Omega $ are in correspondence with points $ e^{i\theta} \in\partial\mathbb{D}/\Gamma$ having $ x $ as radial limit under the universal covering
 -- this is the content of  Theorem \ref{thm:correspondence-mc}.

The culminating point of our construction is to develop a theory of prime ends for multiply connected domains which provides a compactification of $ \Omega $, compatible with the action of the universal covering, and which, for simply connected domains, agrees with the compactification given by Carathéodory. This is the content of Theorem \ref{thm-prime-ends-MC}, which is rather technical and based on all the theory developed in Section \ref{sect-main-result}. Thus, without going into detail for the moment, this theorem gives a natural way of associating subsets of the boundary of the domain with points in the unit circle, in terms of null-chains and impressions, as in the simply connected case.
 
Finally, and as a consequence of our prime end theory, we can enhance the topological description of $ \partial\Omega $. Indeed, we give an extension of the Carathéodory-Torhorst Theorem to the multiply connected case. In order to state it, we define a \textit{true crosscut} of $\Omega$ as an open Jordan arc $C \subset\Omega$ such that $\overline C = C\cup\{a, b\}$, with $a$ and $b$ distinct and contained on the same boundary component of $\Omega$, and such that $\Omega\setminus C$ has exactly two connected components, one of which -- its \textit{true crosscut neighbourhood} -- is simply connected.  We invite the reader to consult Theorem \ref{thm-prime-ends-MC} for the definition of prime ends used here. With this in mind, we prove the following.
 
\begin{named}{Proposition D}{\bf (Carathéodory-Torhorst Theorem for multiply connected domains)}\label{teo:D}
Let $\Omega\subset\widehat{\mathbb{C}}$ be a hyperbolic multiply connected domain, and let $\pi\colon\DD\to\Omega$ be a universal covering with deck transformation group $\Gamma$. Then, there is a one-to-one correspondence between its prime ends and $(\partial\DD\setminus\Lambda_{NT})/\Gamma$. Furthermore, $\pi$ extends continuously to $\partial\DD\setminus\Lambda$ if and only if any true crosscut neighbourhood in $\Omega$ has a locally connected boundary.
\end{named}

Notice that the boundary of a true crosscut neighbourhood consists of the Jordan curve $C$, which is locally connected, plus a continuum contained in $\partial\Omega$, so that local connectedness of the boundary only depends on this continuum on $\partial\Omega$. Furthermore, this theorem is best possible, in the sense that it applies to {any} hyperbolic domain on the Riemann sphere, and that the universal covering cannot extend continuously to any point on the limit set (see Lemma \ref{lemma-limit-sets3}). 

Once this geometric study of the domain $ \Omega $ in terms of crosscuts is developed, we come back to the group of deck transformations and its limit sets, and see how the previous results translate in this setting. It turns out that true crosscuts can be used to characterise universal coverings whose limit set is a Cantor set.

\begin{named}{Corollary E}{\bf (Geometric characterization of limit sets)}\label{cor:E}
Let $\Omega\subset\Chat$ be a hyperbolic multiply connected domain, and let $\pi\colon\DD\to\Omega$ be a universal covering. Then, $\Lambda$ is a Cantor subset of $\partial\DD$ if and only if there exists a true crosscut in $\Omega$.
\end{named}

Moreover, we can use the previous machinery to construct examples of plane domains (and hence, of Fuchsian groups) having `pathological' limit sets, such as a plane domain whose limit set is a Cantor set of positive Lebesgue measure (Sect. \ref{ssec:group-examples}).

Finally, we want to stress the potential of our results and constructions in connection to dynamics. It is well-known that the theory of accesses and prime ends for simply connected domains has several applications in holomorphic dynamics (see e.g. \cite{Rempe,BFJK-Accesses,Mam23,JF23}, and  also \cite[Chap. 17]{Milnor}), and in real dynamical systems (see e.g. \cite{prime-ends}). It is our belief that our deep understanding of the universal covering may lead to similar results in a multiply connected situation.

\subsection*{Structure of the paper} Section \ref{sect-preliminaries} collects basic background in geometry and topology, and Section \ref{section-prelimnary-universal-covering} includes preliminary results on the boundary extension of the universal covering, following the work of Ohtsuka. We have included this background material to try to make the paper as self-contained and as accessible to non-experts as possible, and to emphasize the elementary nature of our approach. Our own results start in Section \ref{sect-main-result}, where we state and prove our main theorem about the boundary behaviour of the universal covering (Thm. \ref{thm-main-result-complete}), and its consequences, including \ref{teo:A} and \ref{teo:B}. Finally, in Section \ref{section-applications}, we apply our methods to describe accesses to the boundary and prime ends of general multiply connected plane domains, addressing \ref{teo:C}, \ref{teo:D}, and \ref{cor:E}.

Throughout the paper, we included several schematic drawings of universal covering maps. One has to keep in mind that these drawings, which are incomplete since we cannot draw either maps with infinite degree or infinitely many curves, have only the purpose of illustrating the associated concepts, and should not be understood as faithful representations.

\section*{Acknowledgements}
We would like to thank Núria Fagella for her detailed and invaluable feedback during the writing of this paper. The second author thanks Marco Abate for helpful discussions in Pisa.

\section{Preliminaries}\label{sect-preliminaries}

In this section we collect all the basic concepts and results that will be needed in the following sections. All of them are well-known, but we present them in a systematic way for the convenience of the reader.

\subsection{Topology and geometry: basic definitions}\label{ssec:top-geom-disc}

We start by giving a precise definition of the landing set of a curve.
\begin{defi}{\bf (Landing set)}
	Given a curve $ \eta\colon \left[ 0,1\right) \to\widehat{\mathbb{C}} $, we consider its \textit{landing set} \[L(\eta)\coloneqq \left\lbrace w\in\widehat{\mathbb{C}}\colon \textrm{ there exists }\left\lbrace t_n\right\rbrace _n\subset\left[ 0,1\right) , \ t_n\to 1\textrm{ such that }\eta(t_n)\to w \right\rbrace .\]
\end{defi}

By definition, $ L(\eta) $ is a connected, compact subset of $ \widehat{\mathbb{C}} $. Thus, it is either a point or a continuum. We say that  $ \eta $ \textit{lands} at $ p\in\widehat{\mathbb{C}} $ if $ L(\eta)=\left\lbrace p\right\rbrace  $, or, equivalently, if  \[\lim\limits_{t\to 1^-}\eta(t)=p.\]

\begin{defi}{\bf (Accessible point)}\label{def-accessible-point}
	Given a domain $\Omega \subset\widehat{\mathbb{C}} $,
	a point $ p\in{\partial} \Omega $ is \textit{accessible} from $ \Omega $ if there is a curve $ \eta\colon \left[ 0,1\right) \to \Omega $ such that $ \lim\limits_{t\to 1^-} \eta(t)=p $ (or, equivalently, if $ \eta $ lands at $ p $).
\end{defi}

\subsection*{Topology and geometry in the unit disk}
Now consider the unit disk $ \mathbb{D} $, and let $ e^{i\theta}\in\partial \mathbb{D}$. We denote the radial segment at $ e^{i\theta}$ by $ R_\theta $, that is \[R_{\theta}\coloneqq R_\theta (t)= \left\lbrace t e^{i\theta}\colon t\in\left[ 0,1\right) \right\rbrace .\]

We shall also use the following notation, which describes the possible ways of approaching a point in the unit circle.
\begin{defi}{\bf (Crosscut neighbourhoods and Stolz angles)}\label{def:cc-nbhd-stolz}
	Let $ e^{i\theta}\in\partial\mathbb{D} $.
	\begin{itemize}
		\item 	A \textit{crosscut} $ C $ is an open Jordan arc $ C\subset\mathbb{D} $ such that $ \overline{C}=C\cup \left\lbrace e^{i\alpha_1},e^{i\alpha_2}\right\rbrace  $, with $ \alpha_1, \alpha_2\in \left[ 0, 2\pi\right) $. If $ e^{i\alpha_1}=e^{i\alpha_2} $, we say that $ C $ is {\em degenerate}; otherwise it is {\em non-degenerate}. The points $ e^{i\alpha_1} $ and $ e^{i\alpha_2} $ are the \textit{endpoints} of $ C $.
		
		\item A \textit{crosscut neighbourhood} of $ e^{i\theta} \in\partial\mathbb{D}$ is an open set $ N\subset\mathbb{D} $ such that $ e^{i\theta}\in\partial N$, and $C\coloneqq \partial N \cap\mathbb{D} $ is a  non-degenerate crosscut.
		We usually write $ N_\theta$ or  $ N_C $, to stress the  dependence on the point $e^{i\theta}$ or on the crosscut $ C $.
		Note that for a crosscut neighbourhood $ N $, $ \partial\mathbb{D}\cap \overline{N} $ is a non-trivial arc.
		\item An (Euclidean) \textit{Stolz angle} at $e^{i\theta}$ is a set of the form 	\[A_{\alpha, \rho}(e^{i\theta})=\left\lbrace z\in\mathbb{D}\colon \left| \textrm{Arg }e^{i\theta}- \textrm{Arg }(e^{i\theta}-z)\right| <\alpha, \ \left| e^{i\theta}- z \right| <\rho \right\rbrace .\]
		\item The {\em horodisk} $ H(e^{i\theta}, R) \subset \mathbb{D}$ of center $ e^{i\theta}\in\partial\mathbb{D} $ and radius $ R>0 $ is the Euclidean disk of radius $ \frac{R}{R+1} $ tangent to $ \partial \mathbb{D} $ at $ e^{i\theta} $.  Its boundary $ \partial H(e^{i\theta}, R) \subset \mathbb{D}$ is called a {\em horocycle}. 
		\item We say that  $ \eta$ \textit{lands non-tangentially} at $e^{i\theta}\in\partial\mathbb{D}$  if $ \eta $ lands at $e^{i\theta}$, and there exists a Stolz angle $ A_{\alpha, \rho} (e^{i\theta})$ at $ e^{i\theta} $  and $ t_0\in (0,1) $ such that  $ \eta(t)\subset A_{\alpha, \rho}(e^{i\theta})  $, for $ t\in (t_0, 1) $.
	\end{itemize}
\end{defi}

See Figure \ref{fig-radial-nbhds} for a geometric representation of the concepts defined above.

\begin{figure}[h]
	\centering
	\captionsetup[subfigure]{labelformat=empty, justification=centering}
	\hfill
	\begin{subfigure}{0.22\textwidth}
		\includegraphics[width=\textwidth]{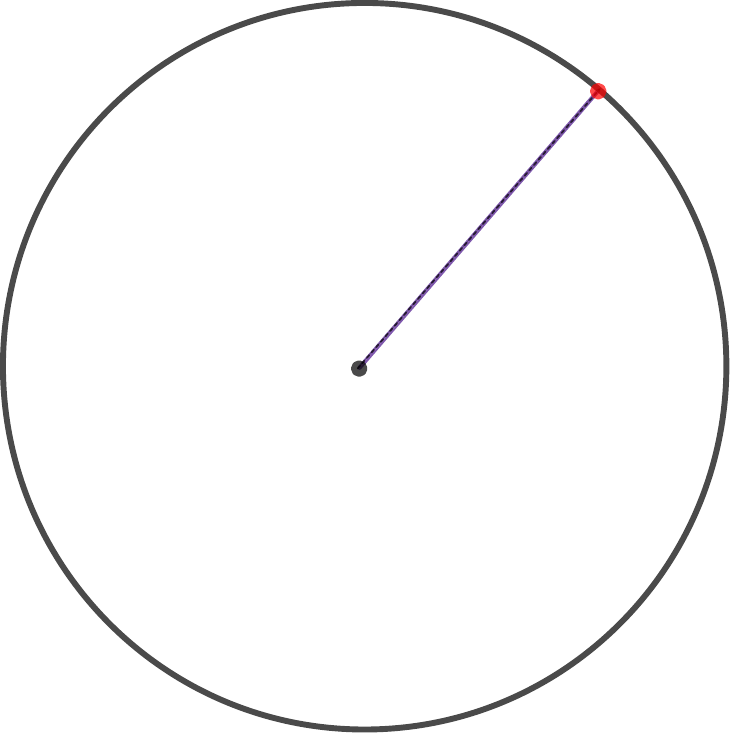}
		\setlength{\unitlength}{\textwidth}
		\put(-0.15, 0.87){$e^{i\theta}$}
		\put(-0.5, 0.65){\footnotesize $R_{\theta}$}
		\caption{\footnotesize Radial segment}
	\end{subfigure}
	\hfill
	\begin{subfigure}{0.22\textwidth}
		\includegraphics[width=\textwidth]{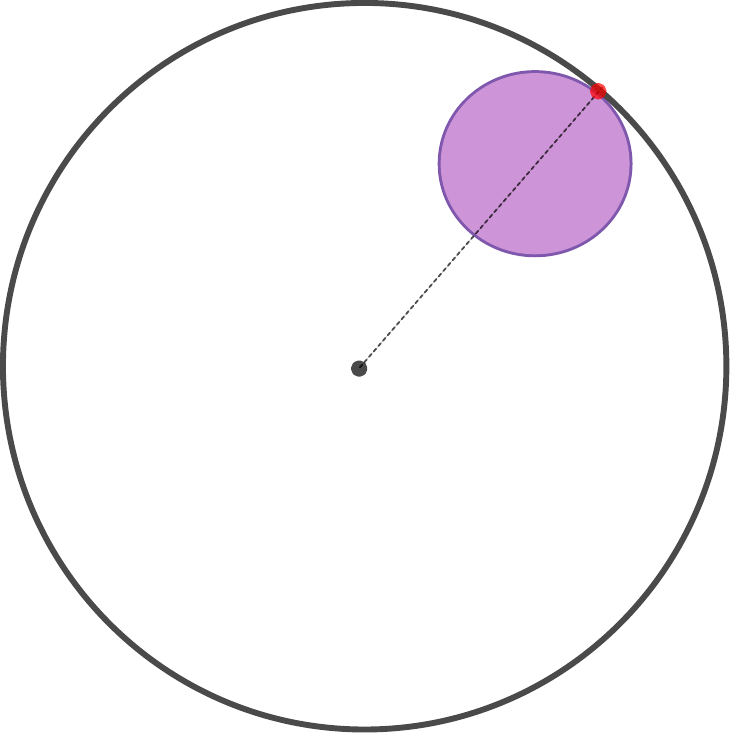}
		\setlength{\unitlength}{\textwidth}
		\put(-0.15, 0.87){$e^{i\theta}$}
		\put(-0.83, 0.7){\footnotesize $H(e^{i\theta}, R)$}
		\caption{\footnotesize Horodisk}
	\end{subfigure}
	\hfill
	\begin{subfigure}{0.22\textwidth}
		\includegraphics[width=\textwidth]{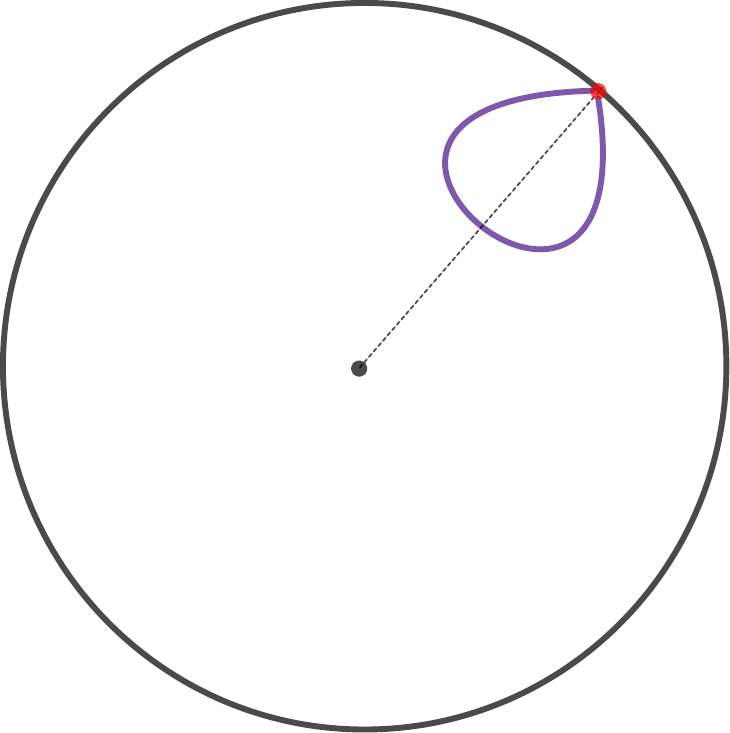}
		\setlength{\unitlength}{\textwidth}
		\put(-0.15, 0.87){$e^{i\theta}$}
			\put(-0.48, 0.7){\footnotesize $C$}
		\caption{\footnotesize Degenerate crosscut}
	\end{subfigure}
	\hfill
	\vspace{1cm}
	
	\hfill
	\begin{subfigure}{0.22\textwidth}
		\includegraphics[width=\textwidth]{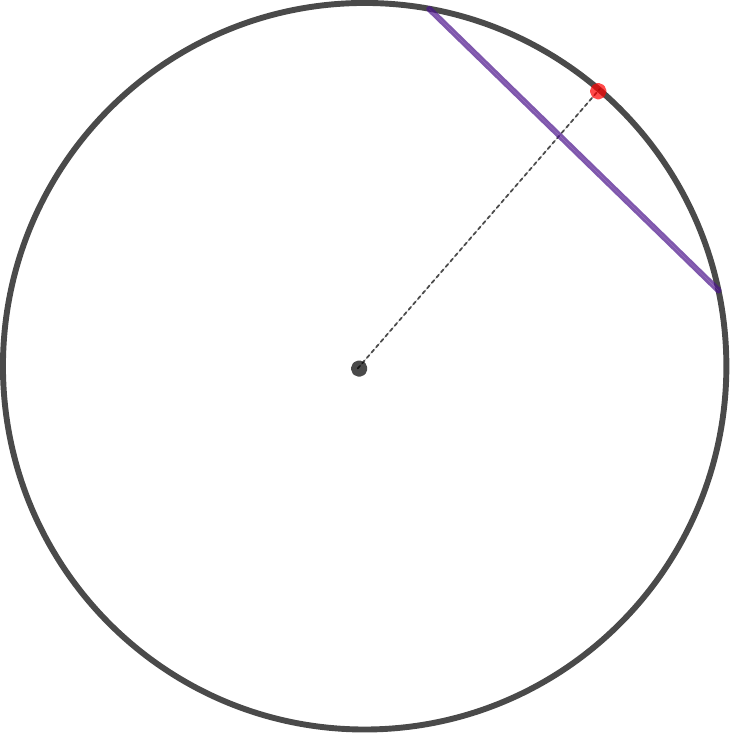}
		\setlength{\unitlength}{\textwidth}
		\put(-0.15, 0.87){$e^{i\theta}$}
		\put(-0.4, 0.8){\footnotesize $C$}
		\caption{\footnotesize (Non-degenerate) crosscut}
	\end{subfigure}
	\hfill
	\begin{subfigure}{0.22\textwidth}
		\includegraphics[width=\textwidth]{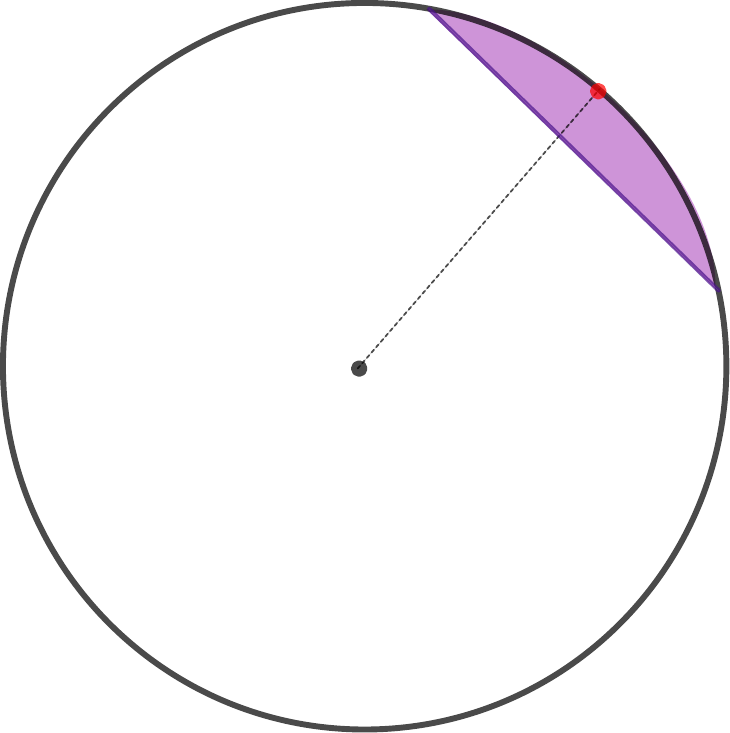}
		\setlength{\unitlength}{\textwidth}
		\put(-0.15, 0.87){$e^{i\theta}$}
		\put(-0.4, 0.8){\footnotesize $N_{\theta}$}
		\caption{\footnotesize Crosscut neighbourhood}
	\end{subfigure}
	\hfill
	\begin{subfigure}{0.22\textwidth}
		\includegraphics[width=\textwidth]{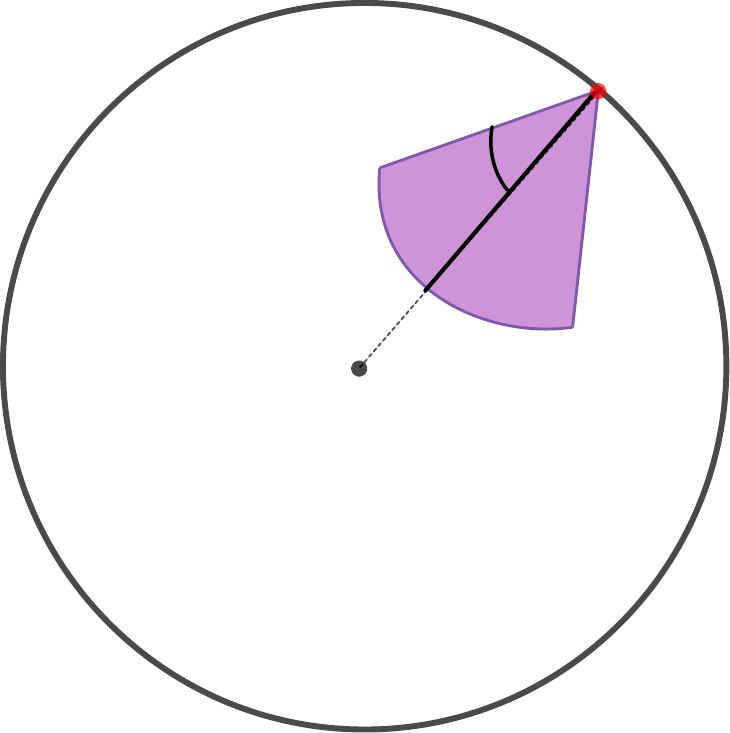}
		\setlength{\unitlength}{\textwidth}
		\put(-0.15, 0.87){$e^{i\theta}$}
		\put(-0.89, 0.65){\footnotesize $A_{\alpha, \rho}(e^{i\theta})$}
			\put(-0.4, 0.75){\footnotesize $\alpha$}
				\put(-0.37, 0.62){\footnotesize $\rho$}
		\caption{\footnotesize Stolz angle }
	\end{subfigure}
	\hfill
	\hfill
	\caption{\footnotesize Different sets related to $ e^{i\theta}\in\partial\mathbb{D} $.}\label{fig-radial-nbhds}
\end{figure}
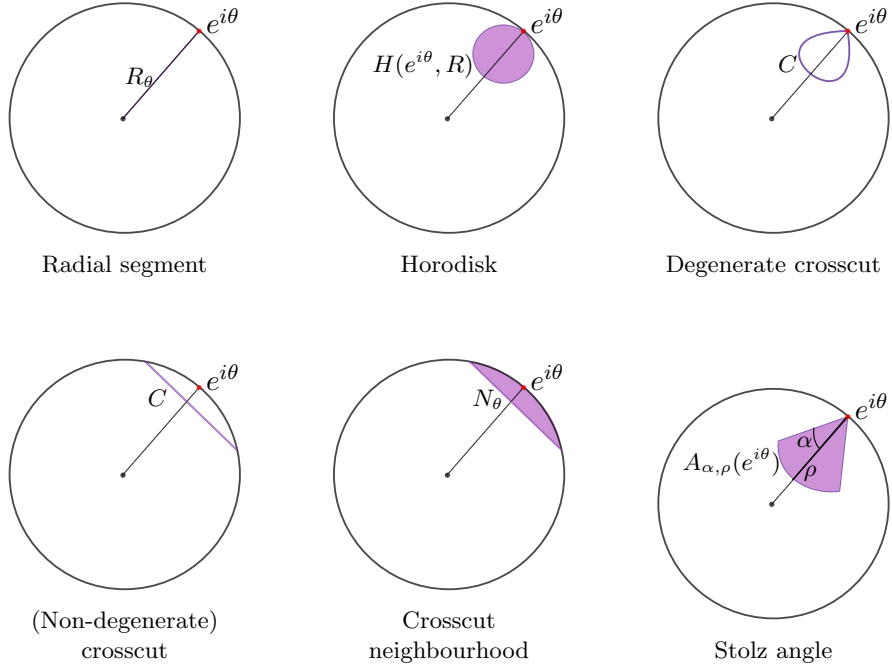

Next, we introduce the notion of null-chain, which is nothing else  than a specific way of choosing a  basis of crosscut neighbourhoods  for points $ e^{i\theta}\in\partial\mathbb{D} $. See Figure \ref{fig-nullchain}.

\begin{defi}{\bf (Null-chain)}\label{null-chain}
	Let $ e^{i\theta}\in\partial\mathbb{D} $. We say that a collection of crosscuts $ \left\lbrace C_n\right\rbrace _n \subset \mathbb{D}$ is a {\em null-chain} for $ e^{i\theta}  $ if\begin{enumerate}[label={(\alph*)}]
		\item the crosscuts $ C_n  $ are pairwise disjoint and have disjoint endpoints;
		\item if $ N_n $ is the crosscut neighbourhood of $ e^{i\theta} $ bounded by $ C_n $, then $ N_{n+1}\subset N_n $;
		\item $ \bigcap\limits_n \overline{N_n}=\left\lbrace e^{i\theta}\right\rbrace  $.
	\end{enumerate}
	We say that $ \left\lbrace N_n\right\rbrace _n $ is the chain of crosscut neighbourhoods associated to $ \left\lbrace C_n\right\rbrace _n $.
\end{defi}

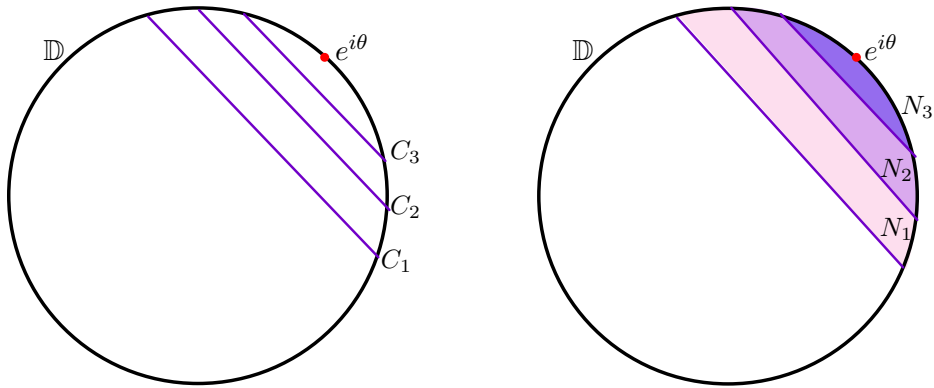
\begin{figure}[h]\centering
	\includegraphics[width=12cm]{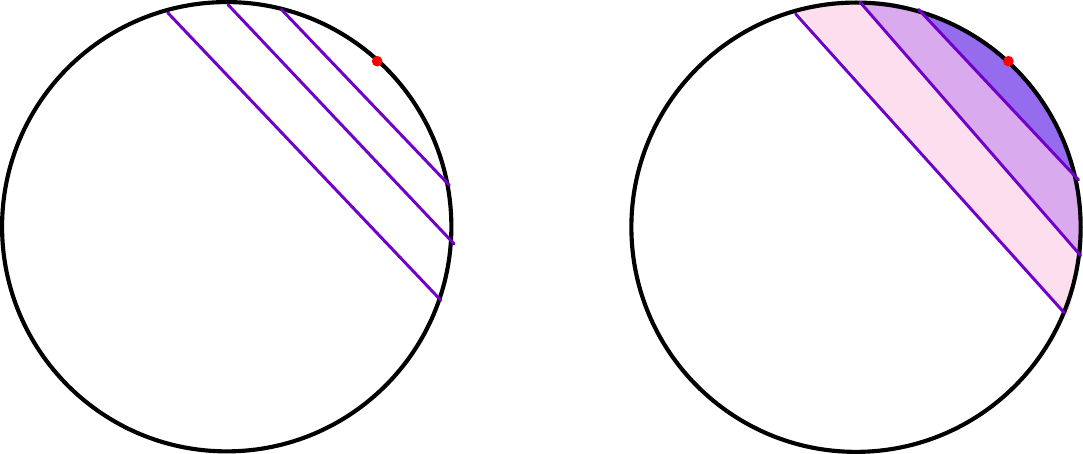}
	\setlength{\unitlength}{12cm}
		\put(-0.06, 0.36){$e^{i\theta}$}
		\put(-0.64, 0.36){$e^{i\theta}$}
		\put(-0.96, 0.36){$\mathbb{D}$}
			\put(-0.38, 0.36){$\mathbb{D}$}
					\put(-0.59, 0.13){\small $C_1$}
						\put(-0.58, 0.19){\small $C_2$}
							\put(-0.58, 0.25){\small $C_3$}
												\put(-0.043, 0.165){\small $N_1$}
							\put(-0.043, 0.23){\small $N_2$}
							\put(-0.02, 0.3){\small $N_3$}
	\caption{\footnotesize Example of null-chain $ \left\lbrace C_n\right\rbrace _n $ and its associated chain of crosscut neighbourhoods $ \left\lbrace N_n\right\rbrace _n $ at the point $ e^{i\theta}\in\partial\mathbb{D} $.}\label{fig-nullchain}
\end{figure}

We say that two null-chains $ \left\lbrace C_n\right\rbrace _n $ and $ \left\lbrace C'_n\right\rbrace _n $  are {\em equivalent} if, for every sufficiently large $ m $, there exists $ n $ such that $ N_n\subset   {N'_{m}} $ and $ {N'_{n}}\subset   {N_{m}} $. This defines an equivalence relation between null-chains.
It is easy to see that null-chains are either equivalent or have eventually disjoint crosscut neighbourhoods \cite[Lemma 17.8]{Milnor}. The equivalence classes of null-chains in $ \mathbb{D} $ correspond bijectively to points in $ \partial\mathbb{D} $ \cite[Thm. 17.12]{Milnor}.

\subsection{The hyperbolic geometry of the unit disk}\label{ssec:hypdisc}
The hyperbolic geometry of plane domains plays an important role in our results, and we give an overview of its relevant properties in Section \ref{subs:hypgeo}. However, the study of hyperbolic geometry begins with the unit disk, and so we take a moment now to study the hyperbolic geometry of $\DD$. We refer the reader to \cite[Chap. 1]{Aba23}, \cite[Chap. 2]{Hub06}, and \cite[Chap. 2]{KeenLakic} for a more in-depth exposition.

Let $ \textrm{Aut}(\mathbb{D}) $ be the group of holomorphic automorphisms of the unit disk $ \mathbb{D} $, i.e.
\[\textrm{Aut}(\mathbb{D})\coloneqq \left\lbrace \gamma\colon\mathbb{D}\to \mathbb{D}\colon \ \gamma \textrm{ is holomorphic and one-to-one} \right\rbrace. \]

It is well-known that $ \textrm{Aut}(\mathbb{D}) $ coincides with the group of Möbius transformations $ M\colon\widehat{\mathbb{C}}\to\widehat{\mathbb{C}} $ which leave the unit disk invariant \cite[Prop. 1.1.5]{Aba23}. Note that such a Möbius transformation can be written explicitly as \[M(z)=e^{i\theta}\dfrac{a-z}{1-\overline{a}z}, \]for some $ \theta\in \left[ 0, 2\pi\right)  $ and $ a\in\mathbb{D} $. Moreover, the action of the group $ \textrm{Aut}(\mathbb{D}) $ in $ \mathbb{D} $ is {\em transitive}, meaning that for any  $ z,w\in \mathbb{D} $, there exists $\gamma \in \textrm{Aut}(\mathbb{D})  $ with $ \gamma (z)=w $  \cite[Corol. 1.1.9]{Aba23}.

The {\em hyperbolic metric} in $ \mathbb{D} $ is defined as the conformal metric whose isometries are precisely the automorphisms of $ \mathbb{D} $, normalised to have curvature $-1$. Equivalently, the hyperbolic metric can be defined in an infinitesimal way by setting the hyperbolic density at $ z\in\mathbb{D} $ to be
 \[  \rho_\mathbb{D}(z)|dz| = \dfrac{2|dz|}{1-\left| z\right| ^2}.\]
Then, the hyperbolic distance between $ z,w\in W $, denoted $\textrm{dist}_\DD(z,w) $, is obtained as usual by integrating $ \rho_\mathbb{D}$ along curves joining $ z $ and $ w $, and taking the infimum.

Given $ z_0\in\mathbb{D} $ and $ R>0 $, the {\em hyperbolic disk} $ D_\mathbb{D} (z_0,R) $ of center $ z_0 $ and radius $ R $ is 
	\[D_\mathbb{D} (z_0,R)\coloneqq \left\lbrace z\in\mathbb{D}\colon \ \textrm{dist}_\DD(z_0,z)<R\right\rbrace .\]
We say that a curve is a (hyperbolic) {\em geodesic} if for every triple of points $ z_1, z_2, z_3 $ on the curve, with $ z_3 $ between $ z_1 $ and $ z_2 $, we have
\[\textrm{dist}_\mathbb{D}(z_1,z_2)= \textrm{dist}_\mathbb{D}(z_1,z_3)+\textrm{dist}_\mathbb{D}(z_3,z_2).\]
Note that a curve satisfying the previous definition may be finite (and then we call it a {\em geodesic arc}), semi-infinite ({\em geodesic ray}) or infinite ({\em infinite geodesic}). Since it will be clear from the context, we will not make any distinction between these cases, and call all of them simply geodesics.

We will need the following properties of the hyperbolic metric, extracted from \cite[Chap. 1.2]{Aba23} and \cite[Chap. 2]{KeenLakic}. 
\begin{prop}{\bf (Properties of the hyperbolic metric in $ \mathbb{D} $)}\label{prop:hypdisc} The following hold.
	\begin{enumerate}[label={\em (\alph*)}]
		\item For all  $\gamma \in \textrm{\em Aut}(\mathbb{D})  $ and $ z,w\in\mathbb{D} $, we have $ \textrm{\em dist}_\mathbb{D}(z,w)= \textrm{\em dist}_\mathbb{D}(\gamma(z),\gamma (w)) $.
		\item The hyperbolic distance is a complete distance inducing on $ \mathbb{D} $ the Euclidean topology. More precisely, hyperbolic disks are Euclidean disks with (possibly) different center and radius.
		\item The geodesics for the hyperbolic distance are the Euclidean diameters of $ \mathbb{D} $ and the intersections with $ \mathbb{D} $ of Euclidean circles orthogonal to $ \partial\mathbb{D} $. In particular, geodesics are smooth curves and any two distinct points of $ \mathbb{D} $ are connected by a unique geodesic.
	\end{enumerate}
\end{prop}
\begin{figure}[h]
	\centering
	\captionsetup[subfigure]{labelformat=empty, justification=centering}
	\begin{subfigure}{0.49\textwidth}\centering
		\includegraphics[width=0.75\textwidth]{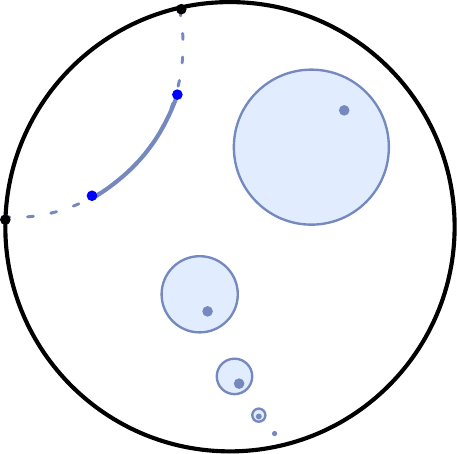}
		\setlength{\unitlength}{\textwidth}
				\put(-0.15, 0.68){$\mathbb{D}$}
								\put(-0.45, 0.57){\small$z$}
								\put(-0.65, 0.41){\small$w$}
												\put(-0.235, 0.55){\small$z_0$}
											\put(-0.34, 0.45){\small$D_\mathbb{D}(z_0, R)$}
		\caption{\footnotesize The hyperbolic disk $ D_\mathbb{D}(z_0, R) $, hyperbolic disks of the same radius approaching $ \partial\mathbb{D} $, and a geodesic from $ z $ to $ w $}
	\end{subfigure}
	\hfill
	\hfill
	\begin{subfigure}{0.49\textwidth}	\centering
		\includegraphics[width=0.75\textwidth]{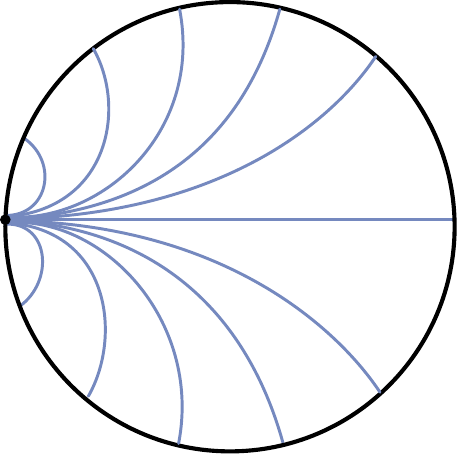}
		\setlength{\unitlength}{\textwidth}
				\put(-0.81, 0.36){$e^{i\theta}$}
		\put(-0.15, 0.68){$\mathbb{D}$}
		\caption{\footnotesize Hyperbolic geodesics at $ e^{i\theta}\in\partial\mathbb{D} $\\ \textcolor{white}{h}}
	\end{subfigure}
	\caption{\footnotesize The hyperbolic metric in the unit disk $ \mathbb{D} $.}\label{fig-properties-hyp-metric}
\end{figure}

\vspace{0.2cm}

\subsection*{Classification of automorphisms of $ \mathbb{D} $.} Let $ \gamma\in\textrm{Aut}(\mathbb{D}) $, so that (as established)
\[\gamma(z)=e^{i\theta}\dfrac{a-z}{1-\overline{a}z} \]
for some $ \theta\in \left[ 0, 2\pi\right)  $ and $ a\in\mathbb{D} $. Such maps are classified into three different types, according to the distribution of their fixed points, or, equivalently, their normal form. This is the content of Proposition \ref{prop-classification}, which  is straightforward from  \cite[Prop. 1.4.1]{Aba23}, \cite[Lemma 1.4.17]{Aba23} and \cite[Prop. 2.1.14]{Hub06}.

 In this setting, we say that two Möbius transformations $ M_1, M_2\colon \widehat{\mathbb{C}}\to\widehat{\mathbb{C}} $ (or, in particular, automorphisms of $ \mathbb{D} $) are {\em conjugate} if there exists a Möbius transformations $ M\colon \widehat{\mathbb{C}}\to\widehat{\mathbb{C}} $ such that $ M\circ M_1=M_2\circ M $.

\begin{prop}{\bf (Classification of automorphisms of $ \mathbb{D} $)}\label{prop-classification} Let $ \gamma\in\textrm{\em Aut}(\mathbb{D}) $, $ \gamma\neq \textrm{\em id}_\mathbb{D} $. Then exactly one of the following holds.
			\begin{enumerate}[label={\em (\alph*)}]
		\item $ \gamma $ has a unique fixed point in $ \mathbb{D} $. In this case, we say that $ \gamma $ is {\em elliptic}, and $ \gamma|_\mathbb{D} $ is conjugate to $ z\mapsto\lambda z $ in $ \mathbb{D} $, with $ \left| \lambda\right| =1 $.
			\item $ \gamma $ has exactly two fixed points and they belong to $ \partial\mathbb{D} $, say $ e^{i\theta_1} $ and $ e^{i\theta_2} $. In this case, we say that $ \gamma $ is {\em hyperbolic}, and $ e^{i\theta_1} $ and $ e^{i\theta_2} $ are the {\em hyperbolic fixed points} of $ \gamma $. Moreover, $ \gamma|_\mathbb{D} $ is conjugate to $ z\mapsto \lambda z $ in $ \mathbb{H} $, with $ \lambda>1 $. If $ \sigma $ denotes the geodesic joining $ e^{i\theta_1} $ and $ e^{i\theta_2} $, then $ \gamma (\sigma)=\sigma $.
		\item $ \gamma $ has a unique fixed point and it belongs to  $ \partial\mathbb{D} $, say $ e^{i\theta} $. In this case, we say that $ \gamma $ is {\em parabolic}, and $ e^{i\theta} $ is the {\em parabolic fixed point} of $ \gamma $. Moreover, $ \gamma|_\mathbb{D} $ is conjugate to $ z\mapsto z\pm 1 $ in the upper half-plane $ \mathbb{H} $, and for any horodisk $ H\coloneqq H(e^{i\theta}, R) $ centered at $ e^{i\theta} $, $ \gamma (H)=H$ and $ \gamma(\partial H) =\partial H$.
	\end{enumerate}
\end{prop}

For a wider exposition on the topic, and more characterizations of the previous types of automorphisms of $ \mathbb{D} $, see \cite[Chap. 1.4]{Aba23}, \cite[Chap. 2.1]{Hub06}, \cite[Chap. 2.4]{KeenLakic}, or \cite{Bea83}.

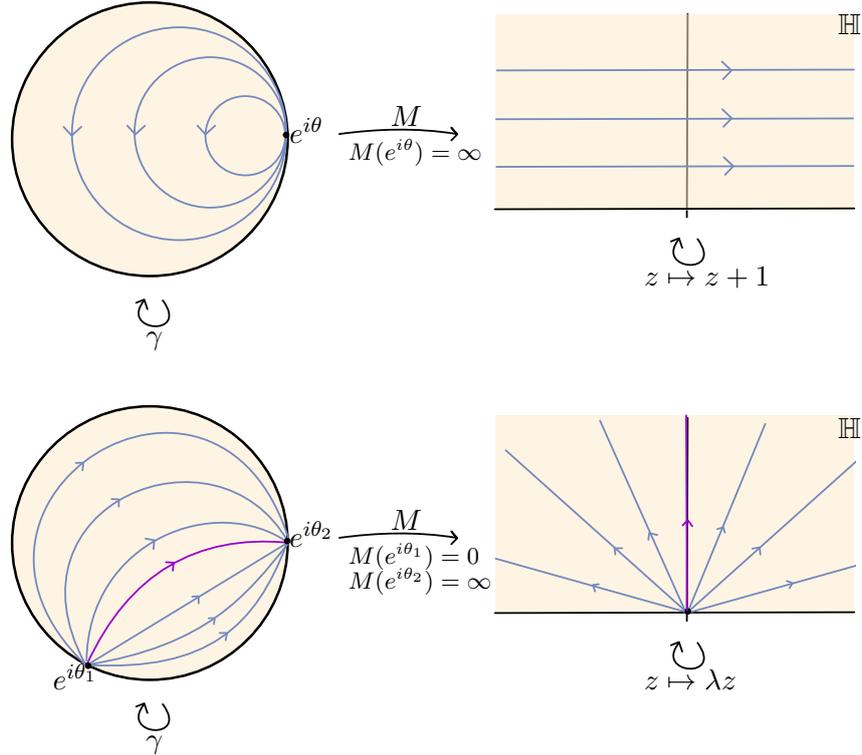
\begin{figure}[h]
	\centering
	\captionsetup[subfigure]{labelformat=empty, justification=centering}
	\begin{subfigure}{0.8\textwidth}
		\includegraphics[width=\textwidth]{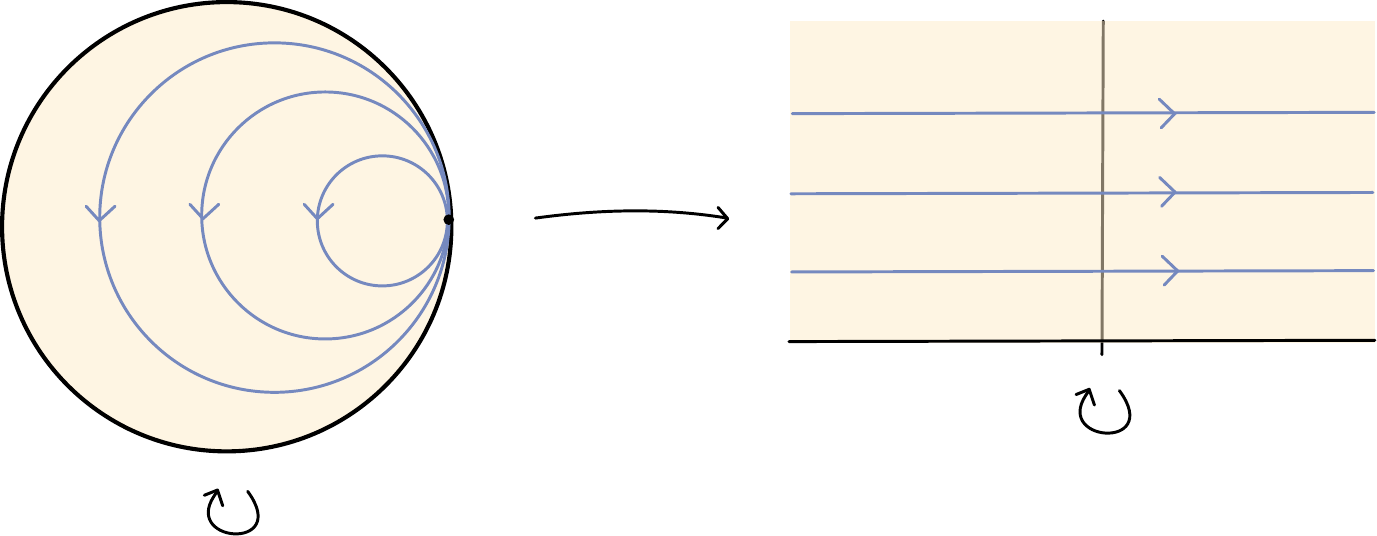}
		\setlength{\unitlength}{\textwidth}
			\put(-0.67, 0.22){$e^{i\theta}$}
						\put(-0.55, 0.24){$M$}
							\put(-0.6, 0.2){\footnotesize $M(e^{i\theta})=\infty$ }
		\put(-0.02, 0.35){$\mathbb{H}$}
			\put(-0.84, -0.02){$\gamma$}
						\put(-0.25, 0.05){$z\mapsto z+1$}
	\end{subfigure}
	\\
	
	\vspace{0.7cm}
	
	\begin{subfigure}{0.8\textwidth}
		\includegraphics[width=\textwidth]{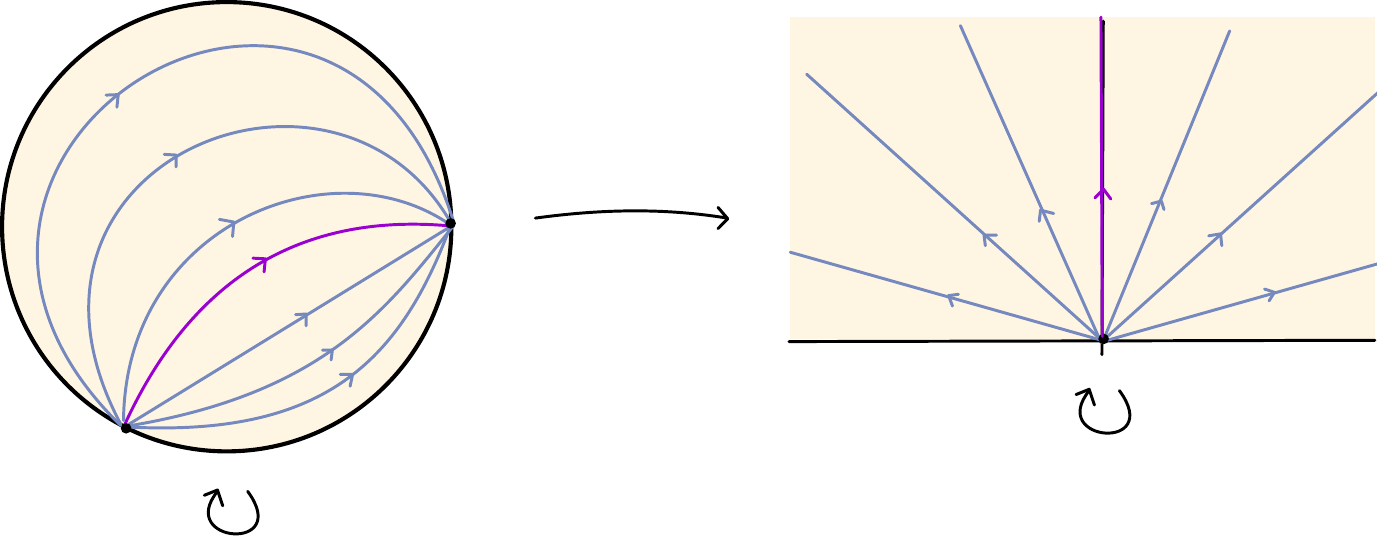}
				\setlength{\unitlength}{\textwidth}
					\put(-0.67, 0.22){$e^{i\theta_2}$}
								\put(-0.95, 0.05){$e^{i\theta_1}$}
		\put(-0.55, 0.24){$M$}
			\put(-0.6, 0.2){\footnotesize $M(e^{i\theta_1})=0$ }
		\put(-0.6, 0.17){\footnotesize $M(e^{i\theta_2})=\infty$ }
		\put(-0.02, 0.35){$\mathbb{H}$}
		\put(-0.84, -0.02){$\gamma$}
		\put(-0.25, 0.05){$z\mapsto \lambda z$}
	\end{subfigure}
	\caption{\footnotesize Schematic representation of the action of a parabolic and a hyperbolic automorphism of $ \mathbb{D} $, respectively.}
\end{figure}

\vspace{0.2cm}
{\bf Hyperbolic Stolz angles.} Although (Euclidean) Stolz angles are the standard way to describe the non-tangential apporach to a point in the unit circle, they are not invariant under  $ \textrm{Aut}(\mathbb{D}) $, which limits their usefulness in dealing with the hyperbolic metric. Thus, let us define {\em hyperbolic Stolz angles}, which describe the same concept of non-tangential approach, but are more convenient to work with, precisely due to their invariance under $ \textrm{Aut}(\mathbb{D}) $ (see \cite[p. 196]{BC08}, \cite[Sect. 2.2]{Aba23}).
\begin{defi}{\bf (Hyperbolic Stolz angle)}
	Let $ e^{i\theta}\in\partial\mathbb{D} $, and let $ \eta_0\subset\mathbb{D} $ be a geodesic ray landing at $ e^{i\theta} $. We define the {\em hyperbolic Stolz angle} of radius $ r$ along $ \eta_0 $ as
	\[\Delta_{\eta_0, r} (e^{i\theta})=\left\lbrace z\in\mathbb{D}\colon \textrm{dist}_\DD (z,\eta_0)<r  \right\rbrace .\]
\end{defi}
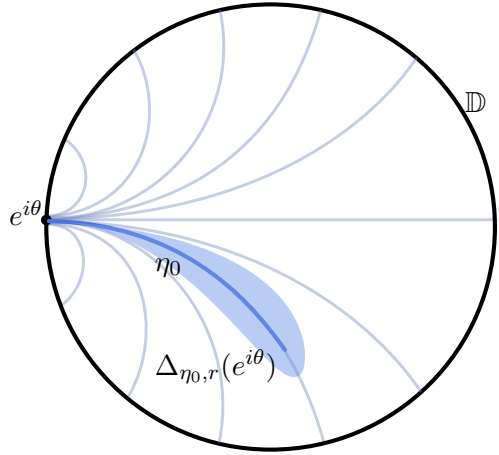
\begin{figure}[h]\centering
	\includegraphics[width=6cm]{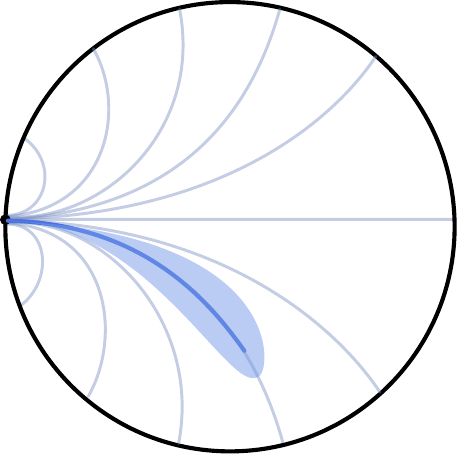}
	\setlength{\unitlength}{6cm}
					\put(-1.07, 0.5){$e^{i\theta}$}
	\put(-0.07, 0.75){$\mathbb{D}$}
					\put(-0.75, 0.4){$\eta_0$}
									\put(-0.75, 0.17){$\Delta_{\eta_0, r}(e^{i\theta})$}
	\caption{\footnotesize Hyperbolic geodesics passing through $ e^{i\theta}\in\partial\mathbb{D} $, and a hyperbolic Stolz angle around one of them.}
\end{figure}
\begin{lemma}{\bf (Euclidean and hyperbolic Stolz angles are equivalent)}
	Let $ e^{i\theta}\in \partial \mathbb{D} $, and let $ \eta_0\subset\mathbb{D} $ be a geodesic ray landing at $ e^{i\theta} $. Then, for every hyperbolic Stolz angle $ \Delta_{\eta_0, r}(e^{i\theta}) $ there exist $ \alpha_1,\alpha_2\in (0, \pi/2) $ and $ K>0 $ such that
	\[A_{\alpha_1, r/K} (e^{i\theta})\subset \Delta_{\eta_0, r}(e^{i\theta})\subset A_{\alpha_2, Kr} (e^{i\theta}).\]
	Conversely, for any Euclidean Stolz angle $ A_{\alpha, r} (e^{i\theta}) $, there exist $ \varepsilon, K >0 $
	such that
	\[\Delta_{\eta_0, r/K}(e^{i\theta})\subset A_{\alpha, r} (e^{i\theta})\cap D(e^{i\theta}, \varepsilon)\subset \Delta_{\eta_0, Kr}(e^{i\theta}).\]
\end{lemma}
\begin{proof}
	The proof can be found in \cite[Lemma 2.2.2]{Aba23} and \cite[Prop. 2.2.4]{Aba23} with a slightly different notation, and with precise estimates on the constant $ K $.
\end{proof}

In the sequel, we shall write $ \Delta(e^{i\theta}) $ for a generic (hyperbolic) Stolz angle at $  e^{i\theta}$, without specifying $ \eta_0 $ or $ r $.

\subsection{Boundary behaviour of holomorphic maps in $ \mathbb{D}$}\label{ssec:boundarydisc}

We are now ready to describe the boundary behaviour of holomorphic maps in $ \mathbb{D} $ in terms of the images of the different neighbourhoods of a point $ e^{i\theta} $ described in Section \ref{ssec:top-geom-disc}. More precisely, let us define the kind of limits we are going to work with.
\begin{defi}{\bf (Radial limits and cluster sets)}\label{defi-radial-lim}
	Let $ h\colon \mathbb{D}\to\widehat{\mathbb{C}}$, and let $ e^{i\theta}\in\partial \mathbb{D} $.
	\begin{itemize}
		\item The {\em radial limit} of $ h $ at $ e^{i\theta} $ is defined to be $h^* (e^{i\theta} )\coloneqq\lim\limits_{r\to 1^-} h(re^{i\theta})$.

		\item	We say that  $ h $  has \textit{angular limit} at $ e^{i\theta}  $ if, for any Stolz angle $ \Delta $ at $ e^{i\theta} $, the limit \[\lim\limits_{z\to e^{i\theta} ,\\ z\in\Delta}h(z) \]  exists.
			\item If $ \eta\colon \left[ 0,1\right) \to\mathbb{D} $ is a curve  landing at $ e^{i\theta} $, the {\em cluster set} $ Cl_\eta(h, e^{i\theta}) $ of $ h $ at $ e^{i\theta} $ along $ \eta $ is defined as the set of values $ w\in\widehat{\mathbb{C}}$ for which there is an increasing sequence $ \left\lbrace t_n\right\rbrace _n \subset (0,1) $ such that $ t_n\to 1 $ and $h(\eta(t_n))\to w $, as $ n\to\infty $.

		\noindent In the particular case when $ \eta=R_\theta $, we say that $ Cl_\eta(h, e^{i\theta}) $ is  the {\em radial cluster set} of $ h $ at $ e^{i\theta} $, and denote it by $ Cl_R(h, e^{i\theta}) $.
		\item The {\em angular cluster set} $ Cl_\mathcal{A}(h, e^{i\theta}) $ of $ h $ at $ e^{i\theta} $ is defined as the union of the cluster sets along all the curves landing non-tangentially at $ e^{i\theta} $. Equivalently,
		\[Cl_\mathcal{A}(h, e^{i\theta})
		= \bigcup_{\alpha 
		} Cl_{\Delta_\alpha(e^{i\theta})}(h, e^{i\theta}).\]
		\item The {\em (unrestricted) cluster set} $ Cl(h, e^{i\theta}) $ of $ \varphi $ at $ e^{i\theta} $ is the set of values $ w\in\widehat{\mathbb{C}} $ for which there is a sequence $ \left\lbrace z_n\right\rbrace _n \subset\mathbb{D}$ such that $ z_n\to e^{i\theta} $ and $h (z_n)\to w $, as $ n\to\infty $.
	\end{itemize}
\end{defi}

Note that cluster sets and radial cluster sets are connected and compact. Hence, they are either a point or a continuum. On the other hand, angular clusters sets may not be closed, since they are defined as an infinite union of closed sets (see also Prop. \ref{prop-escaping-bounded-bungee}).

The existence (or not) of radial limits for meromorphic maps $ h\colon \mathbb{D}\to \Chat$ is a matter of independent interest.  We state now the Lehto-Virtanen Theorem on the existence of radial and angular limits for holomorphic map omitting three values. Note that the universal covering of any hyperbolic surface $ \Omega $ always omits three values, and hence the Lehto-Virtanen Theorem applies.
\begin{thm}{\bf (Lehto-Virtanen, {\normalfont \cite[Sec. 4.1]{Pom92}})}\label{thm-lehto-virtanen}
	Let $ h\colon\mathbb{D}\to\widehat{\mathbb{C}} $ be a meromorphic map omitting three values in $ \widehat{\mathbb{C}}  $. Let $ \eta $ be a curve in $ \mathbb{D} $ landing at $ e^{i\theta} \in\partial \mathbb{D} $. If $ h(\eta) $ lands at a point $ p\in\widehat{\mathbb{C}} $, then $ h $ has angular limit at $ e^{i\theta} $  equal to $ p $. In particular, radial and angular limits are the same.
\end{thm}

Finally, ambiguous points are defined as follows.
\begin{defi}{\bf (Ambiguous points)}\label{def-ambiguous}
	Let $ h \colon \mathbb{D} \to \widehat{\mathbb{C}} $. We say that $ e^{i\theta}\in\partial\mathbb{D} $ is {\em ambiguous} for $ h $ if there exists two curves $ \eta_1,\eta_2\subset\mathbb{D} $ landing at $ e^{i\theta} $ such that 
	\[ Cl_{\eta_1}(h, e^{i\theta})\cap  Cl_{\eta_2}(h, e^{i\theta})=\emptyset.\]
\end{defi}

The Bagemihl ambiguous point theorem states that any map $ h \colon \mathbb{D} \to \widehat{\mathbb{C}} $ has at most countably many ambiguous points \cite[Cor. 2.20]{Pom92}. Note that not even continuity is assumed for $ h $.

\subsection{Hyperbolic geometry of plane domains}\label{subs:hypgeo}
We now use the language and concepts introduced above to define and explore the hyperbolic geometry of a plane domain. We assume, here and throughout, that $\Omega$ is a \textit{hyperbolic} plane domain, i.e. a connected open subset of $\Chat$ omitting at least three points. For more references on this topic, the reader is invited to seek out \cite[Chap. 1]{Aba23}, \cite[Chap. 3]{Hub06}, and \cite[Chap. 7]{KeenLakic}.

\subsection*{Universal coverings and deck transformations}

Let $ \Omega\subset\Chat $ be a hyperbolic domain. Then, by the Uniformization Theorem (see e.g. \cite[Thm. 1.6.5]{Aba23}), there exists a {\em universal covering map}\[\pi\colon X\to \Omega,\] i.e. a local biholomorphism with the path-lifting property, where $ X $ is biholomorphic to $ \mathbb{D} $. In general, we will assume directly $ X=\mathbb{D} $. Then, the universal covering is unique up to precomposition with $ \gamma\in\textrm{Aut}(\mathbb{D}) $. We note the following consequence of the definition of a universal covering.

\begin{lemma}{\bf (Restrictions of the universal covering)}\label{lem:restriction}
	Let $ \Omega $ be a hyperbolic domain, and let $ \pi\colon\mathbb{D}\to \Omega $ be a universal covering. If $ V\subset \Omega $ and $ U $ is a connected component of $ \pi^{-1}(V) $ and $\pi_1(V)\subset \pi_1(\Omega)$, then $ U $ is simply connected and
	\[\pi|_U\colon U\to V\] is a universal covering of $ V $.
\end{lemma}
\begin{obs}
The hypothesis $\pi_1(V)\subset\pi_1(\Omega)$ in Lemma \ref{lem:restriction} is needed to ensure that the  holes of  $V$ correspond to holes of $ \Omega $, i.e. we do not allow connected components of $\Chat\setminus V$  to be fully contained in $\Omega$.
\end{obs}

We say that $ \gamma\in \textrm{Aut}(\mathbb{D}) $ is a {\em deck transformation} of $ \pi $ if $ \pi\circ\gamma =\pi $. It is clear that the deck transformations of $ \pi $ form a subgroup of $ \textrm{Aut}(\mathbb{D}) $, which we denote by $ \Gamma $. Then, $ \Omega=\mathbb{D}/ \Gamma $, in the sense that the quotient space $\DD/\Gamma$ can be given a complex structure that turns it into a Riemann surface biholomorphic to $\Omega$.

The group of deck transformations acts freely and properly discontinuously in $ \mathbb{D} $ (see e.g. \cite[Sect. 1.6]{Aba23} and definitions therein), so no  $ \gamma\in \Gamma$ different from the identity has fixed points in $ \mathbb{D} $. In particular, this implies that deck transformations are either hyperbolic  or parabolic, and thus have two or one fixed points in $ \partial\mathbb{D} $. Given a group of deck transformations $ \Gamma $, we say that $ e^{i\theta}\in\partial \mathbb{D} $ is a {\em hyperbolic  (fixed) point} if there exists $ \gamma\in\Gamma $ such that $ e^{i\theta} $ is a fixed point of $ \gamma $ and $\gamma$ is a hyperbolic disk automorphism. {\em Parabolic (fixed) points} are defined likewise.

Let us denote by $ \pi_1(\Omega) $ the fundamental group of $ \Omega $, which is known to be isomorphic to the group of deck transformations.
\begin{lemma}{\bf (Isomorphism with the fundamental group, {\normalfont\cite[Prop. 1.39]{Hat02}})}\label{lem:deck}
	Let $ \Omega $ be a hyperbolic domain, let $ \pi\colon\mathbb{D}\to \Omega $ be a universal covering, and let $ \Gamma $ be the group of deck transformations of $ \pi $. Then, $$\Gamma\simeq \pi_1(\Omega).$$ In particular, $ \Omega $ is finitely connected if and only if $ \Gamma $ is finitely generated.
\end{lemma}

The isomorphism in Lemma \ref{lem:deck} can be understood as follows. If $\sigma\subset\Omega$ is a closed loop through a point $z\in\Omega$, the unique path-lifting property of univeral coverings means that, given a preimage $\tilde z_0\in\pi^{-1}(z)$, there is a unique lift $\widetilde\sigma$ of $\sigma$ starting at $\tilde z_0$.  Since $\sigma$ is closed, $\widetilde\sigma$ must terminate at another preimage $\tilde z_1$ of $z$. With some extra work (see e.g. \cite[Sect. 1.3]{Hat02}), one can show that this determines a unique deck transformation $\gamma\in\Gamma$ satisfying $\gamma(\tilde z_0) = \tilde z_1$, and that deforming $\sigma$ within its homotopy class does not change $\gamma$. Note that, if $\sigma$ is homotopically trivial, then $\tilde z_0 = \tilde z_1$, and the corresponding deck transformation is the identity $ \textrm{id}_\mathbb{D} $.

We  introduce the {limit set} and {non-tangential limit set} of the deck transformations, defined as follows.
\begin{defi}{\bf (Limit set and non-tangential limit set)}\label{def-limit-set}
	Let $ \Omega $ be a hyperbolic domain, let $ \pi\colon\mathbb{D}\to \Omega $ be a universal covering, and let $ \Gamma $ be the group of deck transformations of $ \pi $. Then, the {\em limit set} and {\em non-tangential limit set} of the group of deck transformations, denoted by $\Lambda(\Gamma)$ and $\Lambda_{NT}(\Gamma)$ respectively, are defined as
	\[\Lambda(\Gamma)= \left\lbrace e^{i\theta}\in \partial\mathbb{D}\colon \gamma_n(0)\to e^{i\theta}, \textrm{ for some }(\gamma_n)_n\subset \Gamma\right\rbrace ,\]
	\[\Lambda_{NT}(\Gamma)= \left\lbrace e^{i\theta}\in \partial \mathbb{D}\colon \gamma_n(0)\to e^{i\theta}\textrm{ non-tangentially}, \textrm{ for some }(\gamma_n)_n\subset \Gamma\right\rbrace .\]
	
	\noindent	We say that $ e^{i\theta}\in\partial \mathbb{D} $  is {\em singular} if $ e^{i\theta}\in \Lambda(\Gamma) $. Otherwise, we say $ e^{i\theta}$ is {\em regular}.
\end{defi}

Equivalently, the non-tangential limit set $ \Lambda_{NT} $ consist of points $ e^{i\theta}\in \partial \mathbb{D} $ for which there exist a Stolz angle $ \Delta(e^{i\theta}) $ and a sequence $ (\gamma_n)_n\subset \Gamma $ such that $ \gamma_n(0)\to e^{i\theta} $, $ \gamma_n(0)\subset \Delta(e^{i\theta}) $.

We say that 0 is the basepoint in the  definition of $ \Lambda $ and $ \Lambda_{NT} $, in the sense that it is the point in which we evaluate the deck transformations. As we will see, there is no dependence of the limit sets on the basepoint, and any other point in the unit disk could have been chosen, so we omit it in the notation. Moreover, if there is no ambiguity, we shall also drop $\Gamma$ from the notation. We gather the following well-known facts about limit sets.

\begin{lemma}{\bf (Properties of limit sets I)}\label{lemma-limit-sets}
	Let $ \Omega $ be a hyperbolic domain, and let $ \pi\colon \mathbb{D}\to\Omega $ be a universal covering. Then, the following holds.
	\begin{enumerate}[label={\em (\alph*)}]
		\item The definition of $ \Lambda $ and $ \Lambda_{NT} $ does not depend on the basepoint.
		\item Fixed points of hyperbolic and parabolic deck transformations belong to $ \Lambda $. More precisely, fixed points of hyperbolic deck transformations are always in $\Lambda_{NT}$, while fixed points of parabolic deck transformations (if there are any) belong to $\Lambda\smallsetminus\Lambda_{NT}$.
		\item $\Lambda_{NT}$ always has either empty or full measure in $\partial\DD$.
	\end{enumerate}
\end{lemma}
\begin{proof}
	For (a), see \cite[Prop. 3.4.1]{Hub06} (for $\Lambda_{NT}$ one can use the same argument); for (b), see  \cite[p. 76]{Hub06}. In (b), hyperbolic fixed points lie in $\Lambda_{NT}$ by \cite[Prop. 2]{BM74}, while parabolic fixed points are in $\Lambda\smallsetminus\Lambda_{NT}$ by \cite[Prop. 3]{BM74}.
	 A proof of (c) can be found in \cite[p. 6]{Fer89}.
\end{proof}

Observe that, according to Lemma \ref{lem:deck}, when the domain $ \Omega $ we are considering is simply connected, $ \Gamma=\left\lbrace \textrm{id}_\mathbb{D}\right\rbrace  $, so the limit set is empty. If $ \Omega $ is doubly connected (see Sect. \ref{subsect-doubly-connected}), then $ \Gamma $ is generated by a unique deck transformation, which is parabolic (if $ \Omega $ is conformally isomorphic to $ \mathbb{D}^* $) or hyperbolic (otherwise). The limit set consists exclusively of the fixed points of this transformation, i.e. $ \Lambda(\Gamma)$  consists of either one or two points. 
However, if the connectivity of the domain is greater than two, then the situation is more involved (in particular, $ \Lambda(\Gamma) $ is always infinite), as the following well-known property shows.

\begin{lemma}{\bf (Properties of limit sets II)}
	Let $ \Omega $ be a multiply connected domain of connectivity greater than two, and let $ \pi\colon \mathbb{D}\to\Omega $ be a universal covering. Then, 
$ \Lambda $ is either $ \partial\mathbb{D} $ or a Cantor subset of $ \partial\mathbb{D} $.
Moreover, both $ \Lambda $ and $\Lambda_{NT}$ are always non-empty.
\end{lemma}
\begin{proof}
	See  e.g. \cite[p. 76]{Hub06}. 
\end{proof}

\begin{obs}{\bf (Measure of the limit set)}\label{rmk:beardon}
	We remark that statement (c) in Lemma \ref{lemma-limit-sets} is not true in general for the limit set: $\Lambda$ may have neither zero nor full measure. Indeed, in \cite[p. 224]{Bea71}, Beardon constructs a Fuchsian group $\Gamma$ such that $\Lambda(\Gamma)$ has positive measure, but it is not equal to $\DD$. However, to the best of our knowledge, it is not known whether Beardon's example (or any like it) corresponds to the deck transformation group of a plane domain -- i.e. whether $\DD/\Gamma$ is of planar character. As a result of the techniques developed in Section \ref{sect-main-result}, we  are able to give such an example in Section \ref{ssec:group-examples}.
\end{obs}
\subsection*{The hyperbolic metric in a plane domain}
Let $ \Omega\subset\widehat{\mathbb{C}} $ be a hyperbolic domain, and let $ \pi\colon\mathbb{D}\to\Omega $ be a universal covering. Then, the \textit{hyperbolic density} in $ \Omega $ is (locally) defined as \[ \rho_\Omega(z)\coloneqq \rho_\mathbb{D}(\Pi(z))\cdot \left| \Pi'(z)\right| ,\]where $ \Pi $ is a branch of $ \pi^{-1} $. It can be proved that $ \rho_\Omega $ does not depend on $ \pi $ nor on $ \Pi $, and hence is well-defined  \cite[Thm. 10.3]{BeardonMinda}. The hyperbolic distance between $ z,w\in \Omega $, denoted by $ \textrm{dist}_\Omega(z,w) $,  is obtained by integrating $ \rho_\Omega $ along curves joining $ z $ and $ w $, and taking the infimum. 

Given $ z_0\in\Omega $ and $ R>0 $, the {\em hyperbolic disk} $ D_\Omega (z_0,R) $ of center $ z_0 $ and radius $ R $ is 
\[D_\Omega (z_0,R)\coloneqq \left\lbrace z\in\Omega\colon \ \textrm{dist}_\Omega(z_0,z)<R\right\rbrace .\]
A {\em geodesic}  in $\Omega$ is defined as the image under $\pi$ of a geodesic in $\DD$. Again, we shall use the word geodesic to refer indistinctly to geodesic arcs, geodesic rays, or infinite geodesics. Note that a geodesic arc between two points $a$ and $b$ in $\Omega$ is defined as the image under $\pi$ of a geodesic in $\DD$ joining a preimage of $a$ to a preimage of $b$. Because such preimages can be combined in infinitely many different ways, there are infinitely many geodesics joining $a$ to $b$, pairwise non-homotopic  (see  Lemma \ref{lem:superSP}(d)).

The following are standard properties of the hyperbolic metric in $ \Omega $, whose proofs follow from the definition of $ \rho_\Omega $.
\begin{lemma}{\bf (Properties of the hyperbolic metric in $ \Omega $)}\label{lem:superSP}
Let $\Omega\subset\Chat$ be a hyperbolic domain, and let $\pi\colon\DD\to\Omega$ be a universal covering. Consider the hyperbolic metrics $ \rho_\mathbb{D} $ and $ \rho_\Omega $ in $ \mathbb{D} $ and $ \Omega $, respectively. The following hold. 
	\begin{enumerate}[label={\em (\alph*)}]
		\item  $ \pi$ is a local isometry between $ (\mathbb{D}, \rho_\mathbb{D}) $ and $ (\Omega, \rho_\Omega) $.
		\item Let $z, w\in\Omega$. Then, for any $ \tilde z \in\pi^{-1}(z)$, we have
		\[ \textrm{\em dist}_\Omega(z,w)=\min \left\lbrace \textrm{\em dist}_\DD(z, \gamma(w))\colon \ \gamma \in\Gamma \right\rbrace .\]
		In particular, there exists $\tilde w\in\pi^{-1}(w)$ such that
		\[ \textrm{\em dist}_\Omega(z, w) = \textrm{\em dist}_\DD(\tilde z, \tilde w). \]
		\item $(\Omega, \rho_\Omega)$ is a complete metric space.
		\item Any two points $z$ and $w$ in $\Omega$ can be joined by a distance-minimizing geodesic. Furthermore, any two geodesics joining $z$ and $w$ are not homotopic to each other.
		\item Let $\left\lbrace z_n\right\rbrace _n\subset\Omega$ be a sequence such that $z_n\to\partial\Omega$ as $n\to+\infty$. Then, for any $r > 0$,
		\[ \mathrm{diam}_{\widehat{\mathbb{C}}}(D_\Omega(z_n, r))\to 0 \text{ as $n\to+\infty$,} \]
		where $\mathrm{diam}_{\widehat{\mathbb{C}}}$ denotes the spherical diameter.
		\end{enumerate}
\end{lemma}

\subsection*{The limit set $\Lambda(\Gamma)$ and the cluster sets $ Cl(\pi, \cdot) $} One of the difficulties when working with multiply connected domains is that the universal covering becomes very badly behaved around points of the limit set. More precisely, let us recall the following result.

\begin{lemma}{\bf (Properties of limit sets III)}\label{lemma-limit-sets3}
	Let $ \Omega $ be a multiply connected domain, let $ \pi\colon \mathbb{D}\to\Omega $ be a universal covering, and let $ e^{i\theta} \in\partial\mathbb{D}$. Then, the following holds.
	\begin{enumerate}[label={\em (\alph*)}]
		\item If $ e^{i\theta}\in\Lambda $, then $ Cl(\pi, e^{i\theta}) =\overline{\Omega}$.
		\item  If $ e^{i\theta}\notin\Lambda $, then there exists a null-chain $ \left\lbrace C_n\right\rbrace _n $, with crosscut neighbourhoods $  \left\lbrace N_n\right\rbrace _n$, such that $ \pi|_{N_n} $ is univalent and $ \pi(N_n) $ is simply connected, for all $ n\geq 0 $.
	\end{enumerate}
\end{lemma}
\begin{proof}
Statement	(a) is a reformulation of the definition of the limit set in terms of cluster sets (preimages of any point in $ \Omega $ accumulate at every point of $ \Lambda $), and (b) can be found in \cite[Thm. 3.5]{Oht51}.
\end{proof}

\section{Boundary behaviour of the universal covering. Preliminary results}\label{section-prelimnary-universal-covering}

In this section we collect some known results about the  boundary behaviour of the universal covering, starting with the simplest (and the most studied case): when $ \Omega $ is simply connected and $ \pi $ is the Riemann map. Indeed, first, we gather results concerning the boundary behaviour of Riemann maps, including accessibility and prime ends. Second, we show that these results extend easily to doubly connected domains. Finally, we recall the work of Ohtsuka concerning multiply connected domains \cite{Oht54}, which is the basis for our work presented in the following sections.
\subsection{Boundary behaviour of the Riemann map}\label{subsect-Riemann-map}

In the particular case where $ \Omega $ is simply connected, its universal covering is a homeomorphism, known as the {\em Riemann map}. In contrast with the universal covering of a multiply connected domain, the Riemann map has been widely (and successfully) studied. Next, we briefly collect the main results on this topic. For more details and proofs see e.g. \cite[Sect. 17]{Milnor}, \cite[Chap. 2]{Pom92}, as well as the classical references \cite{Noshiro-clustersets,Collingwood-Lohwater}.

Let  $ \Omega\subset\Chat $ be a simply connected domain, and consider $ \varphi\colon\mathbb{D}\to \Omega $ a Riemann map. The first crucial observation (that is no longer true for a multiply connected domain) is that, for $ \left\lbrace z_n \right\rbrace _n\subset\mathbb{D}$, $ z_n\to\partial\mathbb{D} $ if and only if $ \varphi(z_n)\to\partial\Omega $.
Moreover, by the Carathéodory-Torhorst Theorem \cite[Thm. 2.1]{Pom92},  $ \varphi $ extends continuously to $ \overline{\mathbb{D}} $ if and only if $ \partial \Omega $ is locally connected.

Even if $\partial\Omega$ is not locally connected, the following theorem, due to Fatou, Riesz and Riesz, ensures the existence of radial limits almost everywhere.

\begin{thm}{\bf (Existence of radial limits) }\label{thm-FatouRiez}
	Let $ \varphi\colon\mathbb{D}\to \Omega$ be a Riemann map. Then, for Lebesgue almost every point $ e^{i\theta} \in\partial \mathbb{D}$, the radial limit $ \varphi^*(e^{i\theta} ) $ exists. Moreover, fixed $ e^{i\theta}\in \partial\mathbb{D} $ for which $ \varphi^*(e^{i\theta} ) $ exists, then $ \varphi^*(e^{i\theta} )\neq  \varphi^*(e^{i\alpha} )  $, for Lebesgue almost every point $ e^{i\alpha} \in\partial \mathbb{D}$.
\end{thm}


\subsection*{{Accessing the boundary of a simply connected domain}}
A classical result about Riemann maps is the following, which can be seen as a stronger version of the Lehto-Virtanen Theorem \ref{thm-lehto-virtanen} for Riemann maps.
\begin{thm}{\bf (Lindelöf Theorem, {\normalfont \cite[Thm. I.2.2]{CarlesonGamelin}})}\label{thm-lindelof}
	Let $ \varphi\colon\mathbb{D}\to \Omega$ be a Riemann map, and	let  $ \eta\colon \left[ 0,1\right) \to \Omega $  be a curve which lands at a point $ p\in\partial\Omega$. Then, the curve $ \varphi^{-1}(\eta) $ in $ \mathbb{D} $ lands at some point $e^{i\theta}\in\partial \mathbb{D} $. Moreover, $ \varphi $ has angular limit at $ e^{i\theta} $ equal to $ p $. In particular, curves that land at different points on $ \partial\Omega $ correspond to curves which land at different points on $ \partial\mathbb{D} $.
\end{thm}

Therefore, Theorem \ref{thm-lindelof} relates accessible points (Def. \ref{def-accessible-point}) and radial limits.  A stronger relationship is true, as we will see in the Correspondence Theorem \ref{correspondence-theorem}. First, we need the definition of access.
\begin{defi}{\bf (Access)}\label{def:access-sc}
	Given a simply connected domain $ \Omega\subset\widehat{\mathbb{C}} $, let $ z_0\in \Omega$ and $ p\in\partial \Omega $. A homotopy class (with fixed endpoints) of curves $ \eta\colon \left[ 0,1\right] \to\widehat{\mathbb{C}} $ such that $ \eta(\left[ 0,1\right))\subset \Omega $, $ \eta(0)=z_0 $ and $ \eta(1)=p $ is called an \textit{access} from $ \Omega $ to $ p $.
\end{defi}

Note that, since $ \Omega $ is simply connected, the definition of access does not depend on the chosen basepoint. Accesses to the boundary are in bijection with points in $ \partial \mathbb{D} $ for which $ \varphi^* $ exists, as shown in the following well-known theorem \cite[p.35, Ex. 5]{Pom92}. For a complete proof, see \cite{BFJK-Accesses}.

\begin{thm}{\bf (Correspondence Theorem)}\label{correspondence-theorem} Let $ \Omega\subset\widehat{\mathbb{C}} $ be a simply connected domain, let $ \varphi\colon\mathbb{D}\to \Omega $ be a Riemann map, and let $ p\in\partial \Omega $. Then, there is a one-to-one correspondence between accesses from $ \Omega $ to $ p $ and the points $ e^{i\theta}\in \partial \mathbb{D} $ such that $ \varphi^*(e^{i\theta})=p $. The correspondence is given as follows.\begin{enumerate}[label={\em (\alph*)}]
		\item If $ \mathcal{A} $ is an access to $ p\in\partial \Omega $, then there is a point $ e^{i\theta}\in\partial \mathbb{D} $ with $ \varphi^*(e^{i\theta})=p  $. Moreover, different accesses to $p$ correspond to different points in $ \partial\mathbb{D} $.
		\item If the radial limit $ \varphi^* $ at a point $ e^{i\theta}\in\partial\mathbb{D} $ exists and is equal to $ p\in \partial \Omega $, then there exists an access $ \mathcal{A} $ to $ p $. Moreover, for every curve $ \eta\subset \mathbb{D} $ landing at $ e^{i\theta} $, if $ \varphi(\eta) $ lands at some point $ q\in\widehat{\mathbb{C}} $, then $ p=q $ and $ \varphi(\eta)\in \mathcal{A} $.
	\end{enumerate}
\end{thm}

\subsection*{Prime ends for a simply connected domain}
Prime ends give a more geometrical approach to the same concepts, and are defined in terms of equivalence classes of null-chains (which where introduced in Def. \ref{null-chain}). Indeed, the concepts of crosscuts, crosscut neighbourhoods and null-chains extend trivially to arbitrary simply connected domains as follows. 

Consider a simply connected domain $ \Omega $, and a Riemann map $ \varphi\colon\mathbb{D}\to\Omega $. Let 
 $\varphi(0)= z_0\in \Omega $. We say that $ C $ is a \textit{crosscut} in $ \Omega $ if $ C $ is an open Jordan arc in $ \Omega $ such that $ \overline{C}=C\cup \left\lbrace a,b\right\rbrace  $, with $ a,b\in \partial \Omega $. If $ C $ is a crosscut of $ \Omega $ and  $ z_0\notin C $, then $ \Omega\smallsetminus C $ has exactly one component which does
 not contain $ z_0 $; let us denote this component by $ \Omega_C $. We say that $ \Omega_C $ is the \textit{crosscut neighbourhood} in $ \Omega $ associated to $ C$.  A sequence of crosscuts $ \left\lbrace C_n\right\rbrace _n \subset \Omega$ is a {\em null-chain} in $ \Omega $, if the crosscuts $ \left\lbrace C_n\right\rbrace _n $ have disjoint closures, the corresponding crosscut neighbourhoods $ \left\lbrace \Omega_{C_n}\right\rbrace _n $  are nested, and  the spherical diameter of $ C_n $ tends to zero as $ n\to\infty $. 

As it is stated next, null-chains can be chosen so that they are compatible with the action of the Riemann map $ \varphi $, allowing us to define prime ends and impressions (Thm. \ref{thm-prime-ends-sc}).

\begin{defi}{\bf (Admissible null-chain)}
		Let $ \Omega $ be a simply connected domain, and let $ \varphi\colon\mathbb{D}\to\Omega $ be a Riemann map. A null-chain $ \left\lbrace C_n\right\rbrace _n \subset \mathbb{D}$ is {\em admissible} for $ \varphi $ if $ \left\lbrace \varphi(C_n)\right\rbrace _n $ is a null-chain in $ \Omega $.
\end{defi}

\begin{thm}{\bf (Prime ends and impressions, {\normalfont \cite[Thm. 2.15]{Pom92}})}\label{thm-prime-ends-sc}
	Let $ \Omega $ be a simply connected domain, and let $ \varphi\colon\mathbb{D}\to\Omega $ be a Riemann map. For all $ e^{i\theta}\in\partial\mathbb{D} $, there exists a null-chain $ \left\lbrace C_n\right\rbrace _n $ at $ e^{i\theta} $, which is admissible for $ \varphi $.
	The {\em prime end} $ P(e^{i\theta}) $  is the equivalence class of admissible null-chains for $ \varphi $ at $ e^{i\theta} $.
	
\noindent 	 Then, the {\em impression} of the prime end $ P(e^{i\theta}) $ under $ \varphi $ is \[I(\varphi, P( e^{i\theta}))=\bigcap_{n}\overline{\varphi (N_n)}\subset \partial \Omega,\] where $ \left\lbrace C_n\right\rbrace _n $ is an admissible null-chain for $ \varphi $, with crosscut neighbourhoods $ \left\lbrace N_n\right\rbrace _n $. The impression $ I(\varphi, P( e^{i\theta})) $ does not depend on the chosen null-chain $ \left\lbrace C_n\right\rbrace _n $, provided it is admissible.
\end{thm}

Given a prime end $ P $, we say that $ p\in{\partial }\Omega $ is a \textit{principal point} of $ P(e^{i\theta}) $ under $ \varphi $ if $ P $ can be
represented by an admissible  null-chain $ \left\lbrace C_n\right\rbrace _n $ at $ e^{i\theta} $ satisfying that, for all $ r>0 $, there exists $ n_0 $ such that the crosscuts
$ \varphi(C_n)$ are contained in the disk $ D(p, r) $ for $ n\geq n_0 $.
Let $ \varPi(\varphi,P( e^{i\theta}))$ denote the set of all principal points of $ P(e^{i\theta}) $ under $ \varphi $.

The following theorem gives explicitly the relation between cluster sets and prime ends, and between radial cluster sets and principal points.

\begin{thm}{\bf (Prime ends and cluster sets, \normalfont{\cite[Thm. 2.16]{Pom92}}\bf )}\label{thm-prime-ends}
	Let $ \varphi\colon\mathbb{D}\to \Omega $ be a Riemann map, and let $ e^{i\theta}\in \partial\mathbb{D} $. Then,
	\[ I(\varphi, P( e^{i\theta}))=Cl(\varphi, e^{i\theta}),\]
and  \[\varPi(\varphi,P( e^{i\theta}))=Cl_R(\varphi, e^{i\theta})=Cl_\mathcal{A}(\varphi, e^{i\theta})=\bigcap_{\eta }Cl_\eta(\varphi, e^{i\theta}),\] where $ \eta $ runs through all curves in $ \mathbb{D} $ landing at $ e^{i\theta} $.
\end{thm}


As an easy consequence of the previous theorem, we have the following result (compare with \cite[Ex. 2.6.1]{Pom92}).
\begin{cor}{\bf (Riemann maps do not have ambiguous points)}
		Let $ \varphi\colon\mathbb{D}\to \Omega $ be a Riemann map. Then, $ \varphi $ has no ambiguous points.
\end{cor}

Doubly connected domains are somehow special, compared to general multiply connected domains (note that the corresponding limit set is finite, and the fundamental group is generated by a single element).
We sketch next how the previous results on the Riemann map extend straightforward for the universal covering of doubly connected domains.

\subsection{Boundary behaviour of the universal covering of a doubly connected domain}\label{subsect-doubly-connected}

Let $ \Omega $ be a doubly connected hyperbolic domain, so that it is conformally isomorphic (see e.g. \cite[Proposition 3.2.1]{Hub06}) to a non-degenerate round annulus
\[\mathcal{A}_r\coloneqq \left\lbrace z\in\mathbb{C}\colon r<\left| z\right| <1 \right\rbrace, \hspace{0.3cm}r>0,\] or to the punctured disk
\[\mathbb{D}^*\coloneqq \mathbb{D}\smallsetminus\left\lbrace 0\right\rbrace. \]

Assume first that $ \Omega $ is conformally isomorphic to a non-degenerate annulus $ \mathcal{A}_r $, i.e. there exists $ \varphi\colon\mathcal{A}_r \to\Omega$ conformal. Then, the universal covering map $ \pi\colon\mathbb{D}\to\Omega $ can be written explicitly as
\[\pi=\varphi\circ\exp\circ M,\]where $ M $ is a Möbius transformation sending $ \mathbb{D} $ to a suitable vertical strip (see Fig. \ref{fig-covering-annulus}). Moreover, the deck transformation group $ \Gamma $ is generated by a unique hyperbolic automorphism $ \gamma $. By an appropriate choice of $ M $, we can assume that the hyperbolic fixed points of  $ \gamma $ are 1 and $ -1 $.

Then, for the hyperbolic fixed points 1 and $ -1 $ the radial limit does not exist, and the image of any crosscut neighbourhood of these points covers $ \Omega $ infinitely many times. Note that these points are ambiguous for $ \pi $.

On the other hand, if $ \Omega $ is conformally isomorphic to the punctured disk $ \mathbb{D}^* $ via the isomorphism $ \varphi $, the universal covering map $ \pi\colon\mathbb{D}\to\Omega $ can be written explicitly as \[\pi=\varphi\circ\exp\circ M,\]where $ M $ is a Möbius transformation sending $ \mathbb{D} $ to the left half-plane $ \mathbb{H}_- $ (see Fig. \ref{fig-covering-annulus}). The deck transformation group $ \Gamma $ is generated by a parabolic automorphism $ \gamma $, whose fixed point may be assumed to be $ -1 $. Then, $ \varphi^*(-1) $ exists and it is the puncture of $ \Omega $.

\begin{figure}[h]
	\centering
		\includegraphics[width=14cm]{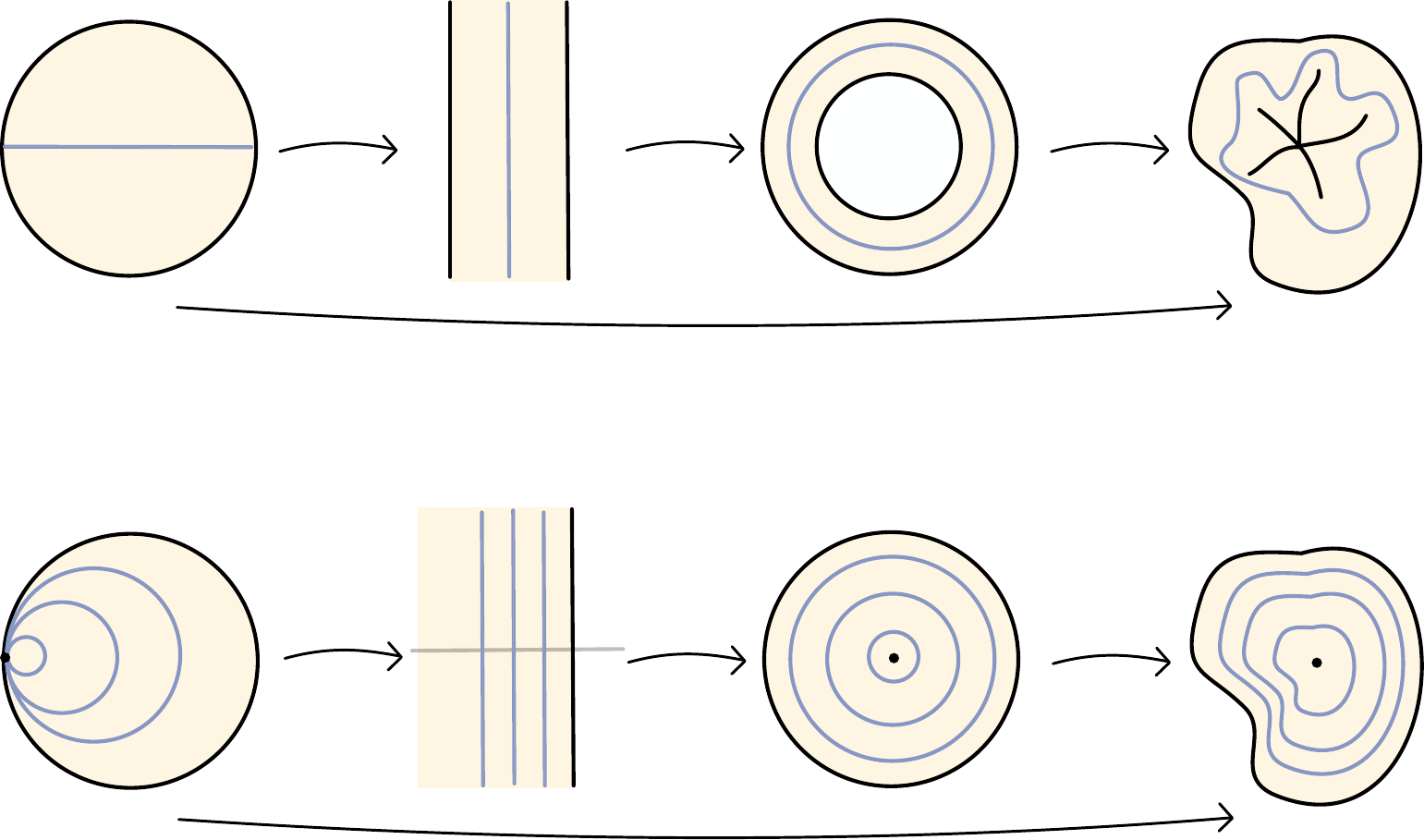}
		\setlength{\unitlength}{14cm}
				\put(-0.84, 0.55){$\mathbb{D}$}
					\put(-0.595, 0.57){$S$}
						\put(-0.31, 0.55){$\mathcal{A}_r$}
							\put(-0.02, 0.55){$\Omega$}
					\put(-0.775, 0.495){$M$}
						\put(-0.54, 0.5){$\exp$}
							\put(-0.23, 0.5){$\varphi$}
								\put(-0.5, 0.37){$\pi$}
								\put(-0.84,0.2){$\mathbb{D}$}
								\put(-0.595, 0.21){$\mathbb{H}$}
								\put(-0.31, 0.2){$\mathbb{D}^*$}
								\put(-0.02, 0.2){$\Omega$}
								\put(-0.775, 0.135){$M$}
								\put(-0.54, 0.14){$\exp$}
								\put(-0.23, 0.14){$\varphi$}
								\put(-0.5, 0.01){$\pi$}
	\caption{\footnotesize Universal covering map of a doubly connected domain $ \Omega $, conformally isomorphic to a non-denegenerate annulus $ \mathcal{A}_r $ or a punctured disk $ \mathbb{D}^* $, respectively.}\label{fig-covering-annulus}
\end{figure}
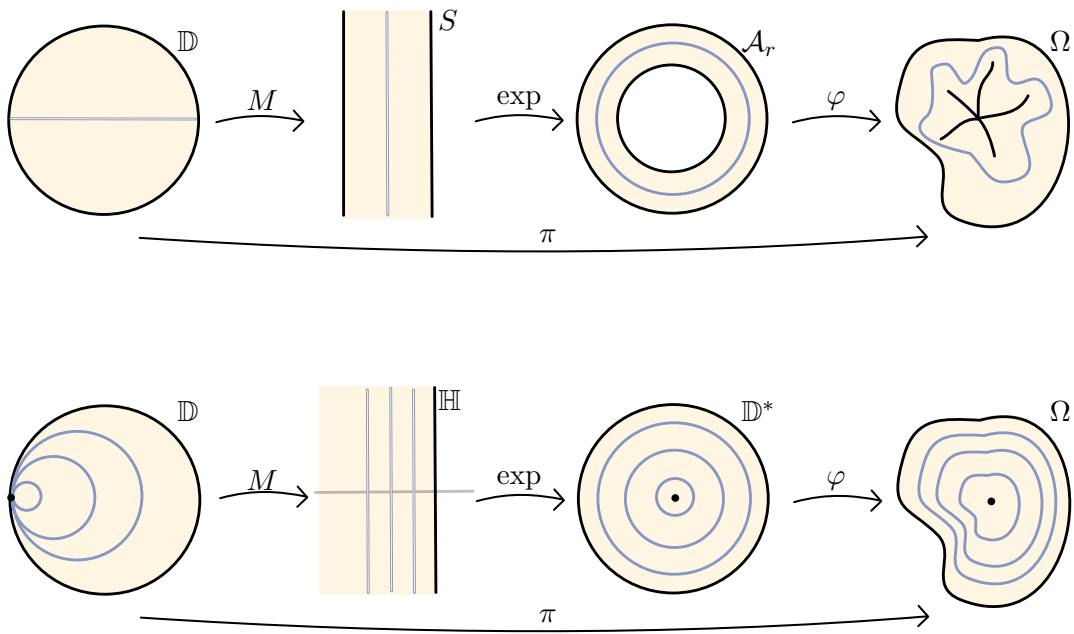

In both cases, one can split the unit disk $ \mathbb{D} $ into countably many simply connected domains, each of them containing exactly one preimage of each point in $ \Omega $. These regions are called {\em fundamental domains}, and $ \pi $ is univalent on these domains, so the previous results on the Riemann map apply. 

Prime ends thus can be defined in each fundamental domain following the same approach as in the simply connected case, with the following considerations. A crosscut $ C $ in $ \Omega $ must have both endpoints in the same connected component, and hence exactly one of the components of $  \Omega \smallsetminus C$ is simply connected. Then, the preimage of a null-chain in $ \Omega $ is a countable union of null-chains in $ \mathbb{D} $, which are contained in disjoint fundamental domains. Therefore, $ \pi $ is univalent on their correspondent crosscut neighbourhoods. 

Another way, maybe easier, of dealing with the same problem, is to consider $ \varphi $ acting around each of both boundary components separatedly, and analyze it as if it were a Riemann map (see Fig. \ref{fig-prime-end-annulus}).

The generalization of the concept of access is more subte, and we postpone it until Section \ref{ssect:accesses}, when dealing with general multiply connected domains.

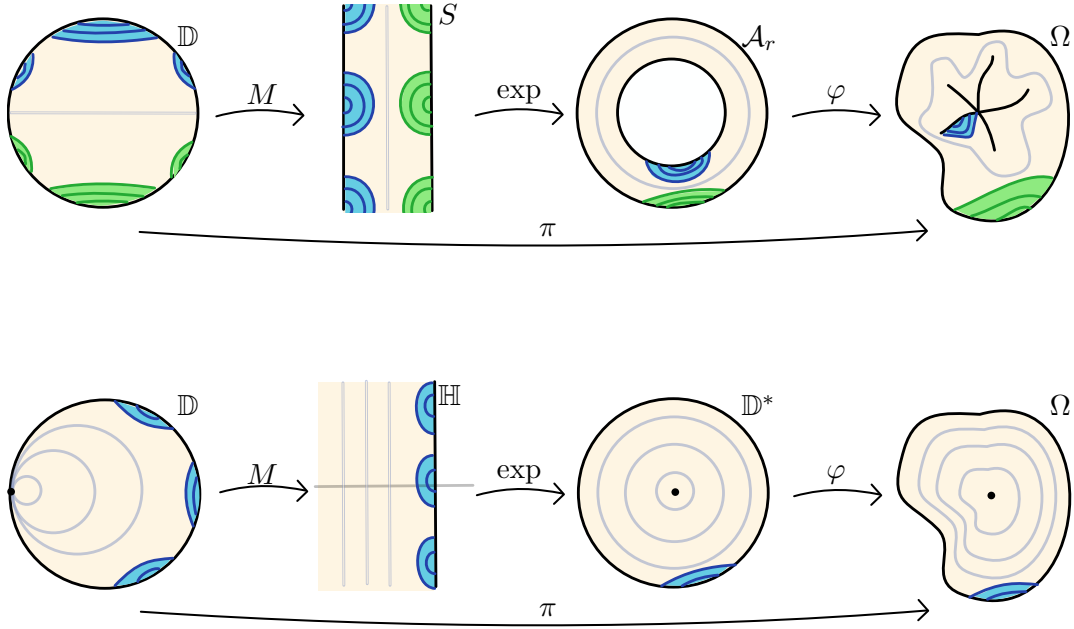
\begin{figure}[h]
	\centering
	\includegraphics[width=14cm]{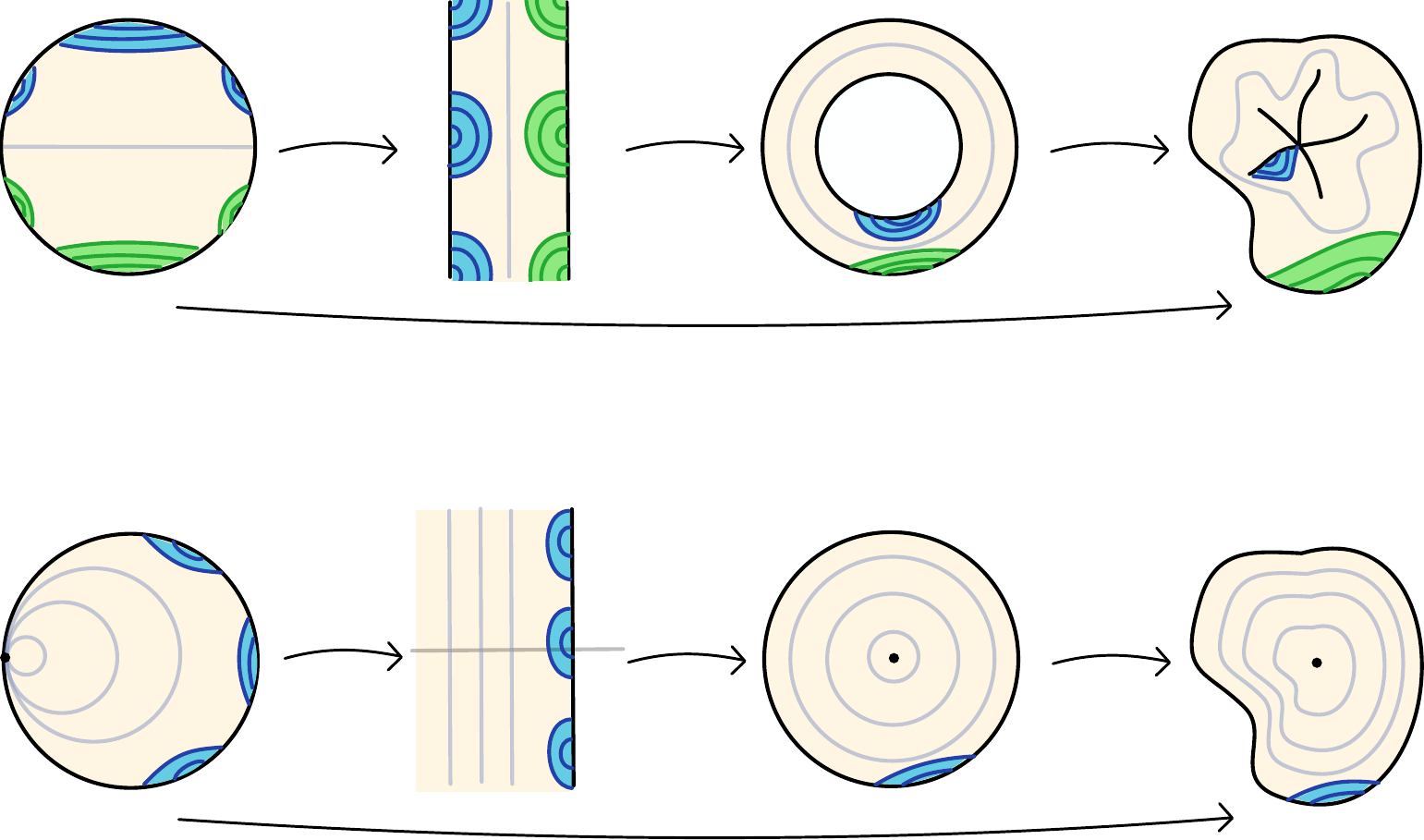}
	\setlength{\unitlength}{14cm}
	\put(-0.84, 0.55){$\mathbb{D}$}
	\put(-0.595, 0.57){$S$}
	\put(-0.31, 0.55){$\mathcal{A}_r$}
	\put(-0.02, 0.55){$\Omega$}
	\put(-0.775, 0.495){$M$}
	\put(-0.54, 0.5){$\exp$}
	\put(-0.23, 0.5){$\varphi$}
	\put(-0.5, 0.37){$\pi$}
	\put(-0.84,0.2){$\mathbb{D}$}
	\put(-0.595, 0.21){$\mathbb{H}$}
	\put(-0.31, 0.2){$\mathbb{D}^*$}
	\put(-0.02, 0.2){$\Omega$}
	\put(-0.775, 0.135){$M$}
	\put(-0.54, 0.14){$\exp$}
	\put(-0.23, 0.14){$\varphi$}
	\put(-0.5, 0.01){$\pi$}
	\caption{\footnotesize Null-chains and impressions of prime ends for a doubly connected domain, conformally isomorphic to a non-denegenerate annulus or a punctured disk, respectively.}\label{fig-prime-end-annulus}
\end{figure}
\subsection{Boundary behaviour of the universal covering of a multiply connected domain}\label{subsect-multiply-connected-domains}

Now we collect the known results about the boundary behaviour of the universal covering of a multiply connected domain of connectivity greater than two. We start with a technical result on how non-contractible curves are lifted by the universal covering, which will become a fundamental tool in the subsequent sections. Next, we sketch the work done by Ohtsuka in \cite{Oht54}.
Finally, we state the Hopf-Tsuji-Sullivan Ergodic Theorem.

\subsection*{Lifts of non-contractible curves in multiply connected domains}
The following is a technical lemma collecting results on the relation between non-contractible curves in $ \Omega $ and the universal covering $ \pi $.
\begin{lemma}{\bf (Lifts of non-contractible curves)}\label{lemma-lifts-of-curves}
	Let $ \Omega \subset \widehat{\mathbb{C}}$ be a multiply connected domain of connectivity greater than two, and let $ \pi\colon\mathbb{D}\to \Omega $ be a universal covering. Let $ \eta $ be a non-contractible Jordan curve in $ \Omega $, and let $ D_\eta $ be a connected component of $ \Omega\smallsetminus\eta $. Exactly one of the following holds. \begin{enumerate}[label={\em (\alph*)}] 
			\item  $D_\eta $ is conformally isomorphic to $ \mathbb{D}^* $, and the following hold.
		\begin{enumerate}[label={\em a.\arabic*}]
			\item $ \pi^{-1}(\eta) $ consists of  countably many disjoint degenerate crosscuts in $ \mathbb{D} $, whose endpoints are parabolic fixed points. Any two crosscuts in $ \pi^{-1}(\eta) $ have no  common endpoint.
			\item $ \pi^{-1}(D_\eta) $ consists of countably many simply connected domains, each of which touches $\partial\DD$ at exactly one point. Each connected component of $ \pi^{-1}(D_\eta) $ is bounded by one crosscut in $ \pi^{-1}(\eta) $.
			\item There exist curves $\eta^*\subset\Omega$ homotopic to $\eta$, such that $ \pi^{-1}(D_{\eta^*}) $ are horodisks.
		\end{enumerate}
		\item  $ D_\eta $ is conformally isomorphic to a non-degenerate annulus, and the following hold.
		\begin{enumerate}[label={\em b.\arabic*}]
			\item There exists a unique closed geodesic $\eta^*\subset\Omega$ homotopic to $\eta$.
				\item $ \pi^{-1}(\eta) $ consists of  countably many disjoint non-degenerate crosscuts in $ \mathbb{D} $, whose endpoints are hyperbolic fixed points. Any two crosscuts in $ \pi^{-1}(\eta) $ have no endpoints in common.
			\item $ \pi^{-1}(D_\eta) $ consists of countably many simply connected domains, each of which is unbounded in $ \mathbb{D} $. Each connected component of $ \pi^{-1}(D_\eta) $ is bounded by one crosscut in $ \pi^{-1}(\eta) $, i.e. it is a crosscut neighbourhood for that crosscut.
		\end{enumerate}
		\item Otherwise, $ D_\eta $ has connectivity greater than two, and the following hold.
				\begin{enumerate}[label={\em c.\arabic*}]
			\item There exists a unique closed geodesic $\eta^*\subset\Omega$ homotopic to $\eta$.
			\item $ \pi^{-1}(\eta) $ consists of  countably many disjoint non-degenerate crosscuts in $ \mathbb{D} $, whose endpoints are hyperbolic fixed points. Any two crosscuts in $ \pi^{-1}(\eta) $ have no endpoints in common.
			\item $ \pi^{-1}(D_\eta) $ consists of countably many simply connected domains, each of which is unbounded in $ \mathbb{D} $. Each connected component of $ \pi^{-1}(D_\eta) $ is bounded by countably many crosscuts in $ \pi^{-1}(\eta) $.
			\item For any crosscut in $ \pi^{-1}(\eta) $, each of its crosscut neighbourhoods contain infinitely many crosscuts in $ \pi^{-1}(\eta) $.
		\end{enumerate}
	\end{enumerate}

\noindent Moreover, let $ \eta_0 $ be another non-contractible curve in $ \Omega $. If $ \eta_0 $ is homotopic to $ \eta $ in $ \Omega $,  then for each crosscut in $ \pi^{-1}(\eta) $, there is a crosscut in $ \pi^{-1}(\eta_0) $  with the same endpoints in $ \partial\mathbb{D} $, and vice versa. If $ \eta_0 $ is not homotopic to $ \eta $ in $ \Omega $, then any two crosscuts in $ \pi^{-1}(\eta) $ and in $ \pi^{-1}(\eta_0) $ have no endpoints in common.
\end{lemma}

\begin{proof}
A proof of (a) can be found in \cite[Thm. 1.8.7]{Aba23}, while property (b) follows from \cite[Thm. 1.8.9]{Aba23} (although the results are stated for {\em regular type} domains, they apply to isolated boundary components). Statement (c) follows from the content of  \cite[Sect. 3]{Oht54}. The final statement  is the combination of \cite[Lemma 1]{Oht54} and \cite[Lemma 1.6.5]{Bus92}.
\end{proof}

Figure \ref{fig-lemma-lifts} gives a sketch of the different cases in Lemma \ref{lemma-lifts-of-curves}.

\begin{figure}[h]
	\centering
	\captionsetup[subfigure]{labelformat=empty, justification=centering}
	\begin{subfigure}{0.9\textwidth}
		\includegraphics[width=\textwidth]{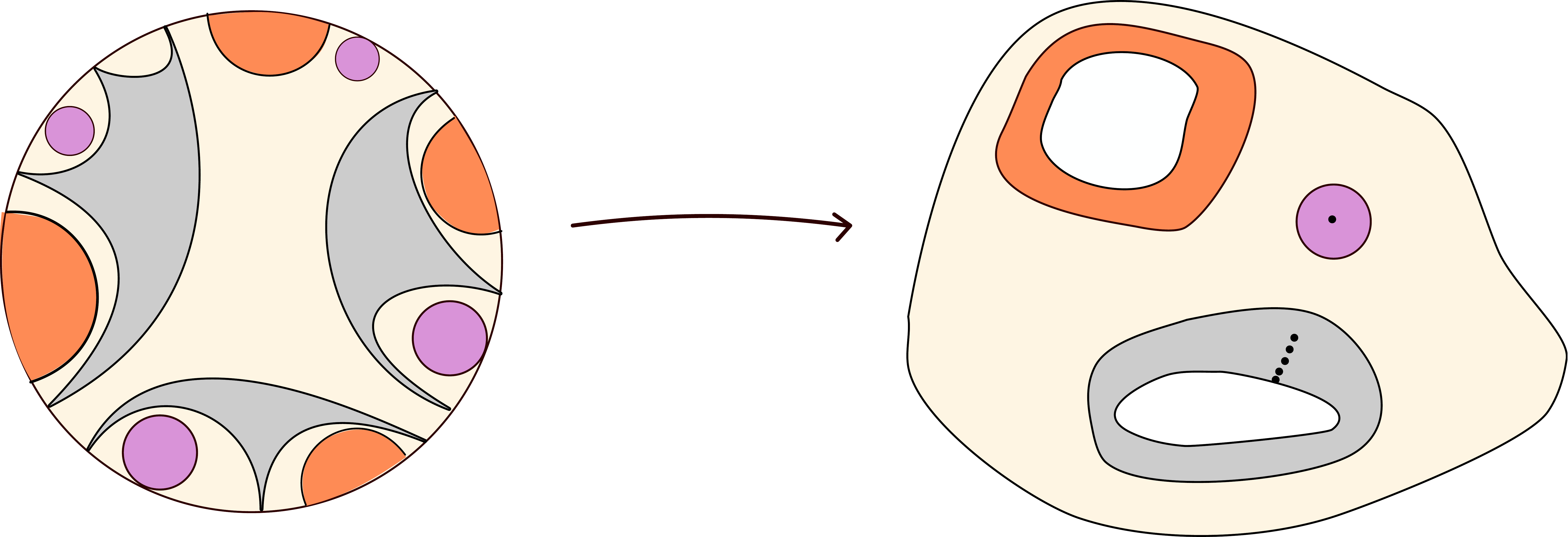}
		\setlength{\unitlength}{\textwidth}
					\put(-0.56, 0.21){$ \pi $}
		\put(-0.75, 0.31){$ \mathbb{D} $}
		\put(-0.14, 0.3){$ \Omega $}
			\caption{\footnotesize Schematic representation of the lift of  non-contractible curves and the domains bounded by such curves, exhibiting all the possibilities in Lemma \ref{lemma-lifts-of-curves}.}
	\end{subfigure}

	\vspace{0.5cm}
	
	\begin{subfigure}{0.9\textwidth}
		\includegraphics[width=\textwidth]{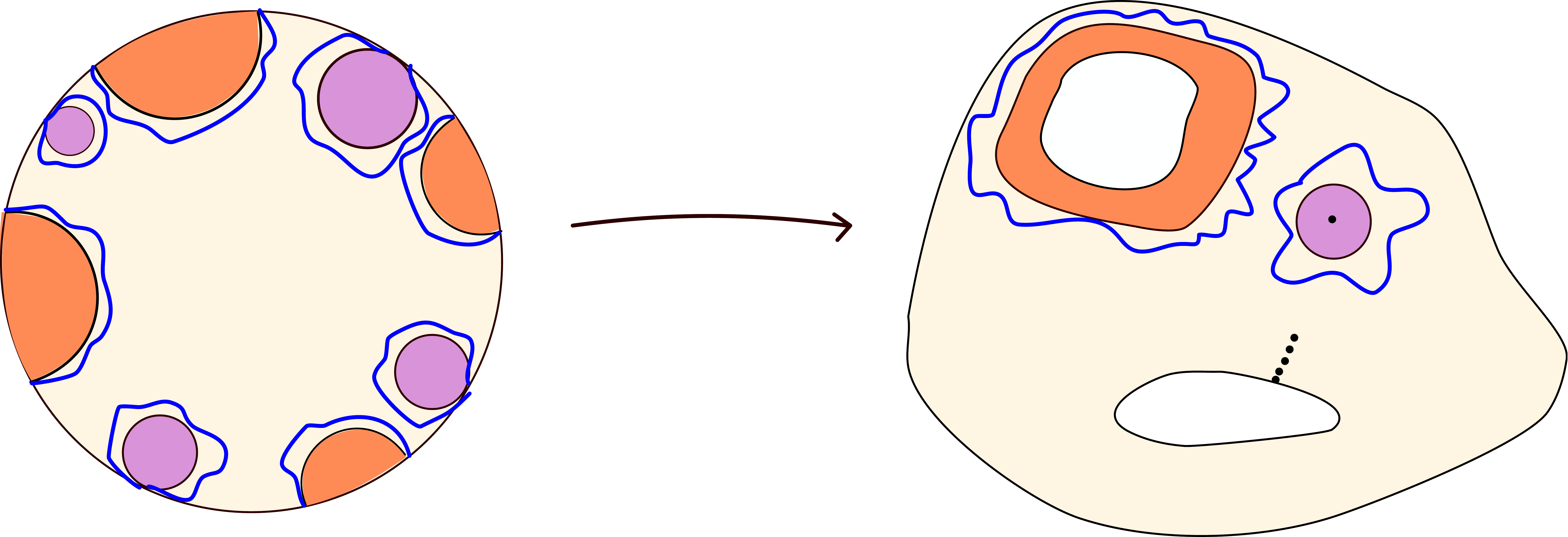}
			\setlength{\unitlength}{\textwidth}
			\put(-0.56, 0.21){$ \pi $}
					\put(-0.75, 0.31){$ \mathbb{D} $}
							\put(-0.14, 0.3){$ \Omega $}
			\caption{\footnotesize Given a non-contractible curve $ \eta $, one shall choose $ \eta^* $ in its homotopy class so that $ \pi^{-1}(\eta^*) $ is either a geodesic or it is a horocycle.  }
	\end{subfigure}
	\caption{\footnotesize Schematic representation of Lemma \ref{lemma-lifts-of-curves}.}\label{fig-lemma-lifts}
\end{figure}

\subsection*{Ohtsuka's exhaustion of multiply connected domains}

Ohtsuka \cite{Oht54} described the boundary of a multiply connected domain by means of exhausting it by finitely connected subdomains and identifying each boundary component symbolically. In this subsection, we give a brief overview of his methods and main results.

Let $ \Omega\subset\widehat{\mathbb{C}} $ be a multiply connected domain, which we assume to be of connectivity greater than two. Let $ \left\lbrace \Omega_n\right\rbrace _{n\geq 1}$  be an exhaustion of $ \Omega $, such that each $ \Omega_n $ is a finitely connected domain compactly contained in $ \Omega $, with $ \overline{\Omega_n} \subset \Omega_{n+1}$ for all $ n\geq 1 $, and $ \bigcup_n \Omega_n =\Omega $. Without loss of generality, we can assume that each $ \Omega_n $ is bounded by (finitely many) simple closed analytic curves, each of them dividing $ \Omega $ into two non-simply connected domains.

We label the boundary curves of the $ {\Omega_n}$'s as follows.  First, let $ \sigma_1, \dots,\sigma_{m} $ be the boundary curves of $ \Omega_1 $, and let $ D_i $, $ i=1,\dots, m $, be the connected component of $ \Omega\smallsetminus \Omega_1 $ bounded by $ \sigma_i $. Then, let the boundary curves  of $ \Omega_2 $ lying in $ D_i $ be denoted by $ \sigma_{i1}, \dots,\sigma_{im_i} $, and denote by $ D_{ij} $ the domain  in $ \Omega\smallsetminus \Omega_2 $ bounded by $ \sigma_{ij}$. The same way, the boundary curves of $ \Omega_3 $  in $ D_{ij} $ are denoted by $ \sigma_{ij1},\dots,  \sigma_{ijm_{ij}} $.

Proceeding inductively, let $ D_{s_1s_2\dots s_{n}} $ be the domain bounded by the curve $ \sigma_{s_1s_2\dots s_{n}}\subset\partial \Omega_n
 $. Then, the boundary curves  of $ \Omega_{n+1} $ lying in $ D_ {s_1s_2\dots s_{n}}$ are labeled as  $ \sigma_{s_1s_2\dots s_{n}s_{n+1}} $, with $ s_{n+1}\in \left\lbrace 1, \dots, m_{\underline{s}^{n}} \right\rbrace  $, where $ \underline{s}^{n} $ denotes the truncated sequence $ s_1s_2\dots s_{n} $. Then,  $ D_{s_1s_2\dots s_{n}}\supset  D_{s_1s_2\dots s_ns_{n+1}}$, and $  D_{s_1s_2\dots s_{n}} $ is the domain with relative boundary $ \sigma_ {s_1s_2\dots s_{n}}$ in $ \Omega $. See Figure \ref{fig-exhaustion}.

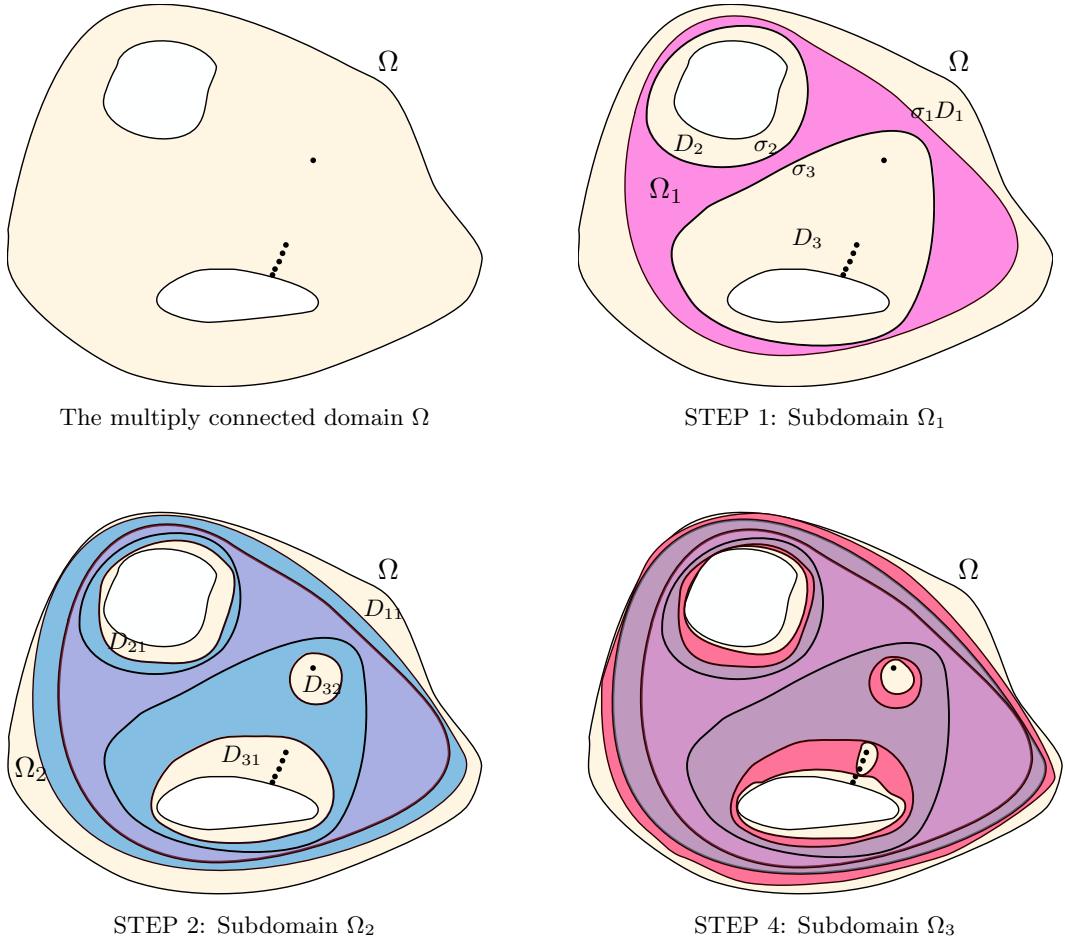
\begin{figure}[h]
	\centering
	\captionsetup[subfigure]{labelformat=empty, justification=centering}

	\begin{subfigure}{0.45\textwidth}
		\includegraphics[width=\textwidth]{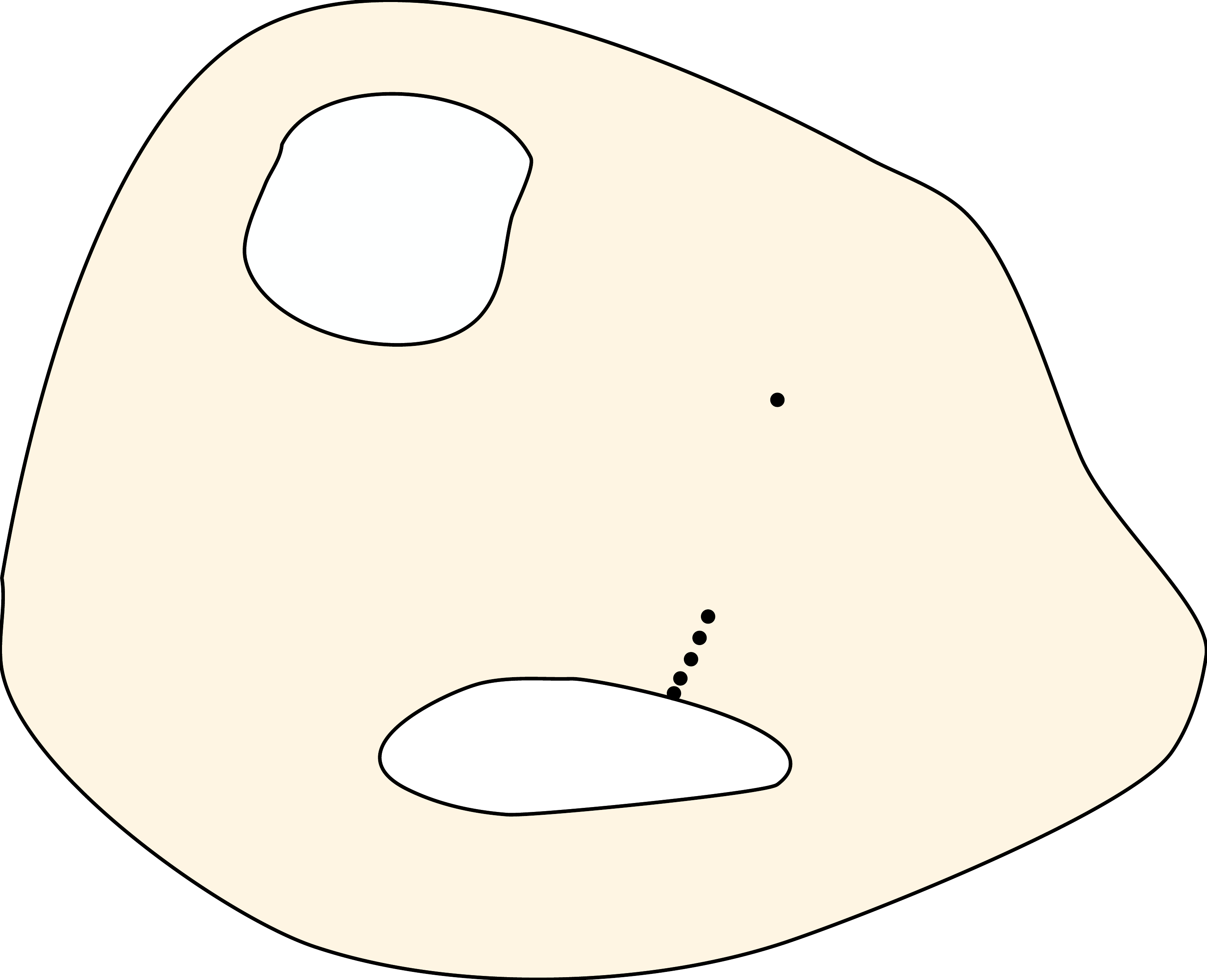}
		\setlength{\unitlength}{\textwidth}
		\put(-0.22,0.67){$ \Omega $}
		\caption{\footnotesize The multiply connected domain $ \Omega $}
	\end{subfigure}
	\hfill
	\begin{subfigure}{0.45\textwidth}
		\includegraphics[width=\textwidth]{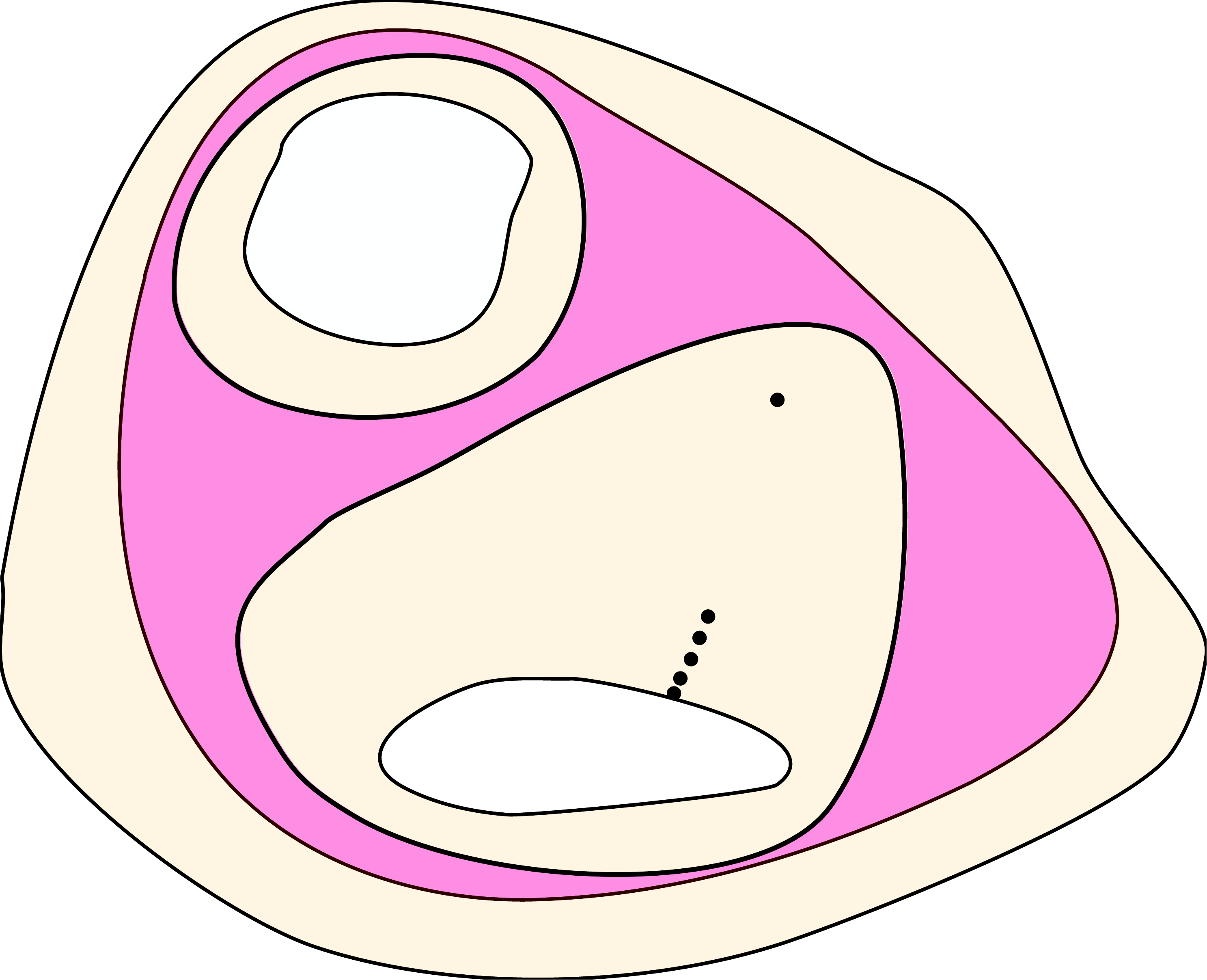}
		\setlength{\unitlength}{\textwidth}
				\put(-0.22,0.67){$ \Omega $}
					\put(-0.85,0.4){$ \Omega_1$}
				\put(-0.25, 0.57){\footnotesize $ D_1 $}
					\put(-0.8, 0.5){\footnotesize $ D_2 $}
						\put(-0.55, 0.3){\footnotesize $ D_3 $}
								\put(-0.3, 0.57){\footnotesize $ \sigma_1 $}
						\put(-0.63, 0.5){\footnotesize $ \sigma_2 $}
						\put(-0.55, 0.45){\footnotesize $ \sigma_3 $}
		\caption{\footnotesize STEP 1: Subdomain $ \Omega_1 $}
	\end{subfigure}
	\vspace{1cm}

	\begin{subfigure}{0.45\textwidth}
		\includegraphics[width=\textwidth]{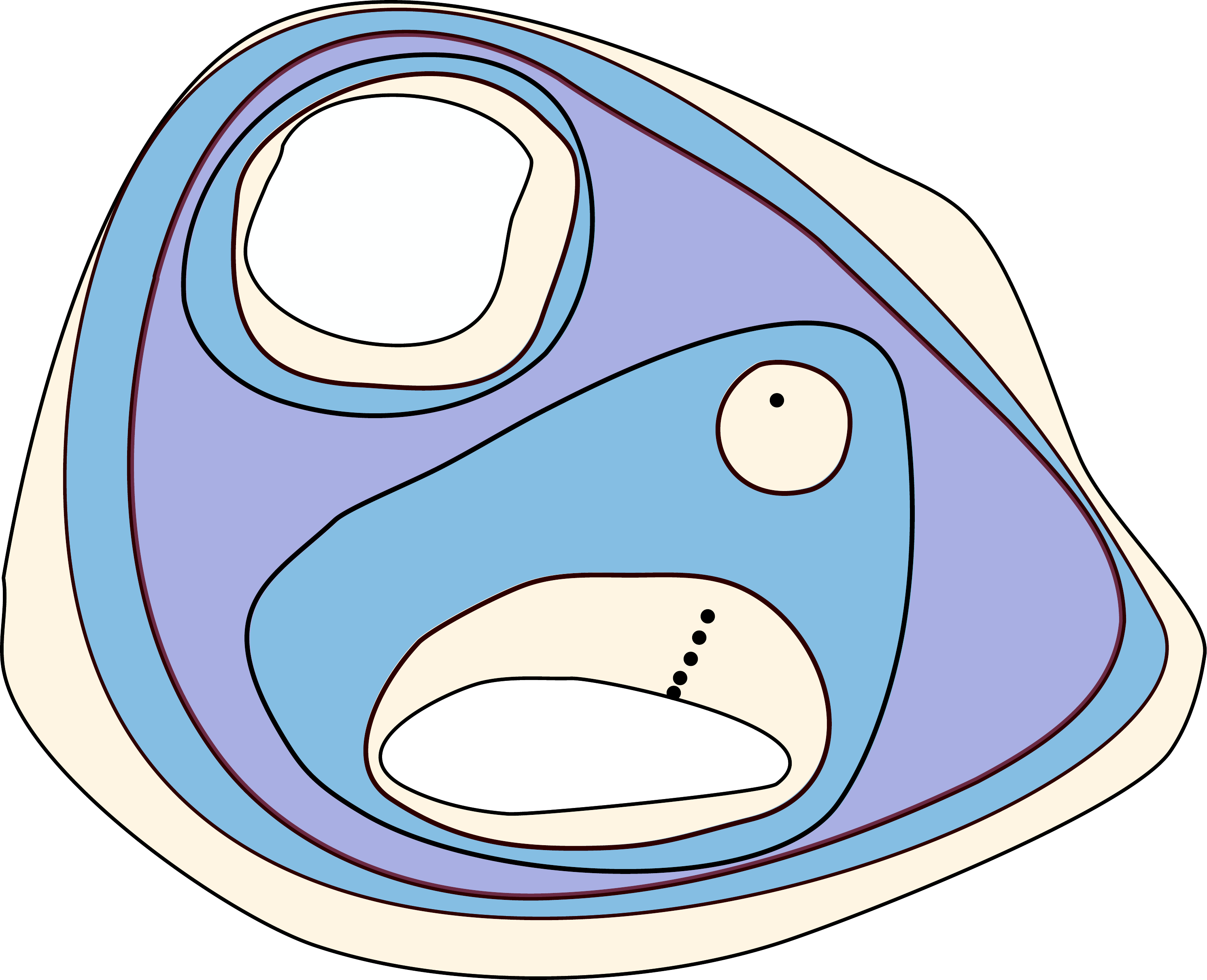}
		\setlength{\unitlength}{\textwidth}
					\put(-0.22,0.67){$ \Omega $}
		\put(-0.985,0.25){$ \Omega_{2}$}
		\put(-0.25, 0.59){\footnotesize $ D_{11}$}
		\put(-0.79, 0.52){\footnotesize $ D_{21} $}
		\put(-0.55, 0.28){\footnotesize $ D_{31} $}
			\put(-0.38, 0.43){\footnotesize $ D_{32} $}
		\caption{\footnotesize STEP 2: Subdomain $ \Omega_2 $}
	\end{subfigure}
	\hfill
	\begin{subfigure}{0.45\textwidth}
		\includegraphics[width=\textwidth]{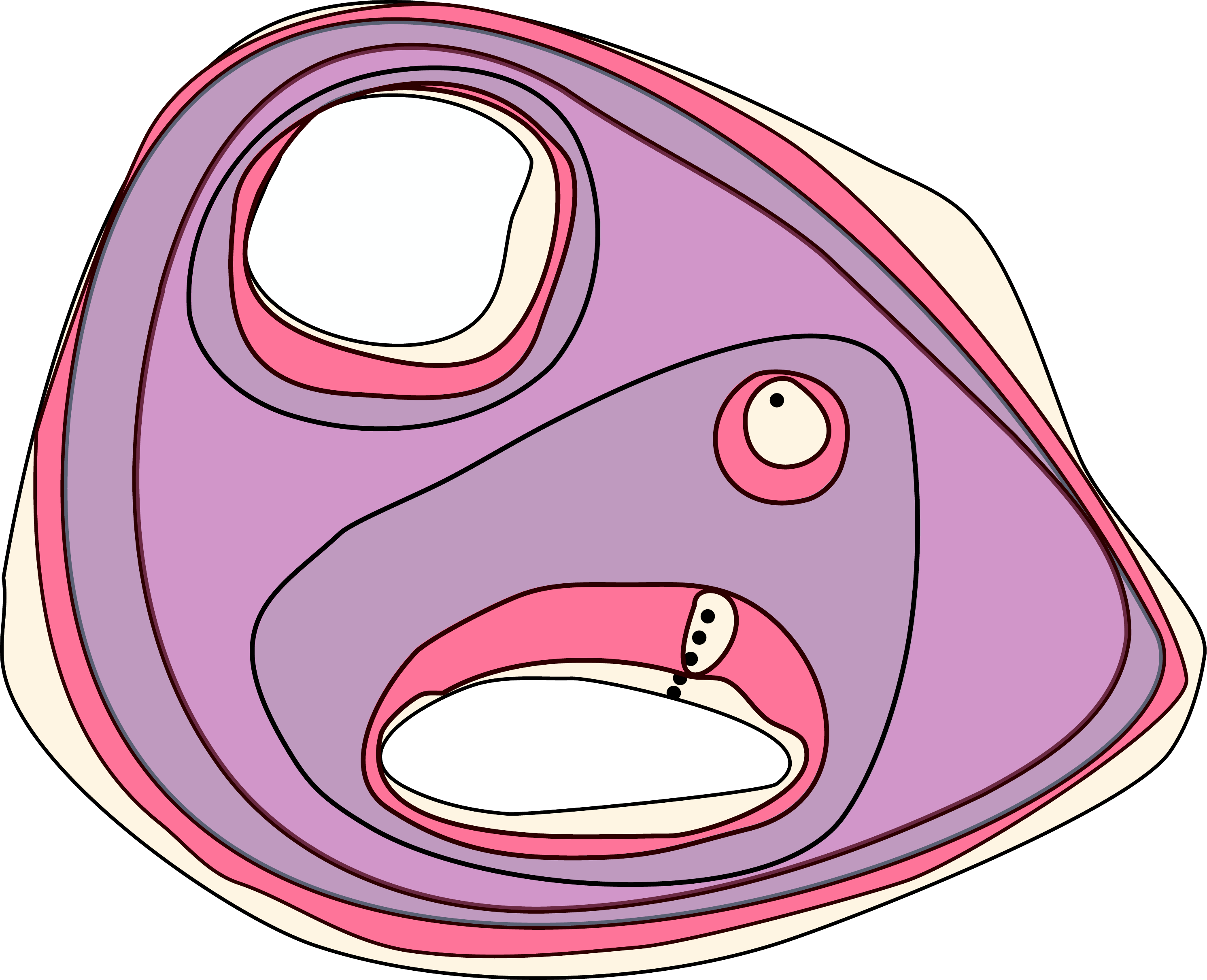}
		\setlength{\unitlength}{\textwidth}
			\put(-0.22,0.67){$ \Omega $}
		\caption{\footnotesize STEP 4: Subdomain $ \Omega_3 $}
	\end{subfigure}
	\caption{\footnotesize The first steps in the construction of Ohtsuka exhaustion for a multiply connected domain $ \Omega $.}\label{fig-exhaustion}
\end{figure}

An exhaustion of $ \Omega $ satisfying the properties above, and labeled as described, will be called an {\em Ohtsuka exhaustion} of $ \Omega$. It allows us to identify each boundary component of $\Omega$ as follows.

\begin{defi}{\bf (Boundary address)}\label{def:boundaryaddress}
Let $ \Omega $ be a multiply connected domain, and let $ \pi\colon\mathbb{D}\to\Omega $ be a universal covering. Let $ \left\lbrace \Omega_n\right\rbrace _n $ be an Ohtsuka exhaution of $ \Omega $. Then, for each boundary component $ \Sigma\subset \partial\Omega $, we say that the sequence
\[\underline{s}(\Sigma)=\left\lbrace s_n\right\rbrace _n, \hspace{0.5cm} s_n\in\left\lbrace 1, 2, \dots, m_{\underline{s}^{n-1}} \right\rbrace, \] is the {\em boundary address} of $ \Sigma $ if \[\Sigma =\bigcap_{n\geq 1} \overline{D_{s_1s_2\dots s_n}} .\]
\end{defi}

It is clear from the construction that, given an Ohtsuka exhaustion of $\Omega$, any boundary component $ \Sigma\subset\partial\Omega $ has a unique boundary address. Conversely, any  address $\underline{s}=\left\lbrace s_n\right\rbrace _n$ with $s_n\in\{1, \dots, m_{\underline{s}^{n-1}}\}$ is the boundary address of a unique boundary component $\Sigma\subset\partial\Omega$.

By considering nested preimages of the sequence $\{D_{s_1s_2\dots s_n}\}_n$ under the universal covering, we can identify the points on $ \partial\mathbb{D} $ that correspond to  $\Sigma$ under $ \pi $ as follows.

\begin{defi}{\bf ($ \alpha $-image)}
Let $ \Omega $ be a multiply connected domain, and let $ \pi\colon\mathbb{D}\to\Omega $ be a universal covering. Let $ \left\lbrace \Omega_n\right\rbrace _n $ be an Ohtsuka exhaution of $ \Omega $, and let $ \Sigma $ be a connected component of $ \partial \Omega $, with address $ \underline s(\Sigma) = \left\lbrace s_n\right\rbrace _n$, and determined by the domains $ \left\lbrace D_{s_1s_2\dots s_n}\right\rbrace _n  $. A connected component of $$\bigcap_{n} \overline{\pi^{-1}(D_{s_1s_2\dots s_n})},$$ minus any fixed points of hyperbolic deck transformations\footnote{Ohtsuka's definition of $\alpha$-image allows for such fixed points to be part of the $\alpha$-image. However, as we will see when relating $\alpha$-images to the limit sets of deck transformations, it is more natural for our purposes to exclude them.}, is called an {\em $\alpha$-image} of $\Sigma$.
\end{defi}

Notice that an $\alpha$-image is the intersection (possibly minus hyperbolic fixed points) of nested, connected compact sets shrinking towards $\partial\DD$, and thus is either a point or an arc (open or closed) on $\partial\DD$. Ohtsuka proved that an $\alpha$-image $\tilde\Sigma$ of the boundary component $\Sigma$ is indeed, in some sense, a preimage of $\Sigma$ under $\pi$.
\begin{thm}{\bf($ \alpha $-images, {\normalfont \cite[Sect. I.5]{Oht54}})}\label{thm:alphaimage}
Let $\tilde\Sigma$ be an $\alpha$-image of the boundary component $\Sigma\subset\partial\Omega$. Then, there exists a curve $\eta\subset\DD$ with $L(\eta)\subset\tilde\Sigma$ such that $L(\pi(\eta))\subset\Sigma$. Conversely, any lift $\tilde\eta\subset\DD$ of a curve $\eta\subset\Omega$ such that $L(\eta)\subset\Sigma$ satisfies $L(\tilde\eta)\subset \tilde\Sigma$ for some $\alpha$-image $\tilde\Sigma$ of $\Sigma$.
\end{thm}

We will prove a stronger version of Theorem \ref{thm:alphaimage} in Section \ref{ssec:crosscuts}.

 Notice that not every point in $\partial\DD$ is part of an $\alpha$-image. Indeed, let $ \Omega_m $ be a subdomain in Ohtsuka exhaustion, so that $ \sigma_{s_1s_2\dots s_m} \subset \partial  \Omega_m$ is a closed geodesic and $ D_{s_1s_2\dots s_m} $ has connectivity greater than two. Let $e^{i\theta}\in\partial \mathbb{D}$ be the landing point of some geodesic $ \widetilde{\sigma} $ in $\pi^{-1}(\sigma_{s_1s_2\dots s_m})$. We claim that such an $e^{i\theta}$ is not in any $ \alpha $-image. If this was the case, for any boundary address $\underline{s} = \left\lbrace s_n\right\rbrace $ giving rise to $e^{i\theta}$ as a point in one of its $\alpha$-images we would have $$e^{i\theta}\in \bigcap_{n} \overline{\pi^{-1}(D_{s_1s_2\dots s_n})}.$$ According to Lemma \ref{lemma-lifts-of-curves}, since endpoints of lifts of non-homotopic closed curves are distinct, for each $ D_{s_1s_2\dots s_n} $ there is a crosscut in $ \partial \pi^{-1}(D_{s_1s_2\dots s_n}) $ separating 0 from $ e^{i\theta} $, and hence crossing $ \widetilde{\sigma} $. Thus,
  $$D_{s_1s_2\dots s_n}\cap \sigma_{s_1s_2\dots s_m}  \neq\emptyset$$ for all but finitely many values of $n$. This is a contradiction because the domains $  D_{s_1s_2\dots s_n}$ lie in $ \Omega\smallsetminus\Omega_n $ and $ \left\lbrace \Omega_n\right\rbrace _n $ exhaust $ \Omega $, while $ \sigma_{s_1s_2\dots s_m} $ is compactly contained in $ \Omega $.

  Ohtsuka went further, proving the existence of plenty of more points in $ \partial\mathbb{D} $ that do not belong to $\alpha$-images, and characterizing and   classifying all of them.
\begin{thm}{\bf ({\normalfont \cite[Thm. 5]{Oht54}})}\label{thm:notalphaimage}
If a point $e^{i\theta}$ is not contained in any $\alpha$-image of any boundary component of $\Omega$, then one of the following holds.
\begin{enumerate}[label={\em(\alph*)}]
	\item There exists $n\in\mathbb{N}$ such that $e^{i\theta}$ belongs to the boundary of a connected component of $\pi^{-1}(\Omega_n)$. There are uncountably many such points in $\partial\DD$.
	\item The image of any curve $\eta\colon[0, 1)\to\DD$ landing at $e^{i\theta}$ is neither compactly contained in $\Omega$ nor tends to $\partial\Omega$ as $t\to 1^{-}$. Equivalently, if one denotes by $\{\tilde\sigma_j\}_j$ all the curves in $\pi^{-1}(\sigma_{s_1s_2\dots s_n})$ separating $0$ from $e^{i\theta}$, enumerating them such that $\tilde\sigma_n$ separates $0$ from $\tilde\sigma_{n+1}$, then the curves $\pi(\tilde\sigma_n)$ are neither compactly contained in $\Omega$ nor converging to $\partial\Omega$ as $n\to+\infty$. 
	
	\noindent Such points exist if and only if $\Omega$ is infinitely connected, and in this case, there are uncountably many of them.
\end{enumerate}
\end{thm}
In Section \ref{ssec:Ohtsuka2}, we will see an alternative way of formulating Theorem \ref{thm:notalphaimage} by classifying points in $\partial\DD$ according to the image of their radial segment under $ \pi $. This classification will aid us in defining and exploring prime ends in multiply connected domains.

\subsection*{The Hopf-Tsuji-Sullivan Ergodic Theorem}
Originally proved by Hopf as a way to characterise hyperbolic surfaces with an ergodic geodesic flow \cite{Hop36}, this theorem has grown into a powerful set of equivalent conditions on Riemann surfaces -- many of them due to Tsuji and Sullivan, as the theorem's name suggests, but some were originally due to Ahlfors, and others (which are specific to plane domains) due to Nevanlinna. We state here only the equivalences that concern us; see \cite{Fer89} for a more comprehensive list. Proofs of most equivalences can be found in \cite[Chap. 11]{Tsu75}, or \cite[Chap. VII.5]{Nev70} for equivalences specific to plane domains.
\begin{thm}{\bf (Hopf-Tsuji-Sullivan Ergodic Theorem)}\label{thm-hopf-tsuji-sullivan}
Let $\Omega$ be a hyperbolic plane domain, and let $\pi\colon\DD\to\Omega$ be a universal covering with deck transformation group $\Gamma$. Then, the following are equivalent.
\begin{enumerate}[label={\normalfont (\alph*)}]
	\item $\Lambda_{NT}$ has full Lebesgue measure in $\DD$.
	\item $\partial\Omega$ has zero logarithmic capacity as a set in $ \widehat{\mathbb{C}} $.
	\item $\partial\Omega$ does not support a harmonic measure.
	\item $\pi$ does not belong to the Nevanlinna class, i.e.
	\[ \sup_{0 < r < 1} \int_0^{2\pi} \log^+\pi|(re^{i\theta})|\,d\theta = +\infty. \]
\end{enumerate}
\end{thm}

Notice that, by another theorem of Nevanlinna, belonging to the Nevanlinna class implies the existence of radial limits of $\pi$ almost everywhere (see e.g. \cite[Thm. 20.2.11]{Conway2}).

\section{The universal covering from different perspectives}\label{sect-main-result}
Let $ \Omega $ be a multiply connected domain of connectivity greater than two, and consider the universal covering\[\pi\colon\mathbb{D}\longrightarrow \Omega.\]

\subsection{The universal covering acting on radii}
For the Riemann map of a simply connected domain, the image of a curve tending to $\partial\DD$ tends to the boundary of the domain. However, that is no longer the case for multiply connected domains. In particular, having a sequence $ \left\lbrace z_n \right\rbrace _n\subset\mathbb{D}$ with $ z_n\to\partial\mathbb{D} $, no longer implies that $ \pi(z_n)\to\partial\Omega $; and there exist points in the unit circle (possibly all of them) for which the cluster set is the whole $ \overline{\Omega} $. This motivates the following definition. 
\begin{defi}{\bf (Boundary behaviour of $ \pi $)}\label{def:escaping-bounded-bungee}
	Let $ \Omega $ be a multiply connected domain,  let $ \pi\colon\mathbb{D}\to \Omega $ be the universal covering and let $ e^{i\theta}\in\partial\mathbb{D} $. Let $ R_\theta(t)=\left\lbrace te^{i\theta} \colon t\in \left[ 0,1\right)  \right\rbrace $.
	\begin{itemize}
		\item  We say that $ e^{i\theta}\in\partial \mathbb{D} $ is {\em of escaping type} if $ \pi(R_\theta (t))\to\partial \Omega $, as $ t\to 1^{-} $.
		\item We say that $ e^{i\theta}\in\partial \mathbb{D} $ is {\em of bounded type} if $\left\lbrace  \pi(R_\theta (t))\colon t\in \left[ 0,1\right) \right\rbrace  $ is compactly contained in $ \Omega $.
		\item Otherwise, we say that $ e^{i\theta}\in\partial \mathbb{D} $ is {\em of bungee type}.
	\end{itemize}
\end{defi}

First, we prove that there is no loss of generality in considering the radial segment $ R_\theta $ in the previous definition, in the sense that $ \pi $ behaves likewise along any other curve  landing non-tangentially at $ e^{i\theta}\in\partial \mathbb{D} $. Moreover, we describe the angular cluster set in each situation.
\begin{prop}{\bf (Boundary behaviour of non-tangential curves)}\label{prop-escaping-bounded-bungee}	Let $ \Omega $ be a multiply connected domain,  let $ \pi\colon\mathbb{D}\to \Omega $ be the universal covering and let $ e^{i\theta}\in\partial\mathbb{D} $. 
	Let $ \eta $ be a curve landing non-tangentially at $ e^{i\theta}$.
		\begin{itemize}
		\item  If $ e^{i\theta}\in\partial \mathbb{D} $ is of {escaping} type, then $ \pi(\eta (t))\to\partial \Omega $, as $ t\to 1^{-} $, and
		\[ Cl_{\mathcal{A}}(\pi, e^{i\theta})= Cl_{R}(\pi, e^{i\theta})\subset \partial \Omega. \] In particular, $ Cl_{\mathcal{A}}(\pi, e^{i\theta}) $ is closed and contained in a unique boundary component of $\Omega$.
		\item If  $ e^{i\theta}\in\partial \mathbb{D} $ is of bounded type, then $\left\lbrace  \pi(\eta (t))\colon t\in \left[ 0,1\right) \right\rbrace  $ is compactly contained in $ \Omega $. Moreover,
		\[ Cl_{\mathcal{A}}(\pi, e^{i\theta})=  \Omega. \] In particular, $ Cl_{\mathcal{A}}(\pi, e^{i\theta}) $ is open.
		\item If $ e^{i\theta}\in\partial \mathbb{D} $ is of {bungee} type, then $\left\lbrace  \pi(\eta (t))\colon t\in \left[ 0,1\right) \right\rbrace  $ is neither compactly contained in $ \Omega $ nor converging to $ \partial \Omega $.
	\end{itemize}
\end{prop}
\begin{proof}
	The proof relies strongly on the fact that the radius $ R_\theta $ and the curve $ \eta $ are contained in some hyperbolic Stolz angle $ \Delta_{R_\theta, r} (e^{i\theta}) $, for some $ r>0 $, the invariance of the hyperbolic density under the universal covering, and the fact that hyperbolic density tends to infinity as we approach $ \partial \Omega $. Indeed, for every $ \tilde z,\tilde w\in\mathbb{D} $,
	\[ \textrm{dist}_\Omega (\pi(\tilde z),\pi(\tilde w))\leq \textrm{dist}_\DD(\tilde z,\tilde w) ,\] and, for all $ z\in\Omega $, $  \dist_\Omega(z, \partial\Omega)=\infty$.

	First, let $ e^{i\theta} $ be of escaping type, so $ Cl_{R}(\pi, e^{i\theta})\subset \Sigma $ for some boundary component $ \Sigma \subset\partial \Omega  $. If $ \eta \subset \Delta_{R_\theta, r} (e^{i\theta}) $, we have that
	$$\textrm{dist}_\Omega(\pi (z), \pi(R_\theta))<r \text{ for all $z\in\eta$}. $$
	By Lemma \ref{lem:superSP}(e), it follows that $ Cl_{\eta}(\pi, e^{i\theta})=Cl_{R}(\pi, e^{i\theta}) \subset \Sigma  $.
	Therefore, $ Cl_{\mathcal{A}}(\pi, e^{i\theta})= Cl_{R}(\pi, e^{i\theta}) $, and since $ Cl_{R}(\pi, e^{i\theta}) $ is always closed, it follows that $ Cl_{\mathcal{A}}(\pi, e^{i\theta}) $ is also closed.

	Second, assume $ e^{i\theta} $ is bounded. For all $ r>0 $, consider the set
	\[\Omega_r \coloneqq \left\lbrace z\in\Omega\colon \textrm{dist}_\Omega (z, \pi(R_\theta))<r\right\rbrace, \] which is compactly contained in $ \Omega $. Then, if $ \eta\subset \Delta_{R_\theta, r} (e^{i\theta})   $, it follows that $ \pi (\eta)\subset\Omega_r $, so it is compactly contained in $ \Omega $. 
	
	Next we have to see that  $ Cl_{\mathcal{A}}(\pi, e^{i\theta})=  \Omega $. One inclusion is clear from the previous argument: if $ z\in Cl_{\mathcal{A}}(\pi, e^{i\theta})$, then $ z\in Cl_{\eta}(\pi, e^{i\theta})$ for some curve  $ \eta $ landing non-tangentially at $ e^{i\theta} $, so $ z\in\Omega $.
	
	To prove the other inclusion, let $w\in\Omega$. The fact that $e^{i\theta}$ is of bounded type implies, by definition, that $\pi(R_\theta)$ has some accumulation point $z^*\in Cl_R(\pi, e^{i\theta})\subset \Omega$, which is at some finite distance $r$ from $w$ (notice that $r$ depends on $w$). Next, by the definition of the radial cluster set, there exists a sequence of points $\tilde z_n\in R_\theta$ such that $\pi(\tilde z_n)\to z^*$ as $n\to\infty$, and in particular such that $\textrm{dist}_\Omega(\pi(\tilde z_n), z^*) < \delta$ for some fixed $\delta > 0$. By the triangle inequality,
	\[ \textrm{dist}_\Omega(\pi(\tilde z_n), w) \leq \textrm{dist}_\Omega(\pi(\tilde z_n), z^*) + \textrm{dist}_\Omega(z^*, w) < \delta + r \]
	for all $n\in\mathbb{N}$; in particular, by Lemma \ref{lem:superSP}(b), there exist points $\tilde w_n\in D_\DD(\tilde z_n, r + \delta)$ such that $\pi(\tilde w_n) = w$. Since $\tilde z_n\in R_\theta$, this implies that the points $\tilde w_n$ belong to the hyperbolic Stolz angle $\Delta_{R_\theta, r + \delta}(e^{i\theta})$. Hence, $w\in Cl_{\mathcal{A}}(\pi, e^{i\theta}))$. Since $w\in\Omega$ was arbitrary we have $$\Omega\subset Cl_{\mathcal{A}}(\pi, e^{i\theta}),$$ as desired.

	The third case is deduced arguing by contradiction. Indeed, the existence of a non-tangential curve $ \eta $ landing at $ e^{i\theta} $ whose image is compactly contained in $ \Omega $ would force the radius $ R_\theta $ to have the same behaviour, and the same if $ \pi(\eta) $ accumulates on the boundary. Hence, if $ e^{i\theta}\in\partial \mathbb{D} $ is of bungee type, then $\left\lbrace  \pi(\eta (t))\colon t\in \left[ 0,1\right) \right\rbrace  $ can neither be compactly contained in $ \Omega $ nor converge to $ \partial \Omega $. This ends the proof of the proposition; see Figure \ref{fig-bounded-escaping-bungee} for an illustration.
	\begin{figure}[h]\centering
		\includegraphics[width=14cm]{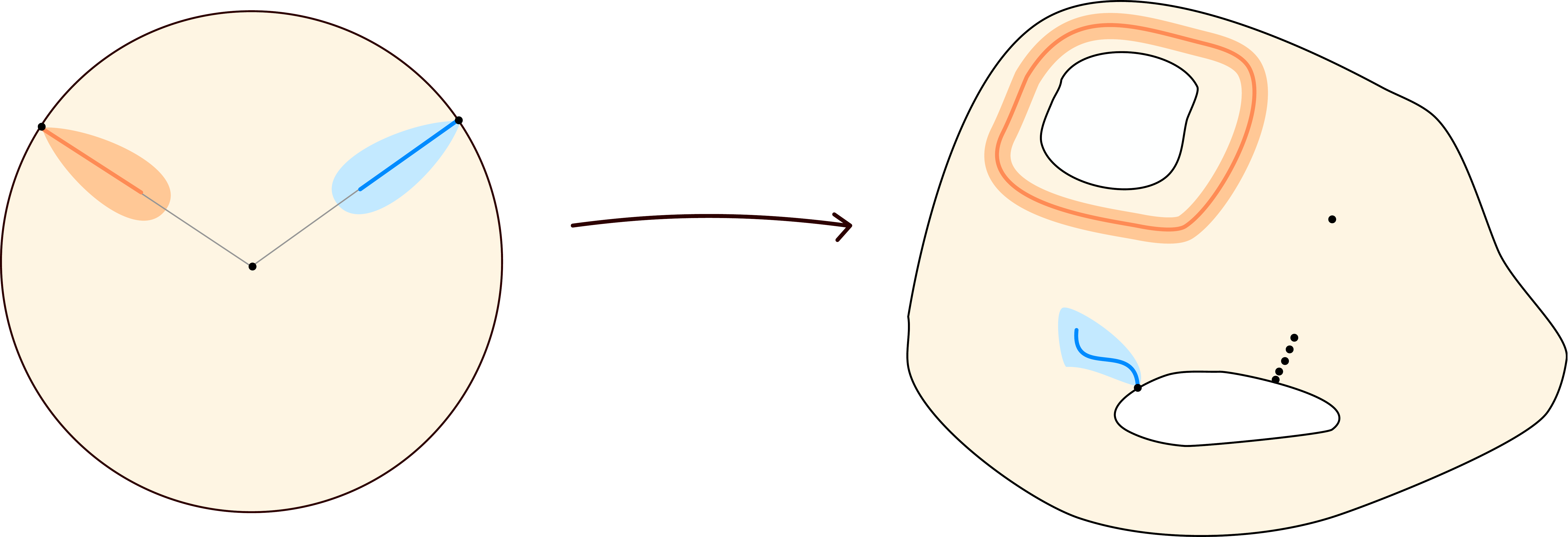}
		\setlength{\unitlength}{14cm}
						\put(-0.56, 0.21){$ \pi $}
		\put(-1, 0.31){$ \mathbb{D} $}
		\put(-0.14, 0.3){$ \Omega $}
		\put(-1.02, 0.25){$ e^{i\theta_1} $}
			\put(-0.7, 0.26){$ e^{i\theta_2} $}
				\put(-0.93, 0.23){$ R_{\theta_1} $}
			\put(-0.8, 0.24){$ R_{\theta_2} $}
				\put(-0.3, 0.17){$ \pi(R_{\theta_1}) $}
			\put(-0.4, 0.1){$ \pi(R_{\theta_2} )$}
		\caption{\footnotesize A geometric interpretation of Proposition \ref{prop-escaping-bounded-bungee}, where $ e^{i\theta_1} $  and $ e^{i\theta_2} $ are points of escaping  and bounded type, respectively, and radial segments and hyperbolic Stolz angles at these points, and their images under $ \pi $, are represented.}\label{fig-bounded-escaping-bungee}
	\end{figure}
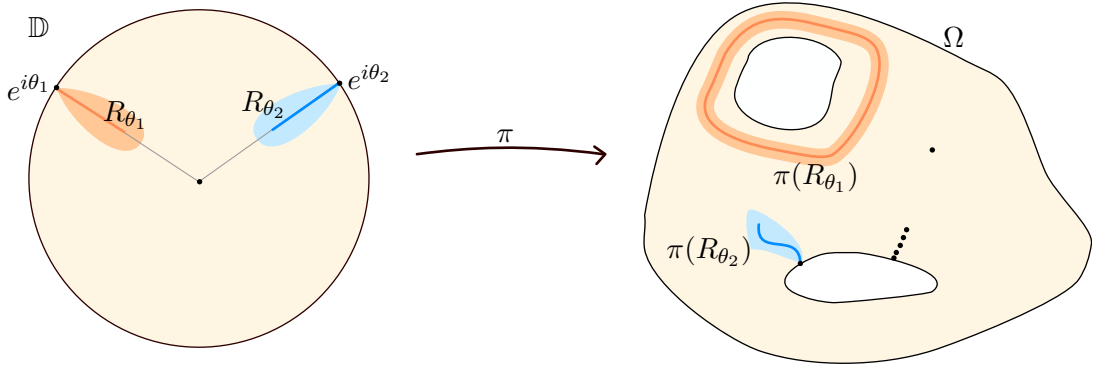
\end{proof}

As preparation for our main result, Theorem \ref{thm-main-result-complete}, we relate the behaviour of the image of curves approaching $e^{i\theta}$ to the limit sets of $\Gamma$, the group of deck transformations.
\begin{lemma}{\bf (Non-tangential curves and limit sets)}\label{lem:non-tangential-curves}
	Let $ \Omega $ be a multiply connected domain, let $ \pi\colon \mathbb{D}\to\Omega $ be a universal covering, and let $ e^{i\theta} \in\partial\mathbb{D}$. Then, the following hold.
	\begin{enumerate}[label={\em (\alph*)}]
		\item If $e^{i\theta}\in\partial\DD\smallsetminus \Lambda_{NT}(\Gamma)$, then $Cl_{\A}(\pi, e^{i\theta}) \subset \partial\Omega$.
		\item If $e^{i\theta}\in \Lambda_{NT}(\Gamma)$, and $\eta\subset\DD$ is a curve landing non-tangentially at $e^{i\theta}$, then $$Cl_\eta(\pi, e^{i\theta})\cap\Omega \neq \emptyset.$$ In particular, $Cl_\mathcal{A}(\pi, e^{i\theta})\cap\Omega \neq \emptyset$.
	\end{enumerate}
\end{lemma}
\begin{proof}
	To prove (a), assume that there exists $w\in\Omega$ such that $w\in Cl_\A(\pi, e^{i\theta})$. By the definition of the angular cluster set, there exists a Stolz angle $\Delta$ at $ e^{i\theta} $ (which we can take to be a hyperbolic Stolz angle) and a sequence $\left\lbrace \tilde z_n\right\rbrace _{n}\subset \Delta$ such that $\pi(\tilde z_n)\to w$. Taking now $r > 0$ small enough, the hyperbolic disc $D_\Omega(w, r)$ is simply connected and contains infinitely many of the values $\pi(\tilde z_n)$. By Lemma \ref{lem:superSP}(b), there exists a sequence $\left\lbrace \tilde w_n\right\rbrace _n\subset\DD$ such that  $ \textrm{dist}_\mathbb{D}(\tilde w_n, \tilde z_n)<r $ and $\pi(\tilde w_n) = w$. Therefore, the points $\tilde w_n$ also belong to a hyperbolic Stolz angle $\Delta'\supset \Delta$ and, since $\pi(\tilde w_n) = w$, there exists, for any $n\geq 1$, a deck transformation $ \gamma_n \in \Gamma$ with $ \gamma_n(\tilde w_1) = \tilde w_n$. Hence, $ \gamma_n(\tilde w_1)\to e^{i\theta} $ non-tangentially, showing that $e^{i\theta}\in \Lambda_{NT}(\Gamma)$; a contradiction.

	The proof of (b) proceeds on a similar logic. Let $\eta\subset\DD$ be a curve landing non-tangentially at $e^{i\theta}\in \Lambda_{NT}(\Gamma)$. By definition of the non-tangential limit set, for every $w\in \Omega$ we can take a hyperbolic Stolz angle $\Delta\subset\DD$ at $e^{i\theta}$ and a sequence $\left\lbrace \tilde w_n\right\rbrace _n\subset\Delta$ such that $\pi(\tilde w_n) = w$. Wihtout loss of generality we can assume that both $\eta$ and $\left\lbrace \tilde w_n\right\rbrace _n$ belong to the same hyperbolic Stolz angle $\Delta$. Thus, we can find points $\tilde z_n\in \eta$ such that $$\textrm{dist}_\Omega(w, \pi(\tilde z_n)) \leq \textrm{dist}_\DD(\tilde w_n, \tilde z_n) < r,$$ where $r$ is a constant defined solely by $\Delta$. It follows that the sequence $\left\lbrace \pi(\tilde z_n)\right\rbrace _n$ belongs to $\overline{D_\Omega(w, r)}$, a compact set, and so has an accumulation point $w^*\in \Omega$. This point, by construction, belongs to $Cl _\eta(\pi, e^{i\theta})$, concluding the proof. 
\end{proof}

\subsection{Admissible crosscuts and null-chains}\label{ssec:crosscuts}
In this section, we analyse the crosscuts of $\DD$ generated by lifting the boundary curves $\sigma_{s_1s_2\dots s_n}$ of a given Ohtsuka exhaustion of $\Omega$. We start with the following definition.
\begin{defi}{\bf (Admissible crosscut and null-chain)}\label{def-admissible-crosscut-null-chain}
		Let $ \Omega $ be a multiply connected domain, let $ \pi\colon \mathbb{D}\to\Omega $ be a universal covering, and let $ e^{i\theta} \in\partial\mathbb{D}$. We say that a crosscut $ C $ at $ e^{i\theta} $, with crosscut neighbourhood $ N_C$, is {\em admissible} if either $ \pi (C) $ is a non-contractible Jordan curve, or $ C $ is non-degenerate and $ \pi (N_C) $ is simply connected.

		\noindent An  {\em admissible null-chain} is a collection of admissible crosscuts $ \left\lbrace C_m\right\rbrace _m \subset \mathbb{D}$ such that\begin{enumerate}[label={(\alph*)}]
			\item the crosscuts $ C_m  $ are pairwise disjoint in $ {\mathbb{D} }$;
			\item if $N_m := N_{C_m}$ is the crosscut neighbourhood of $ e^{i\theta} $ bounded by $ C_m $, then $ N_{m+1}\subset N_m $;
			\item $ \bigcap\limits_n \overline{N_m}=\left\lbrace e^{i\theta}\right\rbrace  $.
		\end{enumerate}
\end{defi}
Notice that an admissible crosscut is a (possibly degenerate) crosscut in $ \mathbb{D} $, whose image under $ \pi $ divides $ \Omega $ in precisely two subdomains. Note that admissible null-chains may have degenerate crosscuts, and therefore they are not null-chains in the strict sense of Definition \ref{null-chain} (see also Sect. \ref{subsect-prime-ends}, where we prove that such unusual null-chains correspond to isolated boundary points of $\partial \Omega$).

\subsection*{Lifting non-contractible curves and constructing null-chains}
Next, consider a given Ohtsuka exhaustion $\{\Omega_n\}_n$ of $\Omega$. First, according to Lemma \ref{lemma-lifts-of-curves}, the lift of each boundary curve $ \sigma_{s_1s_2\dots s_n}\subset\partial\Omega_n$ consists of countably many disjoint crosscuts in $ \mathbb{D} $ (recall that we are assuming that our domain has connectivity greater than two). Hence, lifting all the boundary curves gives rise to a countable collection of disjoint crosscuts in $ \mathbb{D} $, which have different endpoints whenever they correspond to different preimages of the same boundary curve $ \sigma_{s_1s_2\dots s_n}  $ or to two non-homotopic boundary curves. Notice that, by construction, all these crosscuts are admissible.

As opposed to the case of simply connected domains, for which the uniqueness of nested preimages between crosscuts is preserved under the Riemann map, the nestedness relationship among the curves $ \sigma_{s_1s_2\dots s_n}  $ and the domains $ D_{s_1s_2\dots s_n}  $  is not preserved by $ \pi $ in general. Indeed, we shall distinguish the following two cases.
\begin{itemize}
	\item If $ \Sigma \subset\partial\Omega$ is an isolated boundary component, then there exists a level of the exhaustion $ \Omega_n $ for which a component $ \sigma_\Sigma$ of $\partial \Omega_n $ separates $ \Sigma $ from the rest of the boundary, and the associated domain $ D_\Sigma $ is doubly connected. Then, according to Lemma \ref{lemma-lifts-of-curves}(d), $ \pi^{-1}(D_\Sigma) $ consists of countably many disjoint crosscut neighbourhoods, each of them bounded by a preimage $\tilde\sigma_\Sigma$ of $ \sigma_\Sigma $ (see Fig. \ref{fig-isolated-boundary-component}). Moreover, any non-contractible curve $\sigma$ homotopic to $ \sigma_\Sigma  $ in $ D_\Sigma $, including the boundary component of $\partial\Omega_m$, $m > n$, that surrounds $\Sigma$, admits a unique nested crosscut $\tilde\sigma$ with the same endpoints as $\tilde\sigma_\Sigma$, and the uniqueness of nesting is preserved (provided the right choice of preimage is made).

	\item  If $ \Sigma \subset\partial\Omega$ is a non-isolated boundary component, with sequence $ \underline{s}(\Sigma)=s_1s_2\dots $, then none of the domains $ D_{s_1s_2\dots s_n} $, $ n\geq 1 $, is doubly connnected. Then, according to Lemma \ref{lemma-lifts-of-curves}(d), when one considers the crosscuts arising from lifting $ \sigma_{s_1s_2\dots s_n} $, no unique nesting relation is satisfied. More precisely, inside any connected component of $ \pi^{-1}(D_{s_1s_2\dots s_n} ) $, there are countably many choices for a component of  $ \pi^{-1}(D_{s_1s_2\dots s_{n+1}} ) $. See Figure \ref{fig-isolated-boundary-component}, and compare with \cite[Sect. 4]{Oht54}.
\end{itemize}

\begin{figure}[h]
	\captionsetup[subfigure]{labelformat=empty, justification=centering}
	\begin{subfigure}{\textwidth}
		\includegraphics[width=14cm]{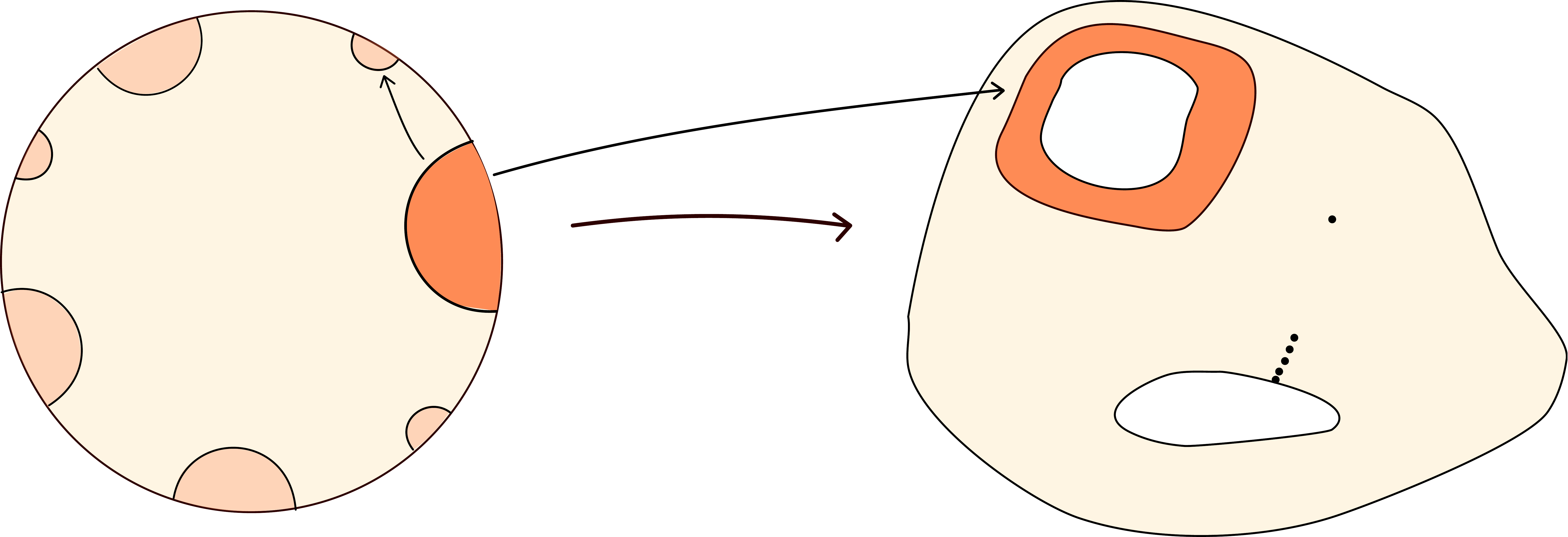}
		\setlength{\unitlength}{14cm}
				\put(-0.56, 0.21){$ \pi $}
		\put(-1, 0.31){$ \mathbb{D} $}
		\put(-0.14, 0.3){$ \Omega $}
				\put(-0.805, 0.25){\footnotesize$ \gamma\in\Gamma$}
						\put(-0.72, 0.2){$ N $}
								\put(-0.27, 0.28){\small $ \Sigma$}
									\put(-0.235, 0.28){\small $ D_\Sigma$}
										\put(-0.34, 0.19){$\sigma_\Sigma$}
											\put(-0.56, 0.275){$ \pi|_N$}
		\caption{\footnotesize Lifts of non-contractible curves around isolated boundary components.}
	\end{subfigure}

	\vspace{0.7cm}
	
	\begin{subfigure}{\textwidth}
		\includegraphics[width=14cm]{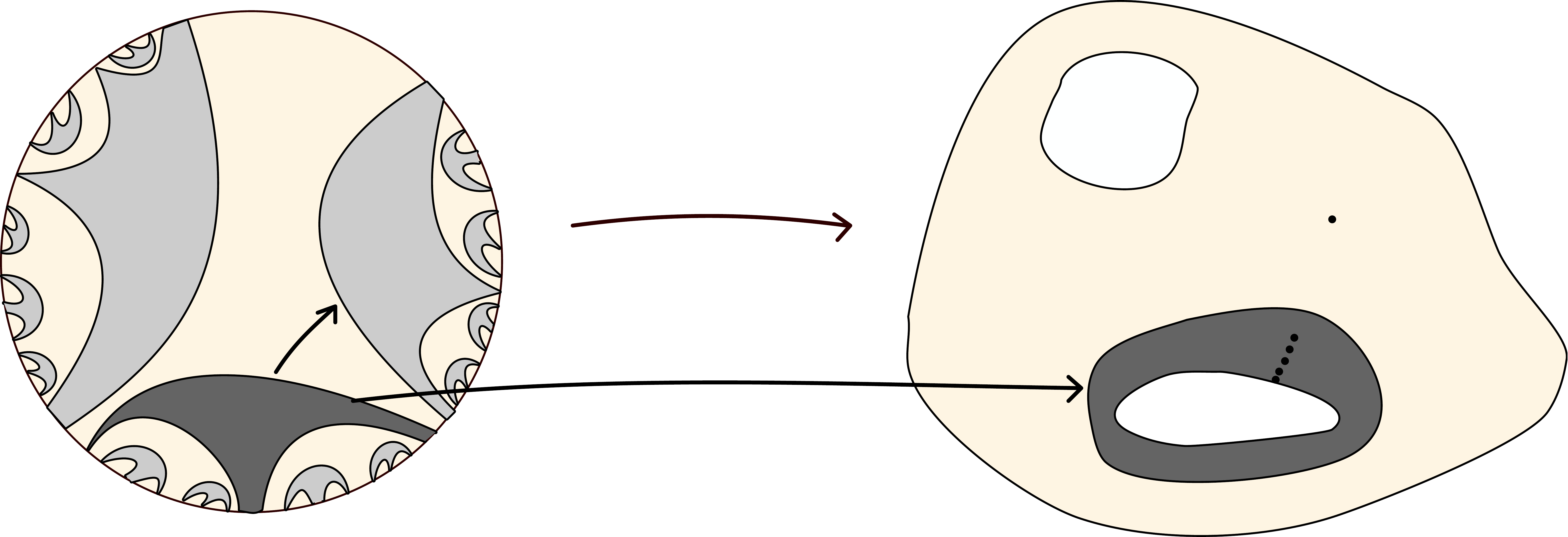}	
				\setlength{\unitlength}{14cm}
			\put(-0.56, 0.21){$ \pi $}
		\put(-1, 0.31){$ \mathbb{D} $}
		\put(-0.14, 0.3){$ \Omega $}
						\put(-0.865, 0.12){\footnotesize$ \gamma\in\Gamma$}
							\put(-0.56, 0.105){$ \pi|_N$}
											\put(-0.9, 0.05){$ N $}
		\caption{\footnotesize Lifts of non-contractible curves around non-isolated boundary components.}
	\end{subfigure}
	\caption{\footnotesize Schematic representation of the two situations that can arise when lifting a domain $ D_{s_1s_2\dots s_n} $ in Ohtsuka exhaustion, depending on whether the non-contractible curve surrounds an isolated or a non-isolated boundary component. }\label{fig-isolated-boundary-component}
\end{figure}

Now let us build null-chains for points in the unit circle, out of the crosscuts $\pi^{-1}(\sigma_{s_1s_2\dots s_n})$, in the following way.
\begin{defi}{\bf (Fundamental sequence of crosscuts)}
	Let $ \Omega $ be a mutiply connected domain, with an Ohtsuka exhaustion $ \left\lbrace \Omega_n\right\rbrace _n $. Let $\pi\colon\mathbb{D}\to\Omega$ be a universal covering map, and let $ e^{i\theta}\in\partial\mathbb{D} $.
	Consider all the crosscuts $ \left\lbrace C_m\right\rbrace _m $ out of all the curves $\pi^{-1}(\sigma_{s_1s_2\dots s_n})$ separating $0$ from $e^{i\theta}$ (including degenerate crosscuts terminating at $e^{i\theta}$), and order them in such a way that $ C_m $ separates $ C_{m+1} $ from $0$. We call this the \emph{sequence of fundamental crosscuts} at $e^{i\theta}$.
\end{defi}

Observe that the fundamental sequence of crosscuts  at a given point $ e^{i\theta}\in\partial\mathbb{D} $ depends on the original Ohtsuka exhaustion $ \left\lbrace \Omega_n\right\rbrace _n $.
Note also that, if $e^{i\theta}\in\partial\DD$ is a hyperbolic fixed point, its sequence of fundamental crosscuts may be finite, or even empty. Moreover, the obtained sequence of fundamental crosscuts may not be an admissible null-chain, since the corresponding crosscut neighbourhoods might not shrink to $e^{i\theta}$ (for instance if $e^{i\theta}$ is regular). In the same vein, the corresponding crosscut neighbourhoods might behave very poorly under $\pi$ -- if $e^{i\theta}$ is singular, their image will always cover $\Omega$. However, we are interested in these crosscuts in relation to the classification of points on the unit circle, and not for the purpose of building a theory of prime ends -- for that, see Section \ref{subsect-prime-ends}.

\subsection{The depth of a point on the unit circle}\label{ssec:depth} As we shall see in Theorem \ref{thm-main-result-complete},
it turns out that the boundary behaviour of $\pi$ at $e^{i\theta}$ is completely determined by the behaviour of the sequence $\left\lbrace \pi(C_m)\right\rbrace _m$ of images of fundamental crosscuts at $e^{i\theta}$ obtained from an Ohtsuka exhaustion of $\Omega$ -- in fact, we have already seen hints of this in the form of Theorem \ref{thm:notalphaimage}. This motivates the following definition.
\begin{defi}{\bf (Sequence of depths)}\label{def:depthalpha}
	Let $ e^{i\theta}\in\partial\mathbb{D} $, and let $ \left\lbrace C_m\right\rbrace _m $ be the sequence of fundamental crosscuts at $ e^{i\theta} $ (from the exhaustion $ \left\lbrace \Omega_n\right\rbrace _n $). We define the {\em sequence of depths} of $ e^{i\theta} $ as \[\underline{d}(e^{i\theta}) = \left\lbrace d_m\right\rbrace _m,\] where $ d_m=n $ if and only if $ \pi (C_m)\subset \partial\Omega_n $.
	We say that $ e^{i\theta} $ is of {\em finite}, {\em infinite}, or {\em oscillating depth} if the sequence $ \underline{d}(e^{i\theta}) $ is (respectively) bounded (or finite), diverges to infinity, or neither.
\end{defi}
\begin{obs}
	Notice that, because the fundamental crosscuts are allowed to be degenerate, fixed points of parabolic deck transformations are of infinite depth. Fixed points of hyperbolic deck transformations, on the other hand, are of finite depth.
\end{obs}
In order to guarantee that the definition above makes sense, we must show that the depth of a point is determined by the universal covering $\pi$, and does not depend on which particular Ohtsuka exhaustion of $\Omega$ is considered. This philosophy is somewhat implicit in \cite{Oht54}, but we make it explicit here for completeness.
\begin{lemma}{\bf (Depth does not depend on the exhaustion)}\label{lemma-equivalence-ohtsuka}
	Let $ \Omega $ be a multiply connected domain, and let $ \pi\colon \mathbb{D}\to\Omega $ be a universal covering. Let $ e^{i\theta} \in\partial\mathbb{D}$. Then, $ e^{i\theta} $ is of finite, infinite or oscillating depth  not depending on the chosen Ohtsuka exhaustion of $ \Omega $.
\end{lemma}
\begin{proof}
The proof is divided in two parts: first, we show that the we can move the exhaustion homotopically, i.e. an exhaustion $\{\Omega_n\}_n$ is equivalent to any other exhaustion $\{\Omega_n'\}_n$ if $\partial\Omega_n$ is homotopic to $\partial\Omega_n'$, for every $n$. This, in turn, is a simple consequence of Lemma \ref{lemma-lifts-of-curves}: for any given $n$, the fundamental crosscuts $C_n$ and $C_n'$ generated by the exhaustion have the same fixed points, and so both give rise to the same sequence of depths.

If $\{\Omega_n\}_n$ and $\{\Omega_n'\}_n$ do not have homotopic boundaries, assume that $\left\lbrace m_k\right\rbrace _k$ is a subsequence so that the depths of $\left\lbrace C_{m_k}\right\rbrace _k$ are bounded, where $\{C_m\}_m$ denotes the fundamental sequence of crosscuts at $e^{i\theta}$. Then, since $\{\Omega_n'\}_n$ is an exhaustion of $\Omega$ and the collection $\left\lbrace \pi(C_{m_k})\right\rbrace _k$ is contained in a compact set, there exists a minimal value of $n$ so that $\Omega_n'$ contains all the curves $\pi(C_{m_k})$, $k\in\mathbb{N}$. Lifting this relation, it follows that each $C_{m_k}$ is contained in a distinct preimage component $\tilde\Omega_{n_k}'$ of $\Omega_n'$, and gathering the component of $\partial\Omega_{n_k}'$ separating $0$ from $e^{i\theta}$ for every $k$ gives rise to a subsequence $\left\lbrace C_{n_k}'\right\rbrace _k$ of fundamental crosscuts of bounded depth generated by the exhaustion $\Omega_n'$. Arguing the same way for sequences that tend to infinity, one ends the proof of the lemma.
\end{proof}

\subsection{Boundary components associated to a point on the unit circle}\label{ssec:association}
In this subsection, we examine the boundary components of $\Omega$ belonging to the angular cluster set of $\pi$ at points $e^{i\theta}\in\partial\DD$. We start with the following definition.
\begin{defi}{\bf (Boundary component associated to $ e^{i\theta}\in\partial \mathbb{D} $)}\label{def:associatedcomponent}
	Let $ e^{i\theta}\in\partial \mathbb{D} $. Let $ \Sigma $ be a boundary component of $ \Omega $. If \[ Cl_{\mathcal{A}}(\pi, e^{i\theta})\cap \Sigma\neq\emptyset, \] we say that $ \Sigma $ is a {\em boundary component associated} to $ e^{i\theta}$.
\end{defi}
According to Definition \ref{def:boundaryaddress}, given an Ohtsuka exhaustion $\{\Omega_n\}_n$ of $ \Omega $, each boundary component $\Sigma\subset\partial\Omega$ has its unique address, which specifies the sequence of boundary components $\sigma_{s_1s_2\dots s_n}$ of $\Omega_n$ that converge to $\Sigma$ as $n\to+\infty$. This motivates the following parallel definition.
\begin{defi}{\bf (Associated boundary address)}\label{def:associatedaddress}
Let $e^{i\theta}\in\partial\DD$, and let $\Sigma$ be a component of $\partial\Omega$. We say that $e^{i\theta}$ has $\underline{s}(\Sigma) = \left\lbrace s_n\right\rbrace _n$ as an \textit{associated boundary address} if we can find a subsequence of fundamental crosscuts $\left\lbrace C_{m_n}\right\rbrace _n$ at $e^{i\theta}$ such that $\pi(C_{m_n}) = \sigma_{s_1s_2\dots s_n}$, for all $n\geq 1$.
\end{defi}
We can make the parallel between Definitions \ref{def:associatedcomponent} and \ref{def:associatedaddress} more explicit.
\begin{prop}{\bf (Associated boundary components)}\label{prop-associated-boundary-comp}
Let $ \Sigma $ be a boundary component of $ \Omega $ with address $ \underline{s} (\Sigma)$.
Then,	$ e^{i\theta} \in\partial\mathbb{D}$ has $ \underline{s} (\Sigma)$ as an associated boundary address if and only if it has $\Sigma$ as an associated boundary component. 

\noindent Moreover, if $ e^{i\theta} $ is of escaping type, then $  Cl_{\mathcal{A}}(\pi, e^{i\theta})\subset\Sigma$ and $ \Sigma $ is the unique boundary component associated to $ e^{i\theta}$. For every countable collection of non-isolated boundary components $ \left\lbrace \Sigma_n \right\rbrace _n$, there exist uncountably many points $ e^{i\theta}\in\partial\mathbb{D} $ of bungee type such that each $e^{i\theta}$ has $ \left\lbrace \Sigma_n \right\rbrace _n$ as associated boundary components.
\end{prop}
\begin{proof}
	First, we argue that we can restrict our attention to the radial limit set. Indeed, let $\left\lbrace z_n\right\rbrace _n\subset\DD$ be a sequence tending non-tangentially to $e^{i\theta}$ and such that $\pi(z_n)\to\partial\Omega$ as $n\to+\infty$. Then, by the definition of a hyperbolic Stolz angle, there exists $r > 0$ and a sequence $\left\lbrace w_n\right\rbrace _n\subset R_\theta$ such that $\textrm{dist}_\Omega(\pi(w_n), \pi(z_n)) < r$ for all $n$. By Lemma \ref{lem:superSP}, the spherical distance $\textrm{dist}_{\Chat}(z_n, w_n)$ tends to $ 0$ as $n\to+\infty$, and therefore both sequences accumulate at the same points on $ \partial\Omega $.

	Now, consider an exhaustion $ \left\lbrace \Omega_n\right\rbrace _n $ of $ \Omega $ such that all the curves in $ \left\lbrace \partial\Omega_n\right\rbrace _n $ are geodesics (which we can do by Lemmas \ref{lemma-lifts-of-curves} and \ref{lemma-equivalence-ohtsuka}). Let $ \left\lbrace C_m\right\rbrace _m $ be the sequence of fundamental crosscuts at $ e^{i\theta} $ given by the exhaustion  $ \left\lbrace \Omega_n\right\rbrace _n $. Then, $ R_\theta $ intersects each crosscut $ C_m $ exactly once, and intersects no other crosscut obtained from the exhaustion, since by Proposition \ref{prop:hypdisc} each curve $\pi^{-1}(\sigma_{s_0s_1\dots s_n})$ is a circle arc orthogonal to $\partial\DD$. 
	
	First assume $ e^{i\theta} \in\partial\mathbb{D}$ has $ \underline{s} (\Sigma)$ as an associated boundary address. By Definition \ref{def:associatedaddress}, there exists a subsequence $ \left\lbrace C_{m_n}\right\rbrace _n $ such that $ \pi(C_{m_n})=\sigma_{s_1s_2\dots s_n} $. This already implies that \[ Cl_R(\pi, e^{i\theta})\cap\Sigma \neq\emptyset.\] For the other implication, we claim $ \pi (R_\theta ) $ cannot accumulate in a boundary component whose boundary address is not associated to $ e^{i\theta} $. Indeed, that would require $\pi(R_\theta)$ to cross a boundary curve $\sigma_{s_1s_2\dots s_n}'$ whose lifts are, by hypothesis, not part of the fundamental crosscuts at $e^{i\theta}$, and therefore cannot intersect $R_\theta$.
	
	The statement for points of escaping type follows straightforward from Proposition \ref{prop-escaping-bounded-bungee}. 

	It is left to see that for every countable collection of non-isolated boundary components $ \left\lbrace \Sigma_n \right\rbrace _n$, there exist uncountably many  points $ e^{i\theta}\in\partial\mathbb{D} $ of bungee type associated to $ \left\lbrace \Sigma_n \right\rbrace _n$. Equivalently, we shall prove that, given a sequence $ \left\lbrace \Sigma_n \right\rbrace _n$ of non-isolated boundary components, there exist uncountably many points $ e^{i\theta}\in\partial\mathbb{D} $ having $ \left\lbrace \underline{s}(\Sigma_n) \right\rbrace _n$ as associated boundary addresses. To do so, we rely on the nestedness properties of the crosscuts coming from the exhaustion. We will outline the argument for two non-isolated boundary components $\Sigma$ and $\Sigma'$; the argument for countably many components is similar (one can use a diagonal argument to alternate between components). Indeed, let $\underline{s}(\Sigma) = s_1s_2\dots$ and $\underline{s}'=\underline{s}(\Sigma') = s_1's_2'\dots$ denote the boundary addresses of $\Sigma$ and $\Sigma'$, where $\Sigma, \Sigma'\subset\partial\Omega$ are non-isolated boundary components, and let $D_{s_1s_2\dots s_n}$ and $D_{s_1's_2'\dots s_n'}$ denote the complementary components of $\Omega_n$ such that
	\[ \Sigma = \bigcap_{n\geq 1} \overline{D_{s_1s_2\dots s_n}} \  \ \ \text{ and } \ \ \  \Sigma' = \bigcap_{n\geq 1} \overline{D_{s_1's_2'\dots s_n'}}. \]
	Now, consider any lift $\tilde D_1$ of $D_{s_1}$; by Lemma \ref{lem:restriction}, $\pi\colon\tilde D_1\to D_{s_1}$ is a universal covering, but since $D_{s_1}$ is (by hypothesis) not simply connected, this covering has infinite degree. In particular, $\partial\tilde D_1$ consists of infinitely many arcs with endpoints in $\partial\DD$, one of which (say $C_1$) separating $0$ from all the others (see Lemma \ref{lemma-lifts-of-curves}). The domain $\tilde D_1$ also contains infinitely many lifts of $\sigma_{s_1}\subset\partial D_{s_1}$; we pick any one of these distinct from $C_1$ and cross back into a component $\tilde \Omega_1$ of $\pi^{-1}(\Omega_1)$. Once there, we choose a crosscut $C_1'$ corresponding to a preimage of $\sigma_{s_1'}\subset \partial  D_{s_1'}$. After crossing $C_1'$ into a component $\tilde D_1'$ of $\pi^{-1}(D_{s_1'})$, which contains (once again) infinitely many preimages of $\sigma_{s_1'}$; we can cross one that is distinct from $C_1'$, then cross another component of $\pi^{-1}(\sigma_{s_1})$ into a new component of $\pi^{-1}(D_{s_1})$. This time, we choose a preimage $C_2$ of $\sigma_{s_1s_2}$, which will take us into a preimage of $D_{s_1s_2}$. Repeating this process inductively, we produce sequences $C_n$ and $C_n'$ of fundamental crosscuts of $\DD$, which are shrinking to a point $e^{i\theta}\in\partial\DD$. This point, by construction, has both $\underline{s}$ and $\underline{s}'$ as associated boundary sequences.

	There are uncountably many points with this property because we made countably many choices among countably many possibilities at each choice.
\end{proof}

\subsection{Ohtsuka's classification of points on the unit circle revisited}\label{ssec:Ohtsuka2}
Using the terminology and results in Sections \ref{ssec:depth} and \ref{ssec:association}, we can restate Ohtsuka's results (Thms. \ref{thm:alphaimage} and \ref{thm:notalphaimage}) as follows.
\begin{summ}{\bf (Ohtsuka's classification of points in $ \partial \mathbb{D} $) 
	} \label{thm-ohtsuka}
	Let $ \Omega $ be a multiply connected domain,  let $ \pi\colon\mathbb{D}\to\Omega $ be a universal covering and let $ e^{i\theta}\in\partial\mathbb{D} $. Let $ \left\lbrace \Omega_n\right\rbrace _n $ be an Ohtsuka exhaution of $ \Omega $. Then, the following hold.
	\begin{enumerate}[label={\em (\alph*)}]
	\item $ e^{i\theta} $ is of finite depth if and only if $ e^{i\theta}  $ lies on the boundary of some connected component of $ \pi^{-1}(\Omega_n) $, for some $ n\geq 1 $. There are uncountably many points of finite depth.
	\item If $ e^{i\theta} $ is of infinite depth and belongs to an $ \alpha $-image $ \tilde\Sigma $ corresponding to $ \Sigma \subset\partial \Omega $, then there exists a curve $ \eta\colon\left[ 0,1\right) \to \mathbb{D}$ such that $ L(\eta)\subset \tilde\Sigma $ and $ L(\pi(\eta))\subset \Sigma $. Conversely, the lift of any curve $ \eta\colon\left[ 0,1\right) \to \Omega$ with $ L(\eta)\subset \Sigma $ has landing set contained in an $ \alpha $-image  corresponding to $ \Sigma $.

	\noindent In particular,
	\begin{enumerate}[label={\em (\roman*)}]
		\item for any boundary component $ \Sigma $, there are at least countably many  points associated to $ \Sigma $;
		\item all regular points have infinite depth, and hence are associated to a unique boundary component;
		\item $ \Sigma $ is a non-isolated boundary component if and only if there exist singular points associated to $ \Sigma $. In this case, there are uncountably many of them.
	\end{enumerate}
		\item $ e^{i\theta} $ is of oscillating depth if and only if the image of any curve $ \eta $ landing at $ e^{i\theta} $ is neither compactly contained in $ \Omega $ nor converges to $\partial \Omega $. 
	
		\noindent There exist points of oscillating depth if and only if $ \Omega $ has infinite connectivity, In this case, there are uncountably many of them.
\end{enumerate}
\end{summ}


According to the previous result, it is clear that points of oscillating depth correspond to points of bungee type, as introduced in Definition \ref{def:escaping-bounded-bungee}. It is less clear {\em a priori} if analogous correspondences hold for points of bounded or escaping type. Moreover, the relation between the depth of a point and the fundamental group is mainly unexplored. By appealing to hyperbolic Stolz angles and using a specific Ohtsuka exhaustions of $\Omega$, we  complement Ohtsuka's work by addressing the previous questions in a satisfactory manner, as show Propositions \ref{prop-bound-esc-bungee2} and \ref{prop:type-group}. This way, we provide a detailed analysis of the angular cluster sets for the different types of points, and describe more precisely the relation with the limit set (and the non-tangential limit set),  improving substantially the implications in part (b) of Theorem \ref{thm-ohtsuka}.

\begin{prop}{\bf (Depths and radial images)}\label{prop-bound-esc-bungee2}
	Let $ \Omega $ be a multiply connected domain, let $ \pi\colon\mathbb{D}\to\Omega $ be a universal covering, and let $ e^{i\theta} \in\partial\mathbb{D}$. Then,
		\begin{enumerate}[label={\em (\alph*)}]
		\item $ e^{i\theta} $ is of finite depth if and only if $ e^{i\theta}  $ is of bounded type;
		\item $ e^{i\theta} $ is of infinite depth if and only if $ e^{i\theta} $ is of escaping type;
		\item $ e^{i\theta} $ is of oscillating depth if and only if $ e^{i\theta} $ is of bungee type.
	\end{enumerate}
\end{prop}
\begin{proof}
Consider, without loss of generality,  an exhaustion $ \left\lbrace \Omega_n\right\rbrace _n $ of $ \Omega $ such that the curves in $ \left\lbrace \partial\Omega_n\right\rbrace _n $  are all closed geodesics. Therefore, the corresponding crosscuts in $ \mathbb{D} $ are arcs of Euclidean circles, orthogonal to $ \partial\mathbb{D} $. In this situation, if one considers all the crosscuts which arise from lifting $ \left\lbrace \Omega_n\right\rbrace _n $, the radius $ R_\theta $ at $ e^{i\theta} $ crosses each of the crosscuts forming the fundamental sequence of crosscuts at $ e^{i\theta} $ exactly once, and is disjoint from the rest of crosscuts.

We now prove (a). If $ e^{i\theta} $ is of finite depth, the image under $ \pi $ of the fundamental sequence of crosscuts at $ e^{i\theta} $ is contained in some level of the exhaustion, say $ \Omega_n $. Since the radius $ R_\theta $ only crosses the fundamental crosscuts at $e^{i\theta}$, it follows that its image is bounded in $ \Omega $, and hence $ e^{i\theta} $ is bounded. Conversely, if $ e^{i\theta} $ is bounded, then the corresponding sequence of fundamental crosscuts must have finite depth -- otherwise, $\pi(R_\theta)$ contains a sequence of points that is not compactly contained in $\Omega$, coming from the intersection of $R_\theta$ with the fundamental crosscuts.

The proof of (b) proceeds similarly, relying on our specially chosen Ohtsuka exhaustion of $\Omega$. Once (b) is proved, then (c) follows by exclusion, and this ends the proof of the proposition.
\end{proof}

\begin{prop}{\bf (Depths and fundamental group)}\label{prop:type-group}
Let $\Omega$ be a multiply connected domain, let $\pi\colon\DD\to\Omega$ be a universal covering, and let $e^{i\theta}\in\partial\DD$. Then,
\begin{enumerate}[label={\em (\alph*)}]
	\item $e^{i\theta}$ is of escaping type if and only if $e^{i\theta}\in\partial\DD\setminus\Lambda_{NT}$.
	\item $e^{i\theta}$ is of bounded type if and only if $e^{i\theta}\in\Lambda_{NT}'$, where $\Lambda_{NT}'$ denotes the non-tangential limit set of a finitely generated subgroup of $\Gamma$.
	\item $e^{i\theta}$ is of bungee type if and only if $e^{i\theta}\in\Lambda_{NT}$, but does not belong to the non-tangential limit set of any finitely generated subgroup of $\Gamma$.
\end{enumerate}
\end{prop}
\begin{proof}
First, it follows immediately from Lemma \ref{lem:non-tangential-curves} that points of escaping type are in $ \partial\mathbb{D}\smallsetminus\Lambda_{NT}$, whereas points of bounded and bungee type are in $ \Lambda_{NT} $. Next, in order to distinguish between bounded and bungee type, we rely on Proposition \ref{prop-bound-esc-bungee2}, which asserts that $ e^{i\theta}  $ is of bounded type if and only if it lies in the boundary of some connected component of $ \pi^{-1}(\Omega_n) $.

Indeed, let $e^{i\theta}\in\partial\DD$ be of bounded type. Then, by definition, the image of the radius $R_\theta$ is compactly contained in $\Omega_n$ for some sufficiently large $n$, and $e^{i\theta}$ lies on the boundary of some component $\tilde\Omega_n$ of $\pi^{-1}(\Omega_n)$. This component is simply connected, and $\pi|_{\tilde\Omega_n}$ is a universal covering of $\Omega_n$ by Lemma \ref{lem:restriction}. Using a Riemann map $\varphi\colon\DD\to\tilde\Omega_n$ that fixes the origin and $e^{i\theta}$, we can apply Lemma \ref{lem:non-tangential-curves} to the new covering $\pi\circ\varphi$ and conclude that $e^{i\theta}$ is in the non-tangential limit set of $\Lambda'$, where $\Lambda'$ is the group of deck transformations of $\pi\circ\varphi$. Since $\Omega_n$ is finitely connected, $\Lambda'$ is finitely generated, and the conclusion follows.
	
Conversely, assume that $e^{i\theta}\in\partial\DD$ is in the non-tangential limit set of a finitely generated subgroup $\Lambda'$ of $\Lambda$. Let $\gamma_1, \ldots, \gamma_m$ be generators of $\Lambda'$; then, by Lemma \ref{lem:deck}, they correspond to (simple) closed curves $\eta_1, \ldots \eta_m$ in $\Omega$, which are contained in some $\Omega_n$ for some large enough $n$. We can assume that the curves $\eta_i$, $1\leq i\leq m$, all pass through the point $\pi(0)$ (since we can deform them within their respective homotopy classes), meaning that the orbit $\Lambda'(0) = \{\gamma(0)\colon \gamma\in\Lambda'\}$ is now a subset of the connected component $\tilde\Omega_n$ of $\pi^{-1}(\Omega_n)$ that contains $0$. Since this orbit contains, by hypothesis, a subsequence converging to $e^{i\theta}$, it follows that $e^{i\theta}\in\partial\tilde\Omega_n$, whence (by Theorem \ref{thm-ohtsuka}) it is of bounded type.
	
The bungee case now follows by exclusion.
\end{proof}

\subsection{The Main Theorem and its immediate consequences}\label{subsect-proof-main-result}
We can now summarise the results of Section \ref{ssec:Ohtsuka2} in the following statement.
\begin{summ}{\bf (Main Theorem)}\label{thm-main-result-complete}
	Let $ \Omega $ be a multiply connected domain, let $ \pi\colon \mathbb{D}\to\Omega $ be a universal covering and let $ e^{i\theta} \in\partial\mathbb{D}$. Then, the following hold.
		\begin{enumerate}[label={\normalfont (\alph*)}]
		\item $ e^{i\theta} $ is of escaping type if and only if $ e^{i\theta} $ is of infinite depth if and only if $ e^{i\theta} \in\partial\mathbb{D}\smallsetminus \Lambda_{NT}$. Moreover,\label{item:MTescaping}
		\begin{enumerate}[label={\normalfont (\roman*)}]
			\item $ e^{i\theta} $ is associated to a unique boundary component $ \Sigma $, and $ Cl_\mathcal{A}(\pi, e^{i\theta})\subset\Sigma$;
			\item if $ \Sigma $ is an  isolated non-degenerate continuum in $ \partial \Omega $, then $ e^{i\theta} $ is regular;
			\item $ \Sigma=\left\lbrace p\right\rbrace  $ is an isolated point in $ \partial \Omega $ if and only if $ e^{i\theta} $ is a parabolic fixed point.
		\end{enumerate}
		\item $ e^{i\theta} $ is of bounded type if and only if $ e^{i\theta} $ is of finite depth if and only if $e^{i\theta}$belongs to the non-tangential limit set of some finitely generated subgroup of $\Gamma$.\label{item:MTbounded}
		\item $ e^{i\theta} $ is of bungee type if and only if $ e^{i\theta} $ is of oscillating depth if and only if $e^{i\theta}\in \Lambda_{NT}$, but does not belong to the non-tangential limit set of any finitely generated subgroup of $\Gamma$. 
		
		\noindent Moreover, if $ e^{i\theta} $ is associated to a boundary component $ \Sigma $, then $ \Sigma $ is a non-isolated boundary component.\label{item:MTbungee}
	\end{enumerate}
	Furthermore,
	\begin{enumerate}[resume, label={\normalfont (\alph*)}]
	\item for any boundary component $ \Sigma $, there are at least countably many points of escaping type associated to $ \Sigma $;\label{item:MTesc-countable}
	\item there are uncountably many points of bounded type; and\label{item:MTbounded-uncountable}
	\item for every countable collection of non-isolated boundary components $ \left\lbrace \Sigma_n \right\rbrace _n$, there exist uncountably many points of oscillating type $ e^{i\theta}\in\partial\mathbb{D} $ having $ \left\lbrace \Sigma_n \right\rbrace _n$ as associated boundary components.\label{item:MTosc-uncountable}
	\end{enumerate}
\end{summ}

The previous theorem
 has as straightforward consequences \ref{teo:A} and \ref{teo:B}, as promised in the introduction. We restate  them here for the reader's convenience.
\begin{named}{Theorem A}
		Let $ \Omega $ be a multiply connected domain of connectivity greater than two, and let $ \pi\colon \mathbb{D}\to\Omega $ be a universal covering. Then, the following hold.
	\begin{enumerate}[label={\normalfont (\alph*)}]
			\item $ \Lambda_{NT} $ always consists of uncountably many points.
			\item $ \Lambda\smallsetminus \Lambda_{NT} $ is non-empty if and only if there exists a non-isolated boundary component or an isolated boundary point. Moreover, $ \Lambda\smallsetminus \Lambda_{NT} $ is uncountable if and only if there exists a non-isolated boundary component.
			\item If the radial limit  $ \pi^*(e^{i\theta}) $ exists, then $ e^{i\theta}\in\partial\DD\setminus\Lambda_{NT}$.
			\item There always exists countably many ambiguous points for $ \pi\colon \mathbb{D}\to\Omega $.
		\end{enumerate}
\end{named}
\begin{proof} Item (a) follows immediately from Theorem \ref{thm-main-result-complete}, and more specifically from items \ref{item:MTbounded} and \ref{item:MTbounded-uncountable}. For item (b), notice that the existence of points in $\Lambda\setminus\Lambda_{NT}$ implies the existence of either a non-isolated boundary component or an isolated boundary point by Theorem \ref{thm-main-result-complete}\ref{item:MTescaping}. Conversely, given any such boundary component, there are countably many points on $\dD$ associated with it by Theorem \ref{thm-main-result-complete}\ref{item:MTesc-countable}. The uncountability of $\Lambda\setminus\Lambda_{NT}$ is, by Theorem \ref{thm-main-result-complete}\ref{item:MTosc-uncountable}, equivalent to the presence of a non-isolated boundary component (notice that, if there are uncountably many points in $\Lambda\setminus\Lambda_{NT}$, at most countably of those are due to isolated boundary points).

For item (c), since a universal covering map satisfies the path-lifting property, if $ \pi^*(e^{i\theta}) $ exists, then $ \pi^*(e^{i\theta}) \in\partial\Omega$. This already implies that $ e^{i\theta} $ is of escaping type, whence (c) follows by Theorem \ref{thm-main-result-complete}\ref{item:MTescaping}. Finally, (d) follows from (a) combined with Proposition \ref{prop-escaping-bounded-bungee}.
\end{proof}

The statements in \ref{teo:A} deserve a few comments. First, concerning (b), since a subdomain of $\Chat$ has a non-isolated boundary component if and only if it is infinitely connected (recall Lemma \ref{lem:deck}), this gives us a different proof of the main result of \cite{BM74} for the case where $\Omega = \DD/\Gamma$ is a plane domain. In particular, we obtain a different proof of the Ahlfors' Area conjecture for Fuchsian groups that does not rely on the equivalence between finitely generated and geometrically finite Fuchsian groups.
	
Concerning (c), we have that, in particular, $\pi$ having radial limits almost everywhere implies that $\Lambda_{NT}$ has measure zero. Since the converse implication is already covered by the Hopf-Tsuji-Sullivan Ergodic Theorem  \ref{thm-hopf-tsuji-sullivan}, we have proved \ref{teo:B}.
	
\begin{named}{Corollary B}
	Let $\Omega$ be a hyperbolic plane domain, and let $\pi\colon\DD\to\Omega$ be a universal covering. Then, the following are equivalent.
	\begin{enumerate}[label={\normalfont (\alph*)}]
		\item $\Lambda_{NT}$ has zero Lebesgue measure on $\DD$.
		\item $\partial\Omega$ has positive logarithmic capacity.
		\item $\pi$ belongs to the Nevanlinna class.
		\item $\pi^*$ exists almost everywhere on $ \partial\mathbb{D} $.
	\end{enumerate}
\end{named}

\section{Applications}\label{section-applications}
As a consequence of the Main Theorem proved in the previous section, we are capable of describing accesses to the boundary for multiply connected domains,  and develop a prime end theory which is compatible with the action of the universal covering map. We do this in Sections \ref{ssect:accesses} and \ref{subsect-prime-ends}. Using this new machinery, we also construct quick-and-easy examples of Fuchsian groups with `pathological' limit sets in Section \ref{ssec:group-examples}.

\subsection{Accessing the boundary of a multiply connected domain}\label{ssect:accesses}
Here we discuss accesses to points on the boundary of $\partial\Omega$. In particular, we address the questions of existence of accesses and their relation with radial limits and the universal covering.

Let us start with the following result -- which may be folklore, but we were unable to find a reference to it. Recall that a point $ p\in\partial\Omega $ is said to be accessible from $ \Omega $ if there exists a curve in $ \Omega $ landing at it (Def. \ref{def-accessible-point}).
\begin{lemma}{\bf (Accessible points are dense)}\label{lem:dense}
	Let $ \Omega\subset\Chat $ be a domain. Then, the points  in $ \partial\Omega $ which are accessible from $ \Omega $ are dense in $ \partial\Omega $.
\end{lemma}
\begin{proof}
	Let $ x\in\partial\Omega $ and $ r>0 $; we have to see that there is an accessible point in $ D(x,r) $. Let $ z\in \Omega \cap D(x,r)  $,  and consider $ \eta\colon \left[ 0,1\right] \to \mathbb{C} $ to be the straight segment joining $ z $ and $ x $, parametrized so that $ \eta(0)=z $ and $ \eta(1)=x $. Thus, $ \eta $ starts in $ \Omega $ and ends outside it, so there exists a minimal $ t_0\in \left( 0,1\right]  $ with $ \eta(t_0) \notin\Omega$. Then, $ \eta(t_0) \in\partial\Omega$, and it is accessible from $ \Omega $, since the curve $ \left\lbrace \eta (t)\colon \ t\in (0, t_0) \right\rbrace  $ lands at it. See Figure \ref{fig-accessible}.
	
			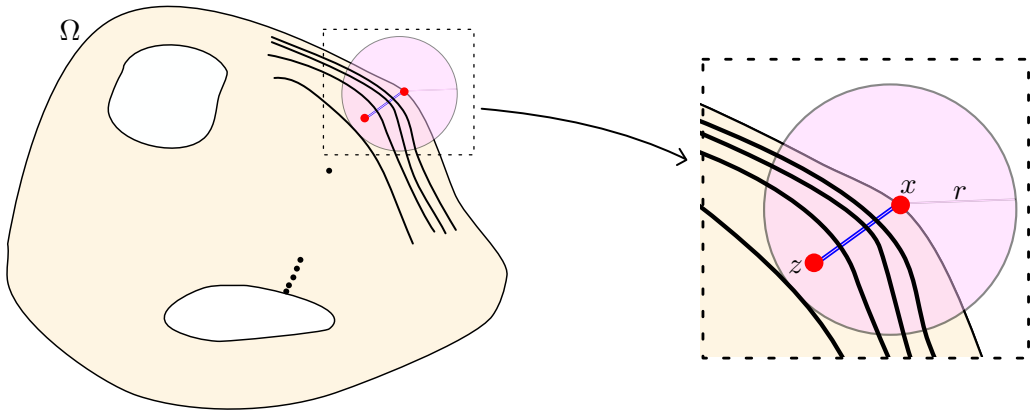
\begin{figure}[h]\centering
		\includegraphics[width=14cm]{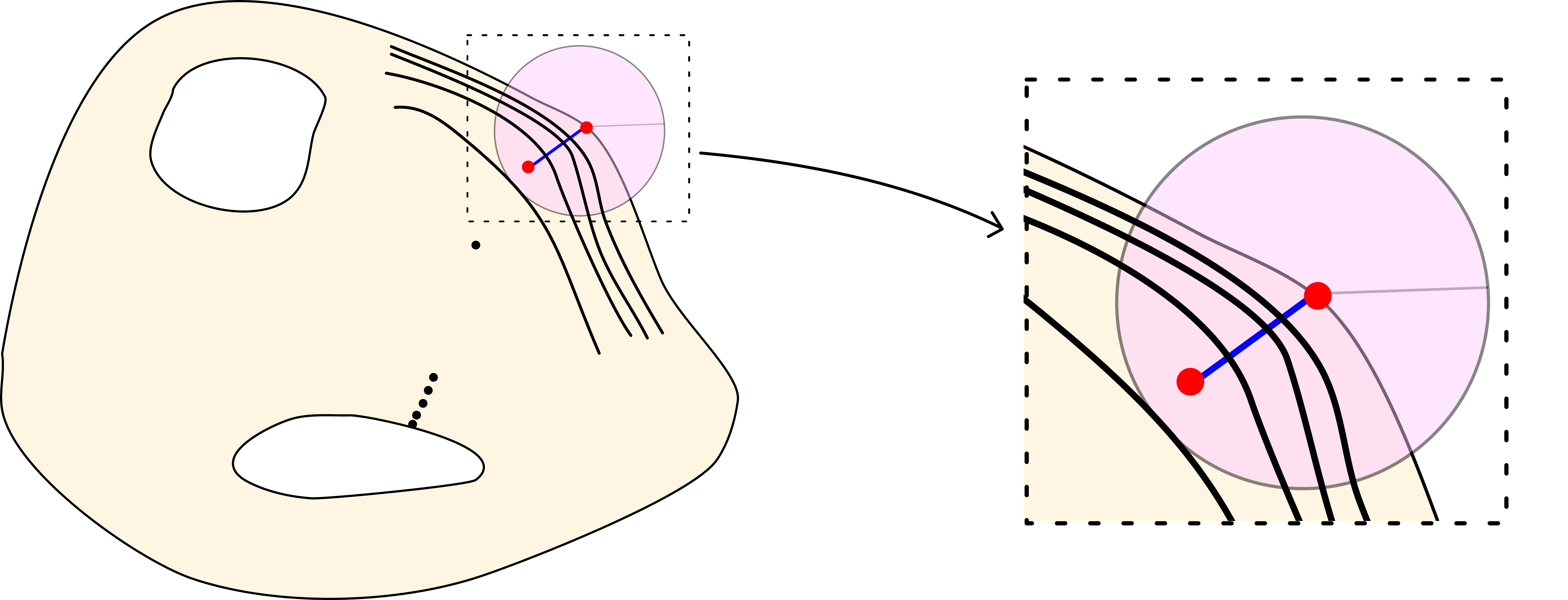}
		\setlength{\unitlength}{14cm}
		\put(-0.95, 0.35){$ \Omega $}
			\put(-0.16, 0.205){$ x $}
				\put(-0.265, 0.13){$ z$}
					\put(-0.11, 0.2){$ r $}
		\caption{\footnotesize A sketch of the proof of Lemma \ref{lem:dense} to find an accessible point in $ D(x,r) $, with a zoom on the right.}\label{fig-accessible}
	\end{figure}
\end{proof}

Since our domains are now multiply connected, it makes sense to define what an accessible boundary component is.

\begin{defi}{\bf (Accessible boundary component)}
	We say that a boundary component $ \Sigma\subset \partial \Omega $ is {\em accessible} from $ \Omega $ if there exists a curve $ \eta \colon \left[0, 1\right) \to \Omega$  whose landing set is contained in $ \Sigma $, i.e. $ L(\eta)\subset \Sigma $.
\end{defi}

We now turn to the question of how to relate accessible points (and accessible boundary components)  to points on the unit circle by means of the universal covering. We start with a simple result mimicking Lindel\"of's Theorem \ref{thm-lindelof} for simply connected domains, which is a more complete version of \ref{teo:C} in the introduction.
\begin{prop}
	{\bf (Accessible points and radial limits)}\label{prop:access}
	Let $ \Omega $ be a multiply connected domain, and let $ \pi\colon\mathbb{D}\to\Omega $ be the universal covering of $ \Omega $. Then, all boundary components $ \Sigma\subset \partial \Omega $ are accessible from $ \Omega $.	More precisely, 
	there exists $ e^{i\theta}\in \partial \mathbb{D} $ escaping such that $ Cl_\mathcal{A}(\pi, e^{i\theta})\subset\Sigma $.
	
	\noindent Moreover, a point $ p\in\partial \Omega $ is accessible from $ \Omega $ if and only if there exists $ e^{i\theta}\in \partial \mathbb{D} $ such that $ \pi^* (e^{i\theta}) =p$. Such $ e^{i\theta} $ is escaping.
	\end{prop}
	\begin{proof}
	The fact that all boundary components are accessible from $ \Omega $ is a straightforward consequence from Theorem \ref{thm-main-result-complete}, items \ref{item:MTescaping}(i) and \ref{item:MTesc-countable}.
	
	For the second statement, it is clear that, if  there exists $ e^{i\theta}\in \partial \mathbb{D} $ such that $ \pi^* (e^{i\theta}) =p$, then $ p $ is accessible. For the converse, assume $ p$ is accessible from $ \Omega $, and let $ \eta\colon \left[ 0,1\right) \to\Omega $, $ \eta(0)=\pi(0) $, be a curve landing at $ p $. Let $ \nu\colon \left[ 0,1\right) \to\mathbb{D} $ be the lift of $ \eta $ starting at $ \nu(0)=0 $. Then, $ \nu (t)\to\partial\mathbb{D} $ as $ t\to 1^{-} $.
	
	We claim that in fact there exists $ e^{i\theta} \in\partial\mathbb{D}$ such that $ \nu (t)\to e^{i\theta} $ as $ t\to 1^- $. To do so, consider an Ohtsuka exhaustion $ \left\lbrace \Omega_n\right\rbrace _n $ of $ \Omega $, and let $\left\lbrace \sigma_{s_1s_2\dots s_n}\right\rbrace _n$ be the boundary curves isolating $ \Sigma $ (i.e., $\Sigma$ has address $s_1s_2\dots s_n$); assume $ \eta $ intersects each $\left\lbrace  \sigma_{s_1s_2\dots s_n}\right\rbrace _n $ at precisely one point (we can, for instance, exchange $\Omega_n$ by any homotopically equivalent exhaustion; recall Lemma \ref{lemma-equivalence-ohtsuka}). Denote by $ \left\lbrace C_{s_1s_2\dots s_n}\right\rbrace _n $ the corresponding crosscuts in $ \mathbb{D} $ which intersect $ \nu $, and by $ \left\lbrace N_{s_1s_2\dots s_n}\right\rbrace _n $ their respective crosscuts neighbourhoods (see Figure \ref{fig-lindelof}).
	
	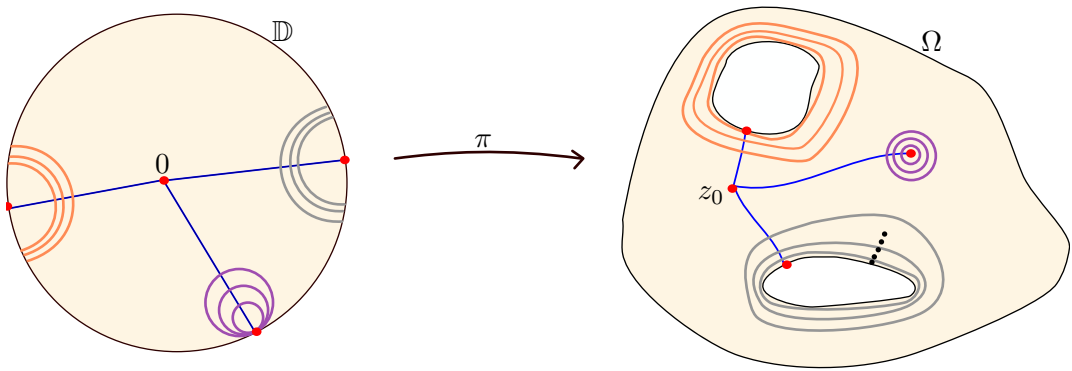
\begin{figure}[h]
		\includegraphics[width=14cm]{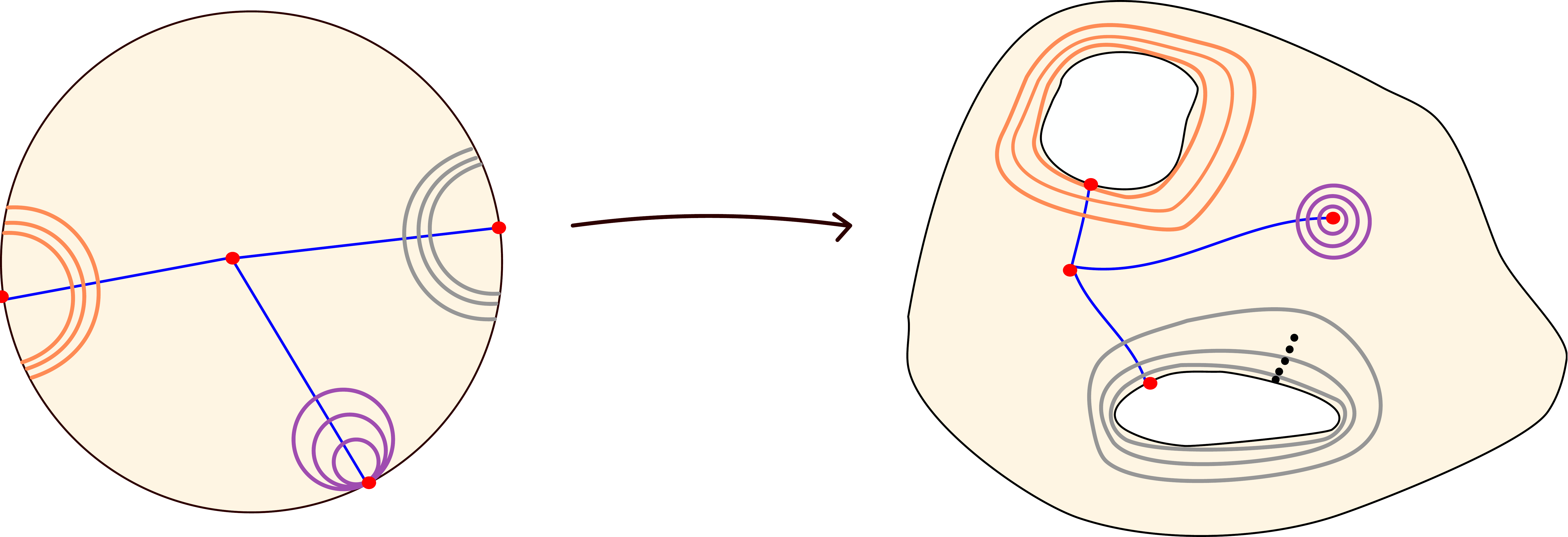}
		\setlength{\unitlength}{14cm}
		\put(-0.56, 0.21){$ \pi $}
		\put(-0.75, 0.31){$ \mathbb{D} $}
		\put(-0.14, 0.3){$ \Omega $}
		\put(-0.35, 0.16){$ z_0$}
		\put(-0.86, 0.185){$ 0$}
		\caption{\footnotesize Idea of the construction in Proposition \ref{prop:access}.}\label{fig-lindelof}
	\end{figure}

Observe that \[L(\nu)\subset\bigcap\limits_{j}\overline{N_{s_0s_1\dots s_j}}.\]
	
	Two possibilities arise. First, if the previous intersection of crosscut neighbourhoods consists of a point $ e^{i\theta} $, then obviously $ \nu (t)\to e^{i\theta} $,  as $ t\to 1^-$. Second, if the previous intersection (which is an $ \alpha $-image for $ \Sigma $) is a circular interval $ I\subset\partial\DD $, then, all the interior points of $ I $ are regular for $ \pi $, and hence $\Lambda_{NT}$ has Lebesgue measure zero. In particular, $ \pi^* $ exists for Lebesgue almost every point in $ I $ by \ref{teo:B}(d). If $ L(\nu) $ was a non-degenerate continuum in $ I $, then the radial limit $ \pi^* $ would not exist for all points in $ L(\nu) $, a set of positive measure, leading to a contradiction. Thus, there exists $ e^{i\theta} \in\partial\mathbb{D}$ with $ \nu (t)\to e^{i\theta} $,  when $ t\to 1 ^-$, as claimed.
	
	Therefore, we have that $ \nu (t)\to e^{i\theta} $ and $ \pi(\nu (t))=\eta(t)\to p $,  as $ t\to 1^- $. By the Lehto-Virtanen Theorem  \ref{thm-lehto-virtanen}, this already implies $ \pi^* (e^{i\theta}) =p$, as desired.
	Finally, since radial limits only exist for escaping points (\ref{teo:A}(c)), it is clear that such $ e^{i\theta} $ must be escaping.
	\end{proof}

Note that any boundary component $ \Sigma $ of $ \partial\Omega $ is a continuum in $ \widehat{\mathbb{C}} $ (i.e. a compact connected metric space). Then, it is either degenerate, if it consists of a single point, or non-degenerate, otherwise. Thus, as a straightforward corollary of the previous proposition, we get that all degenerate boundary components are accessible points.
\begin{cor}{\bf (Degenerate boundary components are accessible)}\label{cor:degenerate}
	Let $ \Omega $ be a multiply connected domain, and let $ \Sigma =\left\lbrace p\right\rbrace  $ be a boundary component. Then, $ p $ is an accessible  point on $ \partial \Omega$.
\end{cor}
\begin{proof}
	Applying Proposition \ref{prop:access}, we get that there exists $ e^{i\theta}\in \partial \mathbb{D} $ escaping such that $ Cl_R(\pi, e^{i\theta})\subset\Sigma $. But, since $ \Sigma =\left\lbrace p\right\rbrace  $ it follows that $ \pi^*(e^{i\theta})$ exists and equals $p $, as desired.
\end{proof}

However, notice that, although all boundary components $ \Sigma\subset \partial \Omega $ are accessible from $ \Omega $, they may  have no accessible points, as shown by the following examples.

\begin{ex}\label{ex-radial-limit-exists}
	Let
	\[	\Omega= \Chat\smallsetminus\left( \bigcup_n \left\lbrace z\in\mathbb{D}\colon \ \left| z\right| = \frac{1}{n}, \ \left| \textrm{Arg }((-1)^nz)\right| <\frac{(n-1)\pi}{n}\right\rbrace \cup \left\lbrace 0\right\rbrace \right);  \]
	see Figure \ref{subfig:expoint}. Then, $0$ (shown in red) is a boundary component of $\Omega$. By Corollary \ref{cor:degenerate}, it is an accessible boundary point of $\Omega$, despite being hidden behind infinitely many reefs. Of course, the \textit{probability} that a random path in $\Omega$ hits $0$ is zero.
\end{ex}

\begin{ex}\label{ex-radial-limit-no-exists}
	This example can be seen as a deformation of the previous one. For $n\geq 1$, let $C_n$ denote the union of the three line segments $[-1/4 - (1 + i)/n, 1/4 + (1 - i)/n]$, $[1/4 + (1 - i)/n, 1/4 + (1 + i)/n]$ and $[1/4 + (1 + i)/n, -1/4 - (1 - i)/n]$. Then, the domain
	\[ \Omega = \Chat\setminus \left(\bigcup_n (-1)^nC_n \cup [-1/4, 1/4] \right) \]
	has the interval $I = [-1/4, 1/4]$ as a boundary component (red in Fig. \ref{subfig:exint}), and it is accessible by Prop. \ref{prop:access}. However, no point on $I$ is accessible, because a curve $\eta$ in $\Omega$ going towards a point $x\in I$ must avoid the infinitely many reefs $(-1)^nC_n$ accumulating on $I$; in particular $\eta$ must accumulate at every point of $I$.
\end{ex}
\begin{figure}[h]\centering
	\begin{subfigure}{0.45\textwidth}\centering
		\includegraphics[width=0.8\textwidth]{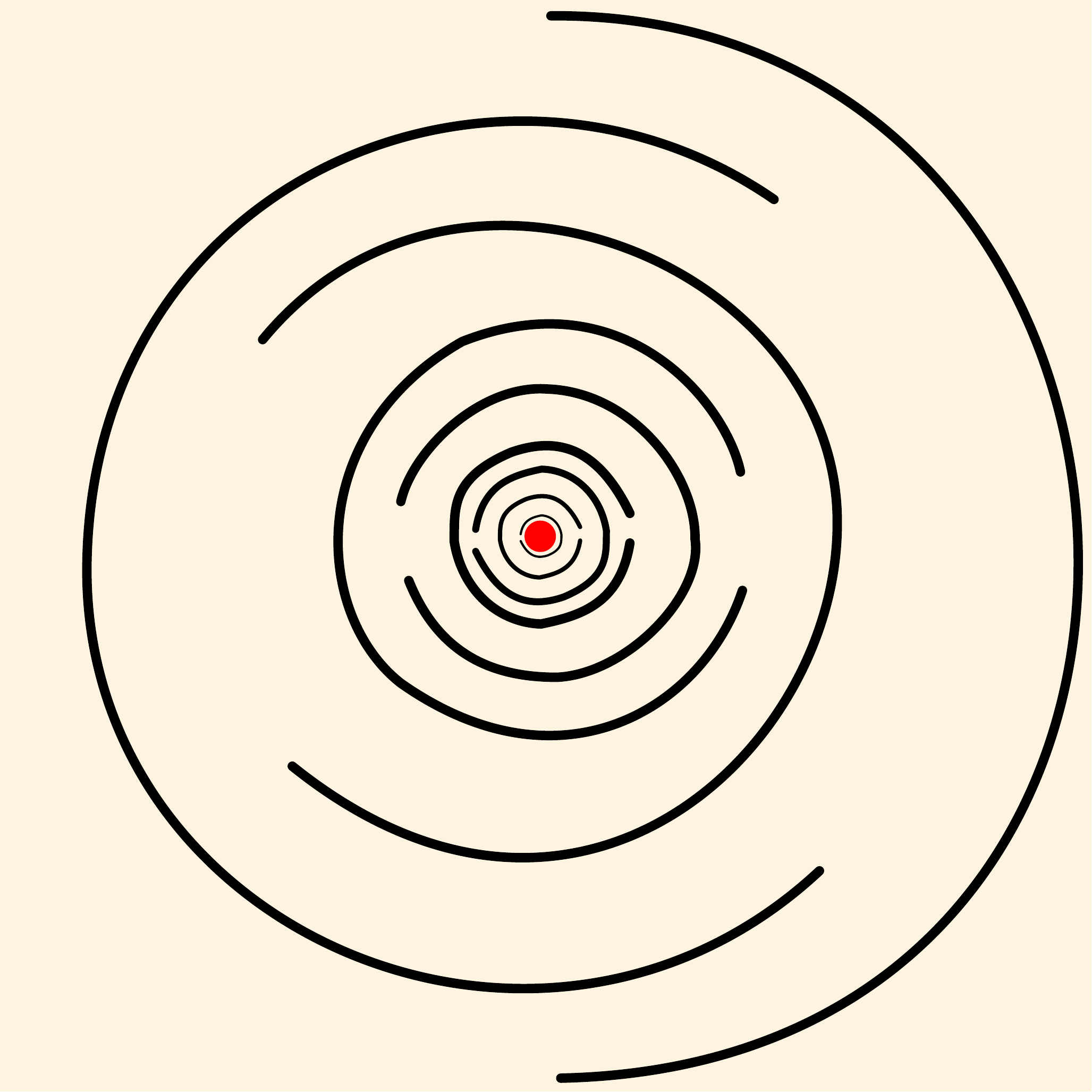}
		\setlength{\unitlength}{0.8\textwidth}
		\put(-0.98,0.94){$ \Omega $}
		\caption{\footnotesize The domain $\Omega$ in Example \ref{ex-radial-limit-exists}.}\label{subfig:expoint}
	\end{subfigure}
\hfill
	\begin{subfigure}{0.45\textwidth}\centering
		\includegraphics[width=0.8\textwidth]{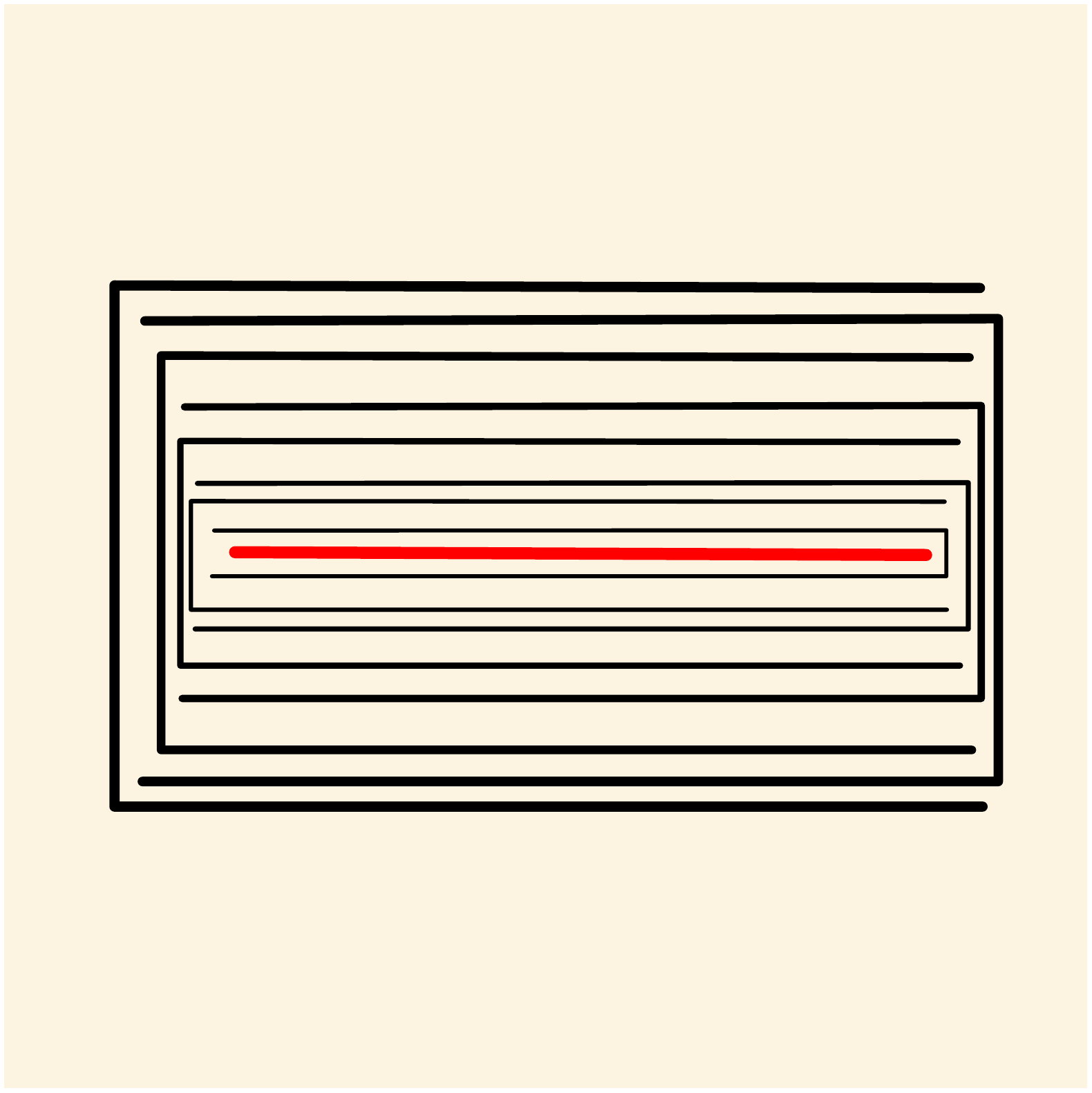}
		\setlength{\unitlength}{0.8\textwidth}
				\put(-0.98,0.94){$ \Omega $}
		\caption{\footnotesize The domain $\Omega$ in Example \ref{ex-radial-limit-no-exists}.}\label{subfig:exint}
	\end{subfigure}
\caption{\footnotesize Different accessible boundary components.}
\end{figure}

Finally, we develop a multiply connected analogue of the Correspondence Theorem  \ref{correspondence-theorem}. Our first task is to determine whether the definition of access to a point on $\partial\Omega$  (Def. \ref{def:access-sc}) remains reasonable -- a simple example shows that it does not.

Indeed, let $\mathcal{A}_r = \{z\in\Cx\colon 1/r < |z| < r\}$ for some $r > 1$. Then, the universal covering of $ \mathcal{A}_r  $ can be thought as the exponential map  defined on the strip $S = \{z\in\Cx\colon |\textrm{Re} z| < \log r\}$, \[\exp\colon S\longrightarrow \Omega, \]  mapping $0$ to $1$ and the imaginary axis onto the unit circle (which is the unique simple closed geodesic of $\Omega$). Let $\eta_1$ and $\eta_2$ be the Jordan curves joining $1\in\Omega$ to $r\in\partial\Omega$. One would expect that they should define the same acccess to the point $r$, but that is clearly not the case according to Definition \ref{def:access-sc}. In fact, there are countably many different {\em accesses} to any point on $\partial\Omega$, corresponding to how many times a representative curve winds around the hole in $\Omega$.  Compare with Figure \ref{fig:access-annulus}.

Moreover, when working with the universal covering, another ambiguity appears: each of the curves $\eta_i$, $i=1,2$, has infinitely many lifts to $S$, each starting at preimage $\tilde z_n = 2n\pi i$, $n\in\mathbb{Z}$ of $1$, and terminating at a preimage $\tilde w_n = \log r + 2n\pi i$, $n\in\mathbb{Z}$, of $r$. Since universal coverings have the path-lifting property, this can be solved by fixing a preimage of the basepoint $1\in\Omega$ and considering all the unique lift of each curve starting at that point.

We see that modifications must be made to the definition of access. Specifically, not only two homotopic curves (with fixed endpoints) must define the same access, but, moreover, two non-homotopic curves which land at the same boundary point and only differ on adding (concatenating) a closed curve must define also the same access.  Morally, we must have some way of ignoring finitely many holes in $\Omega$. However, we cannot ignore infinitely many  holes at once. Indeed, as shown in  Figure \ref{fig:access-infinity}, there may exist two different accesses to $x\in\partial \Omega$ which differ in infinitely many holes, and thus are unrelated by a single element of the fundamental group (note that the two curves do not differ one from the other by concatenating a loop). 

\begin{figure}[h]
	\begin{subfigure}{0.45\textwidth}\centering
		\includegraphics[width=0.8\textwidth]{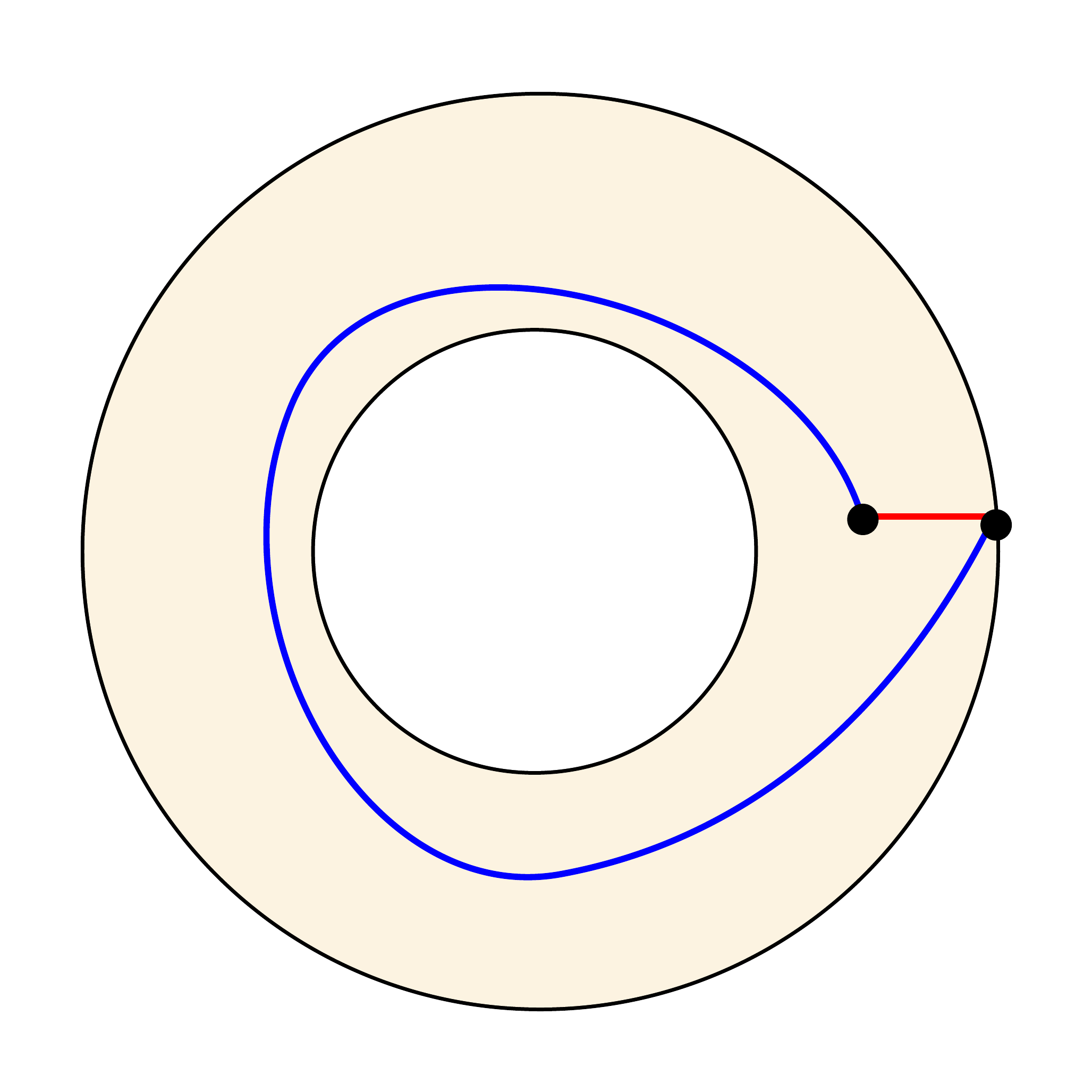}
		\setlength{\unitlength}{0.8\textwidth}
			\put(-0.85,0.8){$ \Omega $}
		\caption{\footnotesize The two curves in the annulus should define the same access. Notice also that the preimages of said curves are related to each other in non-trivial ways.}\label{fig:access-annulus}
	\end{subfigure}
\hfill
	\begin{subfigure}{0.45\textwidth}\centering
		\includegraphics[width=0.8\textwidth]{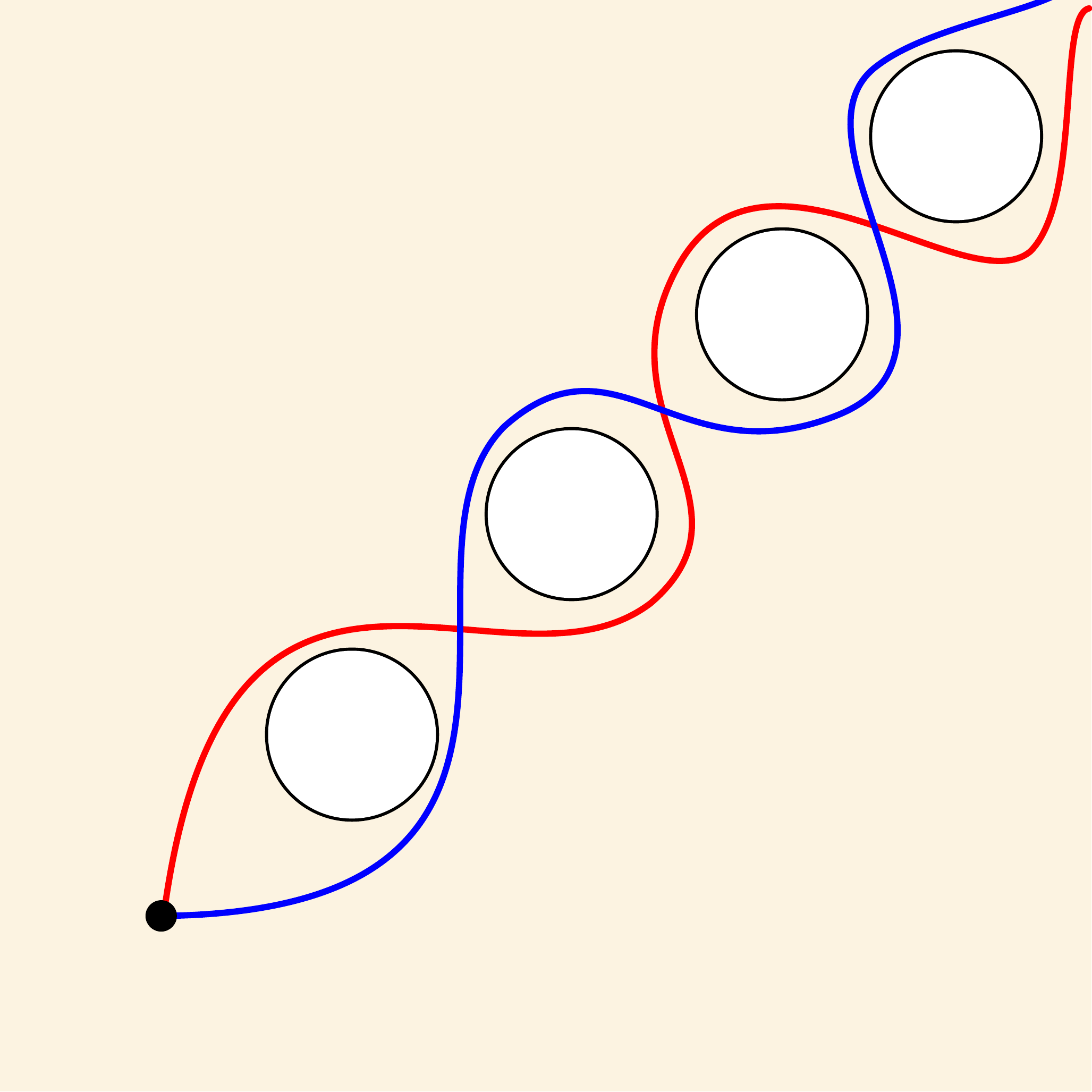}
		\setlength{\unitlength}{0.8\textwidth}
			\put(-0.98,0.94){$ \Omega $}
		\caption{\footnotesize The two curves in this domain should define different accesses to $ \infty\in\partial\Omega $, where $\Omega$ is $\Cx$ minus a disjoint union of disks accumulating at infinity.}\label{fig:access-infinity}
	\end{subfigure}
\caption{\footnotesize Considerations on the definition of access in a multiply connected domain.}
\end{figure}\label{fig-acc}

We land\footnote{Pun intended.} at the following definition of access, which takes into account the relations induced by the fundamental group.
\begin{defi}{\bf (Access for multiply connected domains)}\label{def:access-mc}
Let $\Omega$ be a multiply connected domain, and let $z_0\in\Omega$. Let $p\in\partial\Omega$ be an accessible point. We say that two curves $\eta_{1,2}\colon[0, 1)\to\Omega$ with $\eta_1(0) = \eta_2(0) = z_0$ and $\lim_{t\to1^-} \eta_1(t) = \lim_{t\to 1^-} \eta_2(t) = p$ are {\em equivalent} (as curves  in $\Omega$ joining $z_0$ to $p$) if there exists a continuous function $H\colon[0, 1]\times[0, 1)\to\Omega$ such that:
\begin{enumerate}[label={\normalfont (\alph*)}]
	\item For all $s\in[0, 1]$, $\lim_{t\to 1^-} H(s, t) = p$;
	\item $H(0, t) = \eta_1(t)$ and $H(1, t) = \eta_2(t)$.
\end{enumerate}
An {\em access} to $p$ (from $z_0$) is an equivalence class of curves in $\Omega$ joining $z_0$ to $p$. The set of all accesses to $p$ from $z_0$ is denoted $\mathcal{A}(p; z_0)$.
\end{defi}
In other words, an access to $p$ (from $z_0$) is an equivalence class of curves in $\Omega$ joining $z_0$ to $p$, where curves are said to be equivalent if they are homotopic \emph{through a homotopy that fixes only $p$}. Notice that both curves in Figure \ref{fig:access-annulus} are equivalent, while the curves in \ref{fig:access-infinity} are not. 

Furthermore, the choice of the basepoint $ z_0 $ is irrelevant. Indeed, for any two points $z_0$ and $z_1$ in $\Omega$, any curve joining $z_0$ to $z_1$ gives a well-defined function $\mathcal{A}(p; z_0)\to\mathcal{A}(p; z_1)$; conversely, any access to $p$ from $z_1$ can be homotopically deformed to pass through $z_0$, defining an inverse function $\mathcal{A}(p; z_1)\to\mathcal{A}(p; z_0)$. Therefore, accesses to $p$ from any point in $\Omega$ are in a one-to-one correspondence, and so we can speak of \emph{accesses to $p$}. However, for the purpose of stating and proving our results, it is often more convenient to make explicit reference to a basepoint. 

We note that if $\eta_{1,2}\colon[0, 1)\to\mathbb{D}$ with $\eta_1(0) = \eta_2(0) = 0$ and $\lim_{t\to1^-} \eta_1(t) = \lim_{t\to 1^-} \eta_2(t) = e^{i\theta}\in\partial\mathbb{D}$ non-tangentially, then $ \pi(\eta_{1}) $ and $ \pi(\eta_{2}) $ are homotopic in $ \Omega $, and thus {equivalent} (indeed, if $ H $ is the homotopy in $ \mathbb{D} $ between $ \eta_1 $ and $ \eta_2$, then $ \pi\circ H $ is the homotopy in $ \Omega $ between their images). On the other hand, if two curves $\eta_{1,2}\colon[0, 1)\to\Omega$, with $\eta_1(0) = \eta_2(0) = z_0$ and $\lim_{t\to1^-} \eta_1(t) = \lim_{t\to 1^-} \eta_2(t) = p\in\partial\Omega$ are homotopic (with fixed endpoints, i.e. $ H(s,0)=z_0 $ for all $ s\in \left[ 0,1\right]  $), then its unique lifts $ \widetilde{\eta}_{1,2} $ under the universal covering $ \pi\circ\mathbb{D}\to\Omega $, $ \pi(0)=z_0 $, starting at 0, are homotopic (and thus land at the same $ e^{i\theta}\in\partial\mathbb{D} $). Moreover, the following is true.

\begin{lemma}{\bf (Characterization of equivalent curves)} \label{lemma-eq-curves} Let $\Omega$ be a multiply connected domain, and let $z_0\in\Omega$. Let $p\in\partial\Omega$ be an accessible point. Two curves $\eta_{1,2}\colon[0, 1)\to\Omega$ starting at $z_0$ and landing at $p$ are equivalent if and only if there exists a closed curve $ \sigma$ such that $ \eta_1 $ is homotopic in $ \Omega $ to $  \sigma+ \eta_2$ (where addition denotes concatenation of curves).
\end{lemma}
\begin{proof}
One implication is a straight-forward consequence from the definition of equivalent curves, noting that \[\left\lbrace H(s,0)\colon s\in \left[ 0,1\right] \right\rbrace \] is a loop in $ \Omega $. For the other implication, note that, $ \eta_1 $ is homotopic in $ \Omega $ to $  \sigma+ \eta_2$ by a homotopy $ H $, the same $ H $ defines an equivalence between $ \eta_1 $ and $ \eta_2 $.	See Figure \ref{fig-def-acc}.
\end{proof}
\begin{figure}[h]
	\includegraphics[width=14cm]{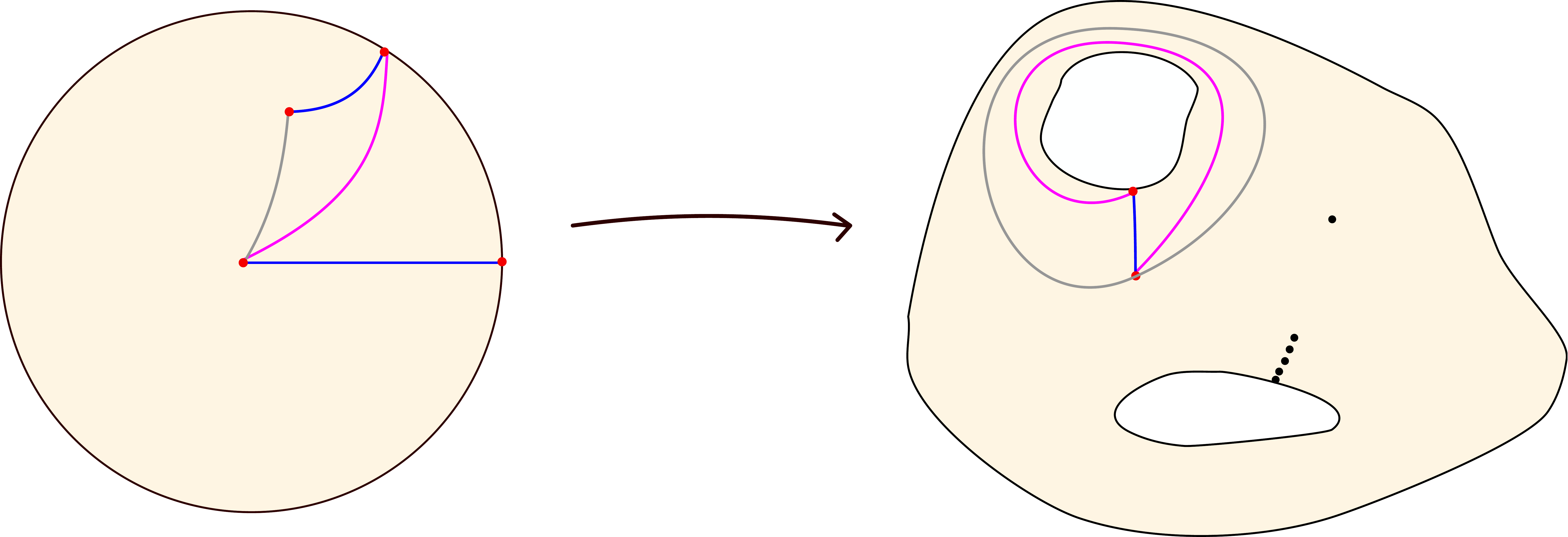}
	\setlength{\unitlength}{14cm}
	\put(-0.56, 0.21){$ \pi $}
	\put(-0.75, 0.31){$ \mathbb{D} $}
	\put(-0.14, 0.3){$ \Omega $}
	\put(-0.28, 0.145){$ z_0$}
	\put(-0.85, 0.15){$ 0$}
	\caption{\footnotesize Two equivalent curves in $ \Omega $, and their lifts in the unit disk $ \mathbb{D} $.}\label{fig-def-acc}
\end{figure}
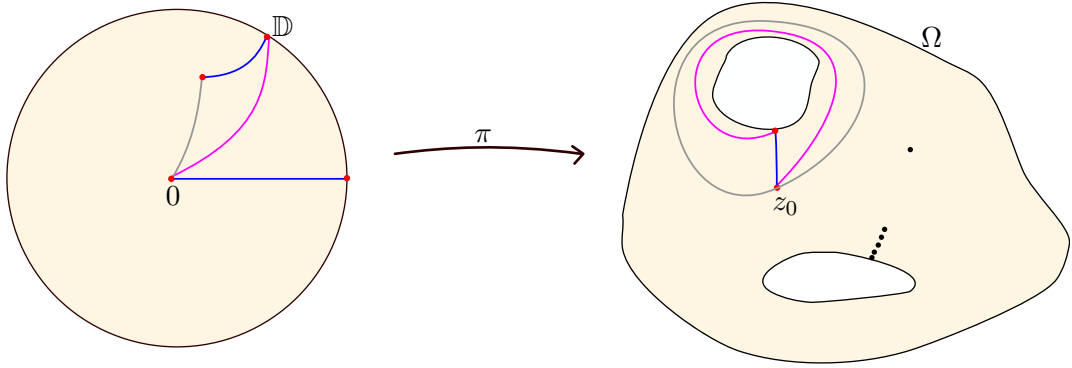

With this in mind, we have the following result (compare with \cite[Thm. 4]{Oht54}, which states less precisely a weaker correspondence in the same spirit).
\begin{thm}{\bf(Correspondence Theorem for multiply connected domains)}\label{thm:correspondence-mc}
Let $\Omega$ be a hyperbolic multiply connected domain, let $z_0\in\Omega$, and let $p\in\partial\Omega$. Let $\pi\colon\DD\to\Omega$ be a universal covering such that $\pi(0) = z_0$. Then,
\[ (\pi^*)^{-1}(p)/\Gamma \simeq \mathcal{A}(p; z_0). \]
More precisely, the following hold.
\begin{enumerate}[label={\normalfont (\alph*)}]
	\item If $\pi^*(e^{i\theta}) = p$, then $\pi^*(\gamma(e^{i\theta})) = p$ for every $\gamma\in\Gamma$.
		\item If $\pi^*(e^{i\theta}) = p$, there exists an access $A$ to $p$ from $z_0$ such that, for any $\gamma\in\Gamma$, any curve $\eta\subset\DD$ starting at $0$ and landing non-tangentially at $\gamma(e^{i\theta})$ satisfies $\pi(\eta)\in A$ (in particular, $\pi(\eta)$ lands at $p$).
	\item If $A$ is an access to $p$ from $z_0$, then there exists a point $e^{i\theta}\in\partial\DD$ such that:
	\begin{enumerate}[label={\normalfont (\roman*)}]
		\item for any curve $\eta\in A$, its (unique) lift $\tilde\eta\subset\pi^{-1}(\eta)$ starting at $0$ lands at $\gamma(e^{i\theta})$ for some $\gamma\in\Gamma$;
		\item if $B\in\mathcal{A}(p; z_0)$ is distinct from $A$ with corresponding point $e^{i\alpha}\in\partial\DD$, then there does not exists a deck transformation $ \gamma\in \Gamma $ with $ \gamma(e^{i\theta})=e^{i\alpha} $.
	\end{enumerate}
\end{enumerate}
\end{thm}
\begin{proof}
Proposition \ref{prop:access} guaranteed that an access $A$ to $p$ from $z_0$ yields a point $e^{i\theta}\in\partial\DD$ such that $\pi^*(e^{i\theta}) = p$, and vice-versa. Hence, we are left to show that this correspondence is well-defined up to homotopy (for accesses) and $\Gamma$-orbit (for points in $\partial \DD$) -- that is to see that items (a)-(c) hold.

First, we show that if $\pi^*(e^{i\theta}) = p$, then $\pi^*(\gamma(e^{i\theta})) = p$ for every $\gamma\in\Gamma$. Indeed, for any deck transformation $\gamma$, consider the curve $\gamma(R_\theta)$. Since $\gamma$ is a deck transformation,  it is clear that $\pi(\gamma(R_\theta))$ lands at $p$, and also that $\gamma(R_\theta)$ lands non-tangentially at $\gamma(e^{i\theta})$ (because $\gamma$ is conformal on $\overline{\DD}$) . By Theorem \ref{thm-lehto-virtanen}, $\pi^*$ has non-tangential limit $p$ at $\gamma(e^{i\theta})$, as promised, finishing the proof of (a).

Next, we prove (b). Let $e^{i\theta}\in\partial\DD$ be such that $\pi^*(e^{i\theta}) = p$; then, the image $\pi(R_\theta)$ defines an access $A$ to $p$ from $z_0$, and any curve $\eta\subset\DD$ starting at $0$ and landing non-tangentially at $e^{i\theta}$ satisfies $\pi(\eta)\in A$ (because any such curve is homotopic to $R_\theta$). 

We are left to show that we can shift the endpoint of such a curve to $\gamma(e^{i\theta})$ while still defining the same access. To that end, take any $\gamma\in\Gamma$ and consider a curve $\eta\subset\DD$ starting at $0$ and landing non-tangentially at $\gamma(e^{i\theta})$. Then, the curve $\gamma^{-1}(\eta)$ starts at $\gamma^{-1}(0)$ and lands non-tangentially at $e^{i\theta}$, and satisfies $\pi(\gamma^{-1}(\eta)) = \pi(\eta)$ since $\gamma$ is a deck transformation. We join $0$ to $\gamma^{-1}(0)$ by a line segment $\ell$, defining a new curve $\eta' = \ell + \gamma^{-1}(\eta)$ (where addition means concatenation of curves).  Its image $\pi(\eta')$ consists of the concatenation of a closed curve $\pi(\ell)\subset\Omega$ through $z_0$  with $\pi(\eta)$. By Lemma \ref{lemma-eq-curves}, $\pi(\eta')$ and $\pi(\eta)$ are equivalent, and thus define  the same access from $ z_0 $ to $ p $, as desired.

To prove (c)(i), let $A\in\mathcal{A}(p; z_0)$, and let $\eta\in A$. Then, the unique lift $\tilde\eta$ of $\eta$ starting at $0$ lands at some $e^{i\theta}\in\partial \DD$ where $\pi$ has non-tangential limit $p$ by Proposition \ref{prop:access}. Take $\eta'\in A$, and we have to show that its unique lift starting at $0$ lands at $\gamma(e^{i\theta})$, for some $\gamma\in\Gamma$. Since $ \eta $ and $ \eta '$ are equivalent, Lemma \ref{lemma-eq-curves}, 
there exists   a closed curve $\sigma$ through $z_0$ such that $\sigma + \eta$ is homotopic to $ \eta' $. The loop  $ \sigma $  corresponds to some deck transformation $\gamma^{-1}$ (by Lemma \ref{lem:deck}).  Again by Lemma \ref{lem:deck}, the curve $\sigma + \eta$ has a unique lift $\widetilde{\sigma + \eta}$ joining $\gamma^{-1}(0)$ to $0$ by a lift of $\sigma$, and then joining $0$ to $e^{i\theta}$ through the original lift $\tilde\eta$. We apply $\gamma$ to obtain $\gamma(\widetilde{\sigma + \eta})$, which joins $0$ to $\gamma(e^{i\theta})$ and satisfies $\pi(\gamma(\widetilde{\sigma + \eta})) = \sigma + \eta$. Since $\sigma + \eta$ is homotopic to $\eta'$, the homotopy lifting property of $\pi$ implies that the lift $\tilde\eta'$ of $\eta'$ starting at $0$ also lands at $\gamma(e^{i\theta})$.

Finally, Theorem \ref{thm:correspondence-mc}(c)(ii) follows immediately from items (b) and (c)(i) of the same theorem.
\end{proof}

\subsection{Prime ends for a multiply connected domain}\label{subsect-prime-ends}
Our goal is to define prime ends for multiply connected domains. In the same spirit as in the simply connected case, we aim for a compactification of $ \Omega $ that is compatible with the radial extension of the universal covering $ \pi\colon\mathbb{D}\to \Omega$. 

As explained in the introduction, even though several theories of prime ends have been developed, none of them interacts successfully with $ \pi\colon\mathbb{D}\to \Omega$. Presumably, the best approach in this sense is made by 
Ohtsuka  \cite[Sect. 1]{Oht54}, who does obtain a compactification of the boundary, supposedly well-behaved under $ \pi $, relying on the exhaustions introduced before. However, it is quite crude in the sense that it does not distinguish between points in the same boundary component. In particular, it does not agree with Carathéodory's prime end theory for simply connected domains. Moreover, his approach depends heavily on the exhaustion he considers. Hence, we present here a different approach, which substantially refines Ohtsuka's work, and improves the understanding of the connection with the universal covering.

As in the simply connected case, we will define a prime end as an equivalence class of chains of crosscuts, and its impression as the intersections of the images of the corresponding crosscut neighbourhoods. However, some modifications have to be done in the multiply connected case, since approaching $ \partial \mathbb{D} $ does not always correspond to approaching $ \partial \Omega$, as we have shown previously. Moreover, the cluster set of $ \pi  $ at any point in the limit set is $ \overline{\Omega} $. 

As the reader may guess, we must consider only admissible null-chains (recall Def. \ref{def-admissible-crosscut-null-chain}) for points of escaping type and disregard some part of the corresponding crosscut neighbourhoods, in order to get a proper definition of prime end. More precisely, we define the following.
\begin{defi}{\bf (Rectified crosscut neighbourhood)}\label{def-true-crosscut-nhd}
Let $\Omega\subset\Chat$ be a hyperbolic domain, and let $\pi\colon\DD\to\Omega$ be a universal covering with deck transformation group $\Gamma$. Let $ e^{i\theta}\in\partial\mathbb{D} $ be of escaping type, and let $ C $ be an admissible crosscut at $ e^{i\theta }$. We define the {\em rectified crosscut neighbourhood} $ N_C $ of $ C $ as the connected component of \[\mathbb{D}\smallsetminus \Gamma(C)\] such that $  e^{i\theta}\in\overline{N_C}$.
\end{defi}

Definition \ref{def-true-crosscut-nhd} requires some comments. First, notice that, by requiring that $e^{i\theta}$ be of escaping type, we avoid ``pathological'' cases such as the existence of a sequence $\{C_n\}_n\subset \Gamma(C)$ such that $C_{n+1}$ separates $e^{i\theta}$ from $C_n$. In other words, $N_C$ is always well-defined. Second (and this is very important), the rectified crosscut neighbourhood of a regular point coincides with its crosscut neighbourhood as defined in Definition \ref{def:cc-nbhd-stolz}, but the same is false for rectified crosscut neighbourhoods of singular points. Third, by Theorem \ref{thm-main-result-complete}, the requirement of $e^{i\theta}$ being of escaping type also guarantees that, given any crosscut $C$ separating $e^{i\theta}$ from $0$, there exists another crosscut $C'\in\Gamma(C)$ (which may be equal to $C$) such that $C'\subset \partial N_{C'}$, i.e. one can (for our purposes) always assume that a crosscut is part of the boundary of its rectified crosscut neighbourhood. Finally, note that, by construction, any connected component of $\pi^{-1}(D_{s_0s_1\dots s_j})$ is a rectified crosscut neighbourhood for the points in the unit circle which lie in its boundary. 

We can now give a precise definition of prime ends and their impressions.

\begin{thm}{\bf (Prime ends and impressions)}\label{thm-prime-ends-MC} Let $ \Omega $ be a multiply connected domain, and let $ \pi\colon\mathbb{D}\to\Omega $ be the universal covering. Let $ e^{i\theta} \in\partial\mathbb{D}$ be a point of escaping type with associated boundary component $ \Sigma $. Then, we define its {\em prime end} $ P(\pi, e^{i\theta}) $  as the equivalence class of admissible null-chains for $ \pi $ at $ e^{i\theta} $.

	\noindent Given an admissible null-chain $ \left\lbrace C_n\right\rbrace _n $ for $ \pi $ at $ e^{i\theta} $, with rectified crosscut neighbourhoods $ \left\lbrace N_n\right\rbrace _n $, we define the {\em impression} of the prime end $ P(\pi, e^{i\theta}) $ as
	\[I (P(\pi, e^{i\theta}))= \bigcap\limits_n \overline{\pi (N_n)}\subset \Sigma,\] which does not depend on the chosen admissible null-chain.

	\noindent Moreover, exactly one of the following statements holds.
	\begin{enumerate}[label={\normalfont (\alph*)}]
		\item $ e^{i\theta} \in\partial\mathbb{D}\smallsetminus\Lambda$ if and only if for every admissible null-chain $ \left\lbrace C_n\right\rbrace _n $ for $ \pi $ at $ e^{i\theta} $, with rectified crosscut neighbourhoods $ \left\lbrace N_n\right\rbrace _n $, we have that $ \pi(N_n) $ is eventually simply connected. We say that the prime end $P(\pi, e^{i\theta})$ is {\em regular}, and in this case
		\[ I(P(\pi, e^{i\theta}))=Cl(\pi, e^{i\theta})\subset \Sigma.\]
	\item $ e^{i\theta} \in\Lambda$, but is not a parabolic fixed point, if and only if for every admissible null-chain $ \left\lbrace C_n\right\rbrace _n $ for $ \pi $ at $ e^{i\theta} $, with rectified crosscut neighbourhoods $ \left\lbrace N_n\right\rbrace _n $, we have that $ \pi(N_n) $ is infinitely connected for all $ n\geq 0 $. We say that the prime end $P(\pi, e^{i\theta})$ is {\em singular}, and in this case
	\[ I(P(\pi, e^{i\theta}))= \Sigma, \hspace{0.5cm} Cl(\pi, e^{i\theta})=\overline{\Omega}.\]
		\item $ e^{i\theta} $ is a parabolic fixed point if and only if for every admissible null-chain $ \left\lbrace C_n\right\rbrace _n $ for $ \pi $ at $ e^{i\theta} $, $ C_n $ is a degenerate crosscut for every large $n$. We say that the prime end $P(\pi, e^{i\theta})$ is {\em parabolic}, and in this case $ \Sigma =\left\lbrace p\right\rbrace  $ and
		\[I(P(\pi, e^{i\theta}))=Cl_{\mathcal{A}}(\pi, e^{i\theta})=\left\lbrace p\right\rbrace, \hspace{0.5cm}Cl(\pi, e^{i\theta})=\overline{\Omega} .\]
	\end{enumerate}
\end{thm}

	\begin{figure}[h]\centering
	\includegraphics[width=14cm]{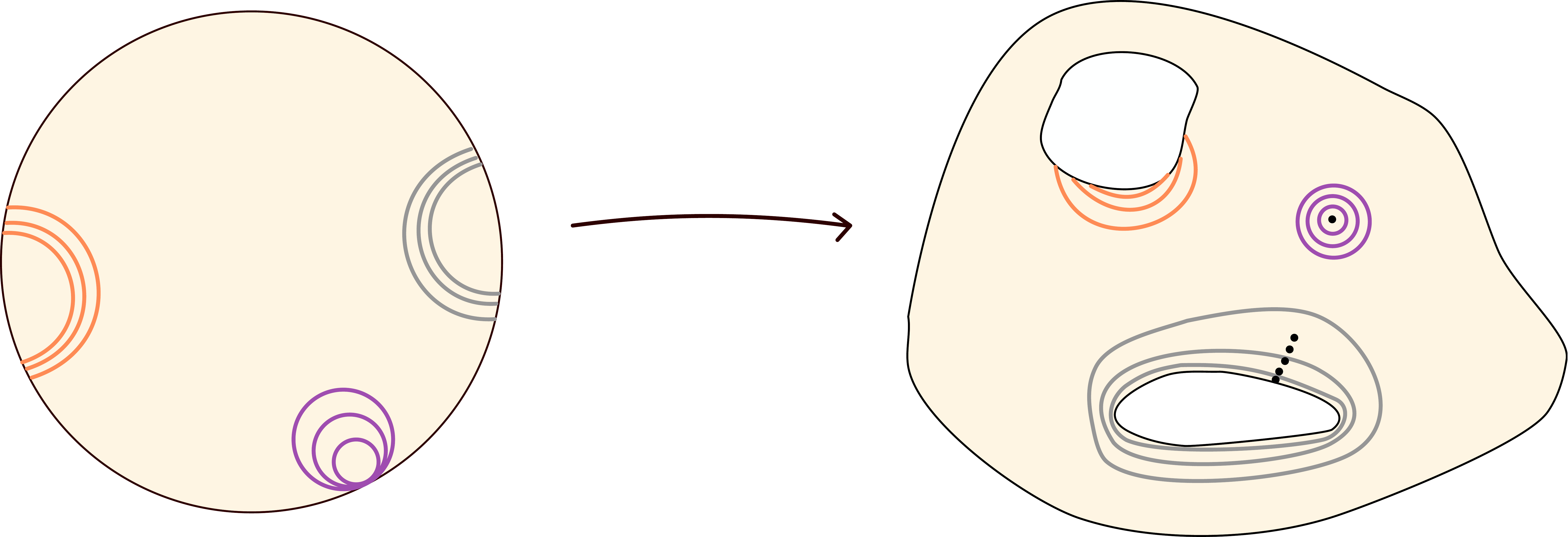}
	\setlength{\unitlength}{14cm}
		\put(-0.56, 0.21){$ \pi $}
	\put(-0.75, 0.31){$ \mathbb{D} $}
	\put(-0.14, 0.3){$ \Omega $}
	\caption{\footnotesize The different types of null-chains considered in Theorem \ref{thm-prime-ends-MC}: regular (orange), singular (grey) and parabolic (purple). Note that the crosscuts in the regular null-chain are true crosscuts in $ \Omega $, indicating that $ \Lambda\subsetneq \partial \mathbb{D} $ (\ref{cor:E}).}
\end{figure}

If a prime end $P(\pi, e^{i\theta})$ is regular (resp. singular, parabolic), then we say that any admissible null-chain at $e^{i\theta}$ is also regular (resp. singular, parabolic).

We postpone the proof of Theorem \ref{thm-prime-ends-MC} until the end of this section. We start by highlighting an easy consequence of our result, which allows us to determine whether the limit set is $ \partial\mathbb{D} $ by means of the geometry of $ \Omega $. 

Notice that, when $ \Omega $ is multiply connected, the notion of crosscut in $ \Omega $ (understood as an open Jordan arc with both endpoints in $ \partial \Omega $) is too crude to be useful in constructing the previous admissible null-chains directly in $ \Omega $. In this sense, the non-contractible curves given by an Ohtsuka exhaustion are more useful, and have to be understood as crosscuts in $ \Omega $  for the multiply connected case, taking into account that this will not yield admissible null-chains for regular points. To deal with such points, and to determine their existence by means of the geometry of $ \Omega $, we need the following definition.
\begin{defi}{\bf (True crosscut)}\label{defi-true-crosscut} 	Let $ \Omega $ be a multiply connected domain.
	We define a {\em true crosscut} in $ \Omega $ as an open Jordan arc $ C\subset\Omega $ such that $ \overline{C}=C\cup \left\lbrace a,b\right\rbrace   $, with $ a\neq b $, $ a,b \in \Sigma\subset\partial\Omega$, and such that $ \Omega\smallsetminus C $ has exactly two connected components, one of which simply connected.
\end{defi}

\begin{named}{Corollary E}{\bf (Geometric characterization of limit sets)}
	Let $\Omega\subset\Chat$ be a hyperbolic multiply connected domain, and let $\pi$ be its universal covering. Then, $\Lambda$ is a Cantor subset of $\partial\DD$ if and only if there exists a true crosscut in $\Omega$.
\end{named}
\begin{proof}
First, if $\Omega$ admits a true crosscut $\eta$, then any lift $\tilde\eta\subset \pi^{-1}(\eta)$ yields a non-degenerate crosscut in $\DD$. The corresponding crosscut  neighbourhood $N_{\tilde\eta}$ is mapped onto the simply connected domain $\Omega_\eta$ bounded by $\eta$ and $ \Sigma $. Therefore, $\pi\colon N_{\tilde\eta}\to \Omega_\eta$ is a conformal isomorphism. In particular, for all $ e^{i\theta}\in \partial N_{\tilde\eta}\cap\partial\mathbb{D} $, we have $ Cl(\pi, e^{i\theta}) \subsetneq \Omega $. By Lemma \ref{lemma-limit-sets3}, such $ e^{i\theta} $ is regular, and $\Lambda\neq \partial\mathbb{D} $, as desired.

On the other hand, if $e^{i\theta}\in\partial\DD$ is a regular point, then by Theorem \ref{thm-prime-ends-MC} any admissible null-chain at $e^{i\theta}$ maps to a chain of true crosscuts in $\Omega$.
\end{proof}

For example, if all the components of $ \partial\Omega $ are points (for instance, if $ \partial\Omega $ is a Cantor set), then $ \Lambda = \partial\mathbb{D} $, as was already pointed out in \cite{Fer89}. However, \ref{cor:E} shows that the same can be true even if $\partial\Omega$ contains continua: take, for instance, $\Omega$ to be the unit disc minus a sequence of points $\{z_n\}_n$ accumulating at every point of $\partial\DD$. Notice also that, by \ref{teo:B}, $\partial\Omega$ supports a harmonic measure, and (since  $\{z_n\}_n$ is a countable set of points) this measure is supported solely on $\partial\DD$. {Thus, even if  $\partial\DD$ is a `small' boundary component of $\Omega$, in the sense that all its $ \alpha $-images are points, it is still a big boundary component in the sense of measure.}

\subsection*{Proof of the Prime End Theorem \ref{thm-prime-ends-MC}}
We split the proof of Theorem \ref{thm-prime-ends-MC} in several steps. We start by dividing admissible null-chains into three cases.
\begin{lemma}{\bf (Characterization of admissible null-chains)}\label{lemma-technical-null-chains} Let $ \Omega $ be a multiply connected domain, and let $\pi\colon\mathbb{D}\to\Omega $ be a universal covering. Let $ e^{i\theta} \in\partial \mathbb{D}$ be of escaping type, and let  $ \left\lbrace C_n\right\rbrace _n $ be an admissible null-chain for $ \pi $ at $ e^{i\theta} $, with rectified crosscut neighbourhoods $ \left\lbrace N_n\right\rbrace _n $. Then, exactly one of the following holds.
	\begin{enumerate}[label={\normalfont (\alph*)}]
		\item There exists $ n_0\geq 0 $ such that $ \pi (N_n) $ is simply connected, for all $ n\geq n_0 $.
		\item  For all $ n\geq 0 $, $ C_n $ is a non-degenerate crosscut and $ \pi (C_n) $ is a simple closed curve.
		\item There exists $ n_0\geq 0 $ such that $ C_n $ is a degenerate crosscut with endpoints coinciding at $ e^{i\theta} $, for all $ n\geq n_0 $.
	\end{enumerate}
Moreover, the previous classification depends only on $ e^{i\theta} $, not on the chosen admissible null-chain.
\end{lemma}
\begin{proof}
First let us see that exactly one of the previous possibilities hold. First, if there exists $ n_0\geq 0 $ such that $ \pi (N_{n_0}) $ is simply connected, then $\pi\colon N_{n_0}\to\pi(N_{n_0})$ is a biholomorphism by Lemma \ref{lem:restriction}, and thus $ \pi (N_{n}) $ is simply connected for all $ n\geq n_0 $. This corresponds to (a). It is also clear that, if $ C_{n_0} $ is degenerate, then $ C_n $ is degenerate for all $ n\geq n_0 $ (and all these crosscuts should have the same endpoint). Thus, (b) holds whenever neither (a) nor (c) hold.

It is left to see that the previous classification does not depend on the chosen admissible null-chain. This relies on the fact that $ \pi (N_{n}) $ is eventually simply connected if and only if $ e^{i\theta} $ is regular (Lemma \ref{lemma-limit-sets3}), and $ C_{n} $ is eventually degenerate if and only if  $ e^{i\theta} $ is a parabolic fixed point (Lemma \ref{lemma-lifts-of-curves}(a)), and these properties clearly do not depend on the null-chain.
\end{proof}

As indicated in Section \ref{subsect-Riemann-map}, two null-chains (admissible or not) in the unit disk are either equivalent or eventually disjoint, and hence the equivalence classes of null-chains correspond precisely to the points in $ \partial\mathbb{D} $. Next we check that given two equivalent admissible null-chains (i.e. two admissible null-chains at the same point $ e^{i\theta}\in\partial\mathbb{D} $), then its impression is the same, so the impression of the prime end is well-defined.

\begin{lemma}{\bf (Prime ends are well-defined)}\label{lemma-impression}
Let $ \Omega $ be a multiply connected domain, and let $\pi\colon\mathbb{D}\to\Omega $ be a universal covering. Let $ e^{i\theta} \in\partial\mathbb{D}$ be of escaping type, with associated boundary component $ \Sigma $. Then,  given any admissible null-chain $ \left\lbrace C_n\right\rbrace _n $ for $ \pi $ at $ e^{i\theta} $, with rectified crosscut neighbourhoods $ \left\lbrace N_n\right\rbrace _n $, its {impression}
\[I (P(\pi, e^{i\theta}))=\bigcap\limits_n \overline{\pi (N_n)}\subset \Sigma,\] is well-defined, and does not depend on the chosen admissible null-chain.
\end{lemma}
\begin{proof}
	We analyse the cases in Lemma \ref{lemma-technical-null-chains} one by one, since they do not depend on our choice of admissible null-chain. For the first case, since $ \pi $ is univalent on the crosscut neighbourhoods $ \left\lbrace N_n\right\rbrace _n $, the statement follows from the one for Riemann maps (Thm. \ref{thm-prime-ends-sc}).  When $ C_n $ is eventually degenerate, this corresponds to a parabolic fixed point, and $ \pi(C_n) $ are simple loops around an isolated boundary point, which is the impression of the prime end. We are left with the case where $ C_n $ is a non-degenerate crosscut and $ \pi (C_n) $ is a simple closed curve. In this situation, we can consider an Ohtsuka exhaustion $ \left\lbrace \Omega_n\right\rbrace _n $ such that $ \pi (C_n) \subset \partial\Omega_n$. Then, the statement follows from the equivalence between Ohtsuka exhaustions (Lemma \ref{lemma-equivalence-ohtsuka}).
\end{proof}

Finally, the proof of Theorem \ref{thm-prime-ends-MC} follows easily from the previous results.

\begin{proof}[Proof of Theorem \ref{thm-prime-ends-MC}]
First, by Lemma \ref{lemma-impression}, the impression of the prime end is well-defined. Then, the mutually exclusive classification of null-chains into regular, singular and parabolic is given by Lemma \ref{lemma-technical-null-chains}, and the relation with the limit set is clear, since a  point is regular if and only if there exists crosscut neighbourhoods $ \left\lbrace N_n\right\rbrace _n $ in which $ \pi  $ is univalent, and thus the domains $ \left\lbrace \pi(N_n)\right\rbrace _n $ are simply connected (Lemma \ref{lemma-limit-sets3}), and the characterization for parabolic fixed points follows from Lemma \ref{lemma-lifts-of-curves}(a). It is left to prove the statements about impressions and cluster sets.

Before proceeding, we recall that, since $ e^{i\theta} $ is of escaping type, it has a unique associated boundary component $ \Sigma $. Thus, $ I (P(\pi, e^{i\theta}))\subset \Sigma $. Then, in the case where the null-chain is regular, $ \pi $ is univalent on the crosscut neighbourhoods $ \left\lbrace N_n\right\rbrace _n $, and hence the statement follows from the one for Riemann maps (Thm. \ref{thm-prime-ends-sc}). 

Since the cluster sets for singular points have been already characterized (Lemma \ref{lemma-limit-sets3}), it is left to prove that  $ I (P(\pi, e^{i\theta}))=\Sigma $. Since the impression does not depend on the chosen admissible null-chain, let us consider a null-chain obtained by lifting an Ohtsuka exhaustion (this is possible because $ e^{i\theta} $ is singular, otherwise the crosscuts obtained would not shrink). Then, the impression of the prime end can be seen as the intersection of the closures of the corresponding domains $ \overline{D_{s_1s_2\dots s_n}} $, which is clearly $ \Sigma $.

Finally, the statement for parabolic fixed points (which are always in the limit set) follows from the fact that its associated boundary component $ \Sigma $ is an isolated point, together with
Lemmas \ref{lemma-limit-sets} and \ref{lemma-impression}. This completes the proof of Theorem \ref{thm-prime-ends-MC}.
\end{proof}

\subsection*{Properties of prime ends}

According to Lemma \ref{lemma-impression}, we can always assume that prime ends are represented by admissible null-chains obtained by lifting an Ohtsuka exhaustion, together with adding regular null-chains at regular points. Thus,  we say that two prime ends $ P(\pi, e^{i\theta_1})$, $ P(\pi, e^{i\theta_2})$ are {\em equivalent} if either $ e^{i\theta_1} =\gamma(e^{i\theta_2})$, for some $ \gamma\in\Gamma $, or $ P(\pi, e^{i\theta_1})$ and $ P(\pi, e^{i\theta_2})$ can be represented by admissible null-chains $ \left\lbrace C^1_n\right\rbrace _n $ and $ \left\lbrace C^2_n\right\rbrace _n $, such that $ \pi  (C^1_n)=\pi  (C^2_n)$, for all $ n\geq 0 $. By Lemma \ref{lemma-impression}, it is clear that these equivalence classes of prime ends are well-defined; let us denote by $ \mathcal{P}(\pi) $ the equivalence classes of prime ends. The following property, which generalises Carathéodory's theorem for prime ends of simply connected domains, is an immediate consequence of Theorem \ref{thm-prime-ends-MC} (we refer the reader to \cite[Theorem 2.15]{Pom92} for a proof of the simply connected case, which can be easily adapted to deal with regular prime ends in our case).
\begin{cor}\label{cor:prime-end-bijection}
Let $\Omega\subset\Chat$ be a hyperbolic domain, and let $\pi\colon\DD\to\Omega$ be a universal covering with deck transformation group $\Gamma$. Then,
\[ \mathcal{P}(\pi) \simeq (\partial\DD\setminus\Lambda_{NT})/\Gamma, \]
in the sense that there exists a bijection between both sides of the above equation.
\end{cor}

The equivalence classes of prime ends $\mathcal{P}(\pi)$ yield a compactification of $ \Omega $ by defining
\[\widehat{\Omega}=\Omega \cup \mathcal{P}(\pi)  \]
endowed with the following topology. Consider the admissible null-chains given by lifting an Ohtsuka exhaustion, and adding regular null-chains at regular points, and their corresponding rectified crosscut neighbourhoods. By deforming the previous crosscuts in its homotopy class, we may assume that any two different null-chains are eventually disjoint. Then, for any rectified crosscut neighbourhood $ N\subset \mathbb{D} $, let $\widehat{N}\subset \widehat{\Omega}$ be the union of $ \pi(N) $ and the collection  of equivalence classes of all prime ends which are represented by admissible null-chains $ \left\lbrace N_n\right\rbrace _n\subset N $. These neighbourhoods $ \widehat{N} $, together with the open subsets of $ \Omega $, form a basis for $ \widehat{\Omega} $. We refer the reader to \cite[p. 388]{Epstein} for more detail.

Before proceeding, let us discuss what is means for a sequence $ \left\lbrace z_n\right\rbrace _n \subset \Omega$ to converge in this topology. On one hand, if $\{z_n\}_n$ is compactly contained in $\Omega$, then convergence in $\widehat\Omega$ coincides with convergence in the spherical topology of $\Chat$. If, on the other hand, $\{z_n\}_n$ is not compactly contained in $\Omega$, then convergence in $\widehat\Omega$ means that the sequence is converging `to a prime end $P$'. If we take $\{C_n\}_n$ to be an admissible null-chain representing $P$, this means that any of the sets $\pi(N_n)$, where $N_n$ is the rectified crosscut neighbourhood of $C_n$, contains all but finitely many of the points in the sequence $\{z_n\}_n$. In particular, we see that $z_n\to P$ implies that the derived set (i.e., the set of all accumulation points w.r.t. to the spherical metric on $\Chat$) of $\{z_n\}_n$ belongs to the impression of $P$. The converse, however, is not true (even in the simply connected case): different prime ends may have the same impression, and thus a sequence of points may converge to the impression without converging to a particular prime end.

The compactification $ \widehat{\Omega} $ given by the prime ends has the following properties, which justifies our definition of prime ends (compare with the requirements for a `good' definition of prime ends given in \cite{Kaufmann}).
\begin{itemize}
	\item \textit{$ \widehat{\Omega} $ with the previous topology is compact.} Indeed, for any $ \left\lbrace z_n\right\rbrace _n \subset \Omega$ with $ z_n\to\partial \Omega $, we first find a subsequence $\{z_{n_k}\}_k$ such that $z_{n_k}$ accumulates at a unique boundary component $\Sigma$. Then, we lift this subsequence to $\DD$ and extract a further subsequence $\{\tilde z_{k_m}\}_m\subset\DD$, with $\pi(\tilde z_{k_m}) = z_{k_m}$, such that $\tilde z_{k_m}\to e^{i\theta}\in\partial\DD$. Since $z_{k_m}$ accumulates only at $\Sigma$, we can conclude (by appealing to an Ohtsuka exhaustion) that $e^{i\theta}$ is of escaping type, and thus corresponds to a prime end $P(\pi, e^{i\theta})$. The construction of the sequence $\{z_{k_m}\}$ now implies that it converges to $P(\pi, e^{i\theta})$ in $\widehat\Omega$.
	\item \textit{The impression of any prime end does not contain points of the domain}, i.e. $ I (P(\pi, e^{i\theta}))\subset\partial\Omega $. This follows immediately from Lemma \ref{lemma-impression}. Notice that, in particular, impressions are contained in boundary components of $ \Omega $ and, by compactness, every boundary component corresponds to at least one equivalence class of prime ends.
	\item \textit{Compatibility with the universal covering.} The construction of prime ends given in Theorem \ref{thm-prime-ends-MC} is essentially done in the unit disk, and transfered to $ \Omega $ by means of the universal covering. Thus, our construction is automatically compatible with the universal covering, and this compatibility is encoded in Corollary \ref{cor:prime-end-bijection} and Theorem \ref{thm-prime-ends-MC} itself.
	\item \textit{In the simply connected case, the definition corresponds to the one given by Carathéodory.} Observe that in the simply connected case, all the above construction applies, but null-chains are always of regular type. In particular, the usual crosscut neighbourhoods and the rectified ones coincide, and hence our definition of prime ends agrees with the one given by Carathéodory. Of course, in this case, Corollary \ref{cor:prime-end-bijection} degenerates into Carathéodory's theorem.
\end{itemize}	

Notice that, for an isolated component $ \Sigma $ of $ \partial\Omega $, the topology induced in a neighbourhood of  $ \Sigma $ coincides with the one given by Carathéodory's  compactification for the simply connected domain $ \widehat{\mathbb{C}}\smallsetminus \Sigma $.

Finally, we outline a generalisation of the Carath\'eodory--Torhorst Theorem to the most general case of a multiply connected domain; notice that $\pi$ never extends continuously to any point in the limit set, so that the following result is best possible.

\begin{cor}\label{cor:CT-MC}
	Let $\Omega\subset\widehat{\mathbb{C}}$ be a hyperbolic multiply connected domain, and let $\pi\colon\DD\to\Omega$ be a universal covering. Then, $\pi$ extends continuously to $\partial\DD\setminus\Lambda$ if and only if any true crosscut neighbourhood in $\Omega$ has a locally connected boundary.
\end{cor}

Since all points in $\partial\DD\smallsetminus\Lambda$ are regular, Theorem \ref{thm-prime-ends-MC} and Lemma \ref{lemma-limit-sets3} imply that the original proof of the Carathéodory-Torhorst theorem can be used to prove Corollary \ref{cor:CT-MC}. Corollaries \ref{cor:CT-MC} and \ref{cor:prime-end-bijection} together make up  \ref{teo:D}.

\subsection{A plane domain with a fat Cantor limit set}\label{ssec:group-examples}
As noted in Remark \ref{rmk:beardon}, Beardon gave an example in \cite{Bea71} of a Fuchsian group $\Gamma$ whose limit set is a Cantor set of positive Lebesgue measure -- i.e. a {\em fat Cantor set}. However, it is not clear from the construction whether the resulting Riemann surface $\DD\setminus\Gamma$ is of planar character. Here, we use the techniques developed in this paper to provide an easy example of a plane domain whose universal covering has a limit set that is a fat Cantor set.

We start with the annulus $\mathcal{A} := \{z\colon 1/R < |z| < R\}$ for some $R > 1$. Then, we take a sequence $\{z_n\}_n\subset \mathcal{A}$ of points accumulating at every point of the outer boundary component $C_2 = \{z\colon |z| = R\}$ of $\mathcal{A}$, and define 
\[ \Omega := \mathcal{A}\setminus\bigcup_n \{z_n\}. \]
Let $\pi\colon\DD\to\Omega$ be a universal covering with $\pi(0) = 1$, and denote by $\Gamma$ its group of deck transformations. First, since any point of the inner boundary component $C_1 = \{z\colon |z| = 1/R\}\subset\partial\Omega$ admits a true crosscut, we see that  $\Lambda$ is a Cantor set (\ref{cor:E}). In particular, $\Lambda_{NT}$ has measure zero. We are left to show that $ \Lambda $ has positive measure, or, equivalently, that $\partial\DD\setminus\Lambda$ does not have full measure.

To that end, notice that any point in $\partial\DD\setminus\Lambda$ is contained in a pre-image under $\pi^*$ of the inner boundary component $C_1$, and conversely any such pre-image of $C_1$ corresponds to an interval in $\partial \DD\setminus\Lambda$ by Theorem \ref{thm-prime-ends-MC}. Furthermore, by the Hopf-Tsuji-Sullivan theorem, the boundary $\partial\Omega$ supports a harmonic measure $\omega(\cdot; z, \Omega)$, $z\in\Omega$. By Löwner's lemma (see \cite[Section XI.5]{Tsu75}),  $$\lambda (\partial\DD\setminus\Lambda)=\omega(\partial\DD\setminus\Lambda; 0, \DD) = \omega(C_1; 1, \Omega).$$  By the subordination principle \cite[Corol. 21.1.14]{Conway2}, $$\omega(C_1; 1, \Omega) \leq \omega(C_1; 1, \mathcal{A}) = \frac{1}{2},$$ and the proof is complete.


\begin{thebibliography}{KLCN15}

\bibitem[Aba23]{Aba23}
M.~Abate, \emph{Holomorphic dynamics on hyperbolic {R}iemann surfaces}, Studies
  in Mathematics, vol.~89, De Gruyter, 2023.

\bibitem[ABBS13]{prime-ends-metric-spaces}
T.~Adamowicz, A.~Bj\"{o}rn, J.~Bj\"{o}rn, and N.~Shanmugalingam, \emph{Prime
  ends for domains in metric spaces}, Adv. Math. \textbf{238} (2013), 459--505.

\bibitem[Ahl10]{Ahlfors-conformal-invariants}
L.~V. Ahlfors, \emph{Conformal invariants}, AMS Chelsea Publishing, Providence,
  RI, 2010, Topics in geometric function theory.

\bibitem[BC08]{BC08}
A.~F. Beardon and T.~K. Carne, \emph{Euclidean and hyperbolic lengths of images
  of arcs}, Proc. London Math. Soc. \textbf{97} (2008), 183--208.

\bibitem[Bea71]{Bea71}
A.~F. Beardon, \emph{Inequalities for certain {F}uchsian groups}, Acta Math.
  \textbf{127} (1971), 221--258.

\bibitem[Bea83]{Bea83}
\bysame, \emph{The geometry of discrete groups}, Springer, 1983.

\bibitem[BFJK17]{BFJK-Accesses}
K.~Bara\'{n}ski, N.~Fagella, X.~Jarque, and B.~Karpi\'{n}ska, \emph{Accesses to
  infinity from {F}atou components}, Trans. Amer. Math. Soc. \textbf{369}
  (2017), no.~3, 1835--1867.

\bibitem[BJ97]{BJ97}
C.~J. Bishop and P.~W. Jones, \emph{Hausdorff dimension and {K}leinian groups},
  Acta Math. \textbf{179} (1997), 1--39.

\bibitem[BM74]{BM74}
A.~F. Beardon and B.~Maskit, \emph{Limit points of {K}leinian groups and finite
  sided fundamental polyhedra}, Acta Math. \textbf{132} (1974), 1--12.

\bibitem[BM07]{BeardonMinda}
A.~F. Beardon and D.~Minda, \emph{The hyperbolic metric and geometric function
  theory}, Quasiconformal mappings and their applications, Narosa, New Delhi,
  2007, pp.~9--56.

\bibitem[Bus92]{Bus92}
P.~Buser, \emph{Geometry and spectra of compact {R}iemann surfaces}, Progress
  in Mathematics, vol. 106, Birkh{\"a}user, 1992.

\bibitem[CG93]{CarlesonGamelin}
L.~Carleson and T.~W. Gamelin, \emph{Complex dynamics}, Universitext: Tracts in
  Mathematics, Springer-Verlag, New York, 1993.

\bibitem[CL66]{Collingwood-Lohwater}
E.~F. Collingwood and A.~J. Lohwater, \emph{The theory of cluster sets},
  Cambridge Tracts in Mathematics and Mathematical Physics, vol. No. 56,
  Cambridge University Press, Cambridge, 1966.

\bibitem[Con95]{Conway2}
J.~B. Conway, \emph{Functions of one complex variable. {II}}, Graduate Texts in
  Mathematics, vol. 159, Springer-Verlag, New York, 1995.

\bibitem[Eps81]{Epstein}
D.~B.~A. Epstein, \emph{Prime ends}, Proc. London Math. Soc. (3) \textbf{42}
  (1981), no.~3, 385--414.

\bibitem[Fer89]{Fer89}
J.~L. Fern{\'a}ndez, \emph{Limit sets of {F}uchsian and {K}leinian groups},
  Extracta Mathematicae \textbf{4} (1989), 1--20.

\bibitem[FM95]{FM95}
J.~L. Fern{\'a}ndez and M.~V. Meli{\'a}n, \emph{Bounded geodesics of {R}iemann
  surfaces and hyperbolic manifolds}, Trans. Amer. Math. Soc. \textbf{347}
  (1995), 3533--3549.

\bibitem[Hat02]{Hat02}
A.~Hatcher, \emph{Algebraic topology}, Cambridge University Press, 2002.

\bibitem[Hed36]{Hed36}
G.~A. Hedlund, \emph{Fuchsian groups and transitive horocycles}, Duke Math. J.
  \textbf{2} (1936), 530--542.

\bibitem[Hop36]{Hop36}
E.~Hopf, \emph{Fuchsian groups and ergodic theory}, Trans. Amer. Math. Soc.
  \textbf{39} (1936), 299--314.

\bibitem[Hub06]{Hub06}
J.~H. Hubbard, \emph{Teichm{\"u}ller theory with applications to geometry,
  topology and dynamics}, Matrix Editions, 2006.

\bibitem[JF23]{JF23}
A.~Jov{\'e} and N.~Fagella, \emph{Boundary dynamics in unbounded {F}atou
  components}, 2023, available at
  \href{https://arxiv.org/abs/2307.11384}{\texttt{arXiv:2307.11384}}.

\bibitem[Kat92]{Kat92}
S.~Katok, \emph{Fuchsian groups}, The University of Chicago Press, 1992.

\bibitem[Kau30]{Kaufmann}
B.~Kaufmann, \emph{\"{U}ber die {B}erandung ebener und r\"{a}umlicher {G}ebiete
  ({P}rimendentheorie)}, Math. Ann. \textbf{103} (1930), no.~1, 70--144.

\bibitem[KL07]{KeenLakic}
L.~Keen and N.~Lakic, \emph{Hyperbolic geometry from a local viewpoint}, London
  Mathematical Society Student Texts, vol.~68, Cambridge University Press,
  Cambridge, 2007.

\bibitem[KLCN15]{prime-ends}
A.~Koropecki, P.~Le~Calvez, and M.~Nassiri, \emph{Prime ends rotation numbers
  and periodic points}, Duke Math. J. \textbf{164} (2015), no.~3, 403--472.

\bibitem[Mam23]{Mam23}
K.~Mamayusupov, \emph{Accesses to parabolic fixed points from its immediate
  basin}, Mosc. Math. J. \textbf{23} (2023), 271--282.

\bibitem[Mil99]{Milnor}
J.~Milnor, \emph{Dynamics in one complex variable}, Friedr. Vieweg \& Sohn,
  Braunschweig, 1999, Introductory lectures.

\bibitem[MT98]{MT98}
K.~Matsuzaki and M.~Taniguchi, \emph{Hyperbolic manifolds and {K}leinian
  groups}, Oxford University Press, 1998.

\bibitem[Nev70]{Nev70}
R.~Nevanlinna, \emph{Analytic functions}, Springer, 1970.

\bibitem[Nos60]{Noshiro-clustersets}
K.~Noshiro, \emph{Cluster sets}, Ergebnisse der Mathematik und ihrer
  Grenzgebiete, (N.F.), vol. Heft 28, Springer-Verlag,
  Berlin-G\"{o}ttingen-Heidelberg, 1960.

\bibitem[Nä79]{Nakki}
R.~Näkki, \emph{Prime ends and quasiconformal mappings}, J. Analyse Math.
  \textbf{35} (1979), 13--40.

\bibitem[Oht51]{Oht51}
M.~Ohtsuka, \emph{Dirichlet problems on {R}iemann surfaces and conformal
  mappings}, Nagoya Math. J. \textbf{3} (1951), 91--137.

\bibitem[Oht54]{Oht54}
M.~Ohtsuka, \emph{Boundary components of {R}iemann surfaces}, Nagoya
  Mathematical Journal \textbf{7} (1954), 65--83.

\bibitem[Pat76]{Pat76}
S.~Patterson, \emph{The limit set of a {F}uchsian group}, Acta Math.
  \textbf{136} (1976), 241--273.

\bibitem[Pom92]{Pom92}
Ch. Pommerenke, \emph{Boundary behaviour of conformal maps}, Grundlehren der
  mathematischen Wissenschaften, vol. 299, Springer, 1992.

\bibitem[Rem08]{Rempe}
L.~Rempe, \emph{On prime ends and local connectivity}, Bull. Lond. Math. Soc.
  \textbf{40} (2008), no.~5, 817--826.

\bibitem[Sui69]{suita}
N.~Suita, \emph{Carath\'{e}odory's theorem on boundary elements of an arbitrary
  plane region}, Kodai Math. Sem. Rep. \textbf{21} (1969), 413--417.

\bibitem[Sul79]{Sul79}
D.~P. Sullivan, \emph{The density at infinity of a discrete group of hyperbolic
  motions}, Publ. Math. IHES \textbf{50} (1979), 171--202.

\bibitem[Tsu75]{Tsu75}
M.~Tsuji, \emph{Potential theory in modern function theory}, 2nd ed., Chelsea
  Publishing Company, 1975.

\end{thebibliography}

\providecommand{\bysame}{\leavevmode\hbox to3em{\hrulefill}\thinspace}
\providecommand{\MR}{\relax\ifhmode\unskip\space\fi MR }
\providecommand{\MRhref}[2]{%
  \href{http://www.ams.org/mathscinet-getitem?mr=#1}{#2}
}
\providecommand{\href}[2]{#2}

\end{document}